\documentclass[oneside]{article}
\usepackage[a4paper, total={6in, 8in}]{geometry}
\usepackage{bm}
\usepackage{amsmath}
\usepackage{hyperref}
\usepackage{multicol}
\usepackage{amsfonts}
\usepackage{graphicx}
\usepackage{amsthm}
\usepackage{cleveref}
\usepackage{subcaption}
\usepackage{mathrsfs}
\usepackage{authblk}
\usepackage{float}
\usepackage{xcolor}
\usepackage{amssymb} % Provides \Bbbk
\usepackage{amsmath} % Provides \mathbb
\usepackage{xcolor}
\usepackage{bm} % for bold math symbols
\usepackage{newtxtext}
\usepackage{newtxmath}
\usepackage{graphicx}
\usepackage{subcaption}
\usepackage{sidecap} % Allows side captions
\usepackage{appendix}

\usepackage[LGR,T1]{fontenc}

\usepackage{amsmath, amsthm, amssymb, verbatim, comment, tikz, pgfplots, mathtools, mathrsfs, ifthen, import, float, subcaption, enumitem, todonotes, bbm}
\tikzset{>=latex} %FOR SETTING STRAIGHT ARROWHEADS
\usepackage{tikz-network}
\usetikzlibrary{decorations.pathreplacing}
\usepackage{pgfplots}

\usepackage{amsfonts}
\usepackage[ruled,vlined,linesnumbered]{algorithm2e}%for psuedo code\usepackage{algpseudocode}
\usepackage{fullpage}
\usepackage{xifthen}
\usepackage{enumerate}
\usepackage{titlesec}
\newcounter{reviewer}
\newcounter{point}[reviewer]
\newtheorem{theorem}{Theorem}[section]
\newtheorem{remark}[theorem]{Remark}

\newtheorem{lemma}[theorem]{Lemma}

\begin{document}
\title{\bf{A Least-Squares-Based Regularity-Conforming Neural Networks (LS-ReCoNNs) for Solving Parametric Transmission Problems
}}

\author{Shima Baharlouei\textsuperscript{1}\thanks{Corresponding author, email:shima.baharlouei@ehu.eus}\hspace{0.2cm},
~Jamie Taylor\textsuperscript{2},
~David Pardo\textsuperscript{1, 3, 4}\\%

\textsuperscript{1}{\small{University of the Basque Country (UPV/EHU), Leioa, Spain.}}\\%
\textsuperscript{2}{\small{Department of Mathematics, CUNEF Universidad, Madrid, Spain.}}\\%
\textsuperscript{3}{\small{Basque Center for Applied Mathematics (BCAM),  Bilbao, Spain.}}\\%
\textsuperscript{4}{\small{Ikerbasque (Basque Foundation For Sciences), Bilbao, Spain.}}}%
\date{}
\maketitle
\begin{abstract}

This article focuses on solving parametric transmission problems in one and two spatial dimensions. These problems belong to a class of partial differential equations that arise in the modeling of physical systems with heterogeneous materials. They often exhibit discontinuities across interfaces and singularities at points where interfaces intersect. To address these problems, we propose a new deep learning approach named {\it{Least-Squares-Based Regularity-Conforming Neural Network (LS-ReCoNN)}}. This approach proposes a loss function that is shown to be a consistent upper bound for the energy-norm error. The method represents the solution as the sum of a principal component and a singular component. The principal component is decomposed into smooth and gradient-jump parts, which capture both the regular solution behavior and reduced regularity across interfaces in one- and two-dimensional problems. The singular component is introduced to model junction singularities and it is approximated using basis functions computed from a one-dimensional finite element eigenvalue problem. For the principal component, a separated representation is employed, consisting of parameter-dependent coefficients and space-dependent functions. A deep neural network approximates the space-dependent functions, while the parameter-dependent coefficients are determined by a least-squares solver, where the optimal coefficients for each parameter instance are obtained online by solving a low-dimensional least-squares problem. Numerical experiments in one and two dimensions demonstrate that LS-ReCoNN effectively captures singularities while maintaining solution accuracy across a wide range of parameter values.

\end{abstract}
{\bf{Keywords}}: Parametric partial differential equations, Transmission problems, Singularity, Deep learning, Neural network, Least-squares, Least-squares-based neural network (LS-Net), Regularity-conforming neural network (ReCoNN), Physics-informed neural networks (PINN).\\
{\bf{Mathematics Subject Classification}}: 	35A17, 	68T07.

%%%%%%%%%%%%%%%%%%%%%%%%%%%%%%%%%%%%%%%%%%%%%%%%%%%%%%%
%%%%%%%%%%%%%%%%%%%%%%%%%%%%%%%%%%%%%%%%%%%%%%%%%%%%%%%
%%%%%%%%%%%%%%%%%%%%%%%%%%%%%%%%%%%%%%%%%%%%%%%%%%%%%%%
%%%%%%%%%%%%%%%%%%% New Section %%%%%%%%%%%%%%%%%%%%%%%
%%%%%%%%%%%%%%%%%%%%%%%%%%%%%%%%%%%%%%%%%%%%%%%%%%%%%%%
%%%%%%%%%%%%%%%%%%%%%%%%%%%%%%%%%%%%%%%%%%%%%%%%%%%%%%%
%%%%%%%%%%%%%%%%%%%%%%%%%%%%%%%%%%%%%%%%%%%%%%%%%%%%%%%

\section{Introduction}\label{sec_Introduction}

Transmission problems represent a fundamental class of partial differential equations (PDEs) that arise in different areas of knowledge, such as in materials science, fluid dynamics, or electromagnetics. In two-dimensional settings, solutions often exhibit discontinuities across material interfaces and singularities at their junctions. Although these singularities and discontinuities are often well understood analytically, they introduce significant numerical difficulties \cite{grisvard2011elliptic, li2000gradient}. In the finite element setting, the lack of regularity leads to reduced convergence rates \cite{babuvska1970finite}, while in Neural Network (NN) approximations with smooth activation functions, they can produce Gibbs-type instabilities, reducing solution accuracy due to instabilities related to numerical quadrature \cite{taylor2024regularity}.

%%%%%%%%%%%%%%%%%%%%%%%%%%%%%%%%%%%%%%%%%%%%%%%%%%%%%%%
%%%%%%%%%%%%%%%%%% New Paragraph %%%%%%%%%%%%%%%%%%%%%%
%%%%%%%%%%%%%%%%%%%%%%%%%%%%%%%%%%%%%%%%%%%%%%%%%%%%%%%

A wide range of classical methods have been developed to solve transmission problems \cite{heinrich2003nitsche, bonnafont2024finite, brenner2003multigrid}.  Among them, we encounter mesh-adaptive methods guided by a {\it{posteriori}} error estimators \cite{babuvska1978posteriori, carstensen1997adaptive, lin2026adaptive}. These methods are particularly effective for transmission problems, as they efficiently resolve singularities through localized mesh refinement in regions of reduced regularity. Another widely used techniques are based on Domain Decomposition Methods (DDMs) \cite{glowinski1991fourth}, which split the global problem into local subproblems. Each subproblem is solved independently while enforcing transmission conditions across interfaces to ensure continuity and compatibility of the global solution.
Classical methods offer several advantages based on decades of research in the area. However, they can become computationally expensive, an issue that is further amplified in parametric settings where numerous problem instances must be solved.

%%%%%%%%%%%%%%%%%%%%%%%%%%%%%%%%%%%%%%%%%%%%%%%%%%%%%%%
%%%%%%%%%%%%%%%%%% New Paragraph %%%%%%%%%%%%%%%%%%%%%%
%%%%%%%%%%%%%%%%%%%%%%%%%%%%%%%%%%%%%%%%%%%%%%%%%%%%%%%

As a partial remedy for that, researchers have increasingly turned to neural networks (NNs) to solve parametric PDEs in the last decade \cite{bhattacharya2021model, brevis2024learning, dal2020data, geist2021numerical, han2018solving, khara2024neufenet, khoo2021solving, kutyniok2022theoretical, uriarte2022finite}. Indeed, the universal approximation theorem \cite{pinkus1999approximation}, along with subsequent generalizations \cite{chen1995approximation, chen1995universal}, demonstrate that NNs have the remarkable ability to approximate complex operators, including solution operators for parametric PDEs. However, existing methods like Physics-Informed Neural Networks (PINNs)~\cite{raissi2019physics} are unsuitable for solving transmission problems, since these problems exhibit low-regularity solutions; other methods like Variational PINNs (VPINNs) \cite{kharazmi2019variational} suffer from critical integration problems \cite{taylor2025stochastic} that are enhanced in presence of the Gibbs phenomena that appear when trying to approximate a discontinuous derivative with a continuous NN function.

%%%%%%%%%%%%%%%%%%%%%%%%%%%%%%%%%%%%%%%%%%%%%%%%%%%%%%%
%%%%%%%%%%%%%%%%%% New Paragraph %%%%%%%%%%%%%%%%%%%%%%
%%%%%%%%%%%%%%%%%%%%%%%%%%%%%%%%%%%%%%%%%%%%%%%%%%%%%%%

In \cite{taylor2024regularity}, we proposed a method to deal with low-regularity solutions using NNs. The key idea was to decompose the solution into a principal part and a singular part,
\begin{equation}
     \mathfrak{u}(x):= u(x) + u_s(x)
\end{equation}
where $\mathfrak{u}(x)$ is the solution, $u(x)$ denotes the principal component, and $u_s(x)$ captures the singular behavior. This decomposition relies on the assumption that the locations of discontinuities and singularities in the domain are known in advance. By embedding this information directly into the architecture, the NN automatically follows the expected regularity of the PDE solution. This architecture design, named Regularity-Conforming Neural Network (ReCoNN) \cite{taylor2024regularity}, leads to more stable training, faster convergence, and avoids spurious oscillations (such as Gibbs-type artefacts) near singularities.

%%%%%%%%%%%%%%%%%%%%%%%%%%%%%%%%%%%%%%%%%%%%%%%%%%%%%%%
%%%%%%%%%%%%%%%%%% New Paragraph %%%%%%%%%%%%%%%%%%%%%%
%%%%%%%%%%%%%%%%%%%%%%%%%%%%%%%%%%%%%%%%%%%%%%%%%%%%%%%

While ReCoNNs effectively handle singularities in transmission problems, parametric problems present an additional challenge: the need to efficiently approximate solutions across the parameter space. In \cite{baharlouei2024least}, we developed a method based on a separated representation of the solution $\mathfrak{u}(x ; p)$ in the form
\begin{align}
    \mathfrak{u}(x ; p) := \sum _{i = 1}^N c_i (p) u_i(x)
    , \qquad \forall p \in \mathbb{P},
    \label{eq2_PGD_decomposition}
\end{align} 
where $p$ is a parameter in the parameter space $\mathbb{P}$. The functions $u_i(x)$ were approximated by a vector-valued NN with $N$ outputs, while the coefficients $c_i(p)$ were determined by a least-squares (LS) solver for each parameter $p$. The resulting method is named the Least-Squares-Based Neural Network (LS-Net) method \cite{baharlouei2024least}. We demonstrated in \cite{baharlouei2024least} that under appropriate technical hypotheses on the parametric problem, a sufficiently large NN  combined with an LS solver, can approximate parametric solutions with any desired accuracy. Additionally, the LS-Net method maintains discretization invariance, meaning that coarse numerical integration can be used during training to reduce computational cost, while finer integration can be employed in post-processing when higher accuracy is required. This method shares similarities with Reduced-Order Modeling (ROM) \cite{azeez2001proper, rozza2008reduced, chinesta2011short}. Both approaches aim to identify low-dimensional subspaces --- span $\{u_i\}_{i=1}^N$ --- that effectively approximate the high-dimensional solution manifold of the underlying problem. Traditional ROM techniques construct this subspace using data-driven or projection-based strategies that involve solving many parameter instances. In contrast, the LS-Net method uses an NN to find these basis functions directly via consistent optimality criteria.

%%%%%%%%%%%%%%%%%%%%%%%%%%%%%%%%%%%%%%%%%%%%%%%%%%%%%%%
%%%%%%%%%%%%%%%%%% New Paragraph %%%%%%%%%%%%%%%%%%%%%%
%%%%%%%%%%%%%%%%%%%%%%%%%%%%%%%%%%%%%%%%%%%%%%%%%%%%%%%

In this work, we combine the two aforementioned strategies: ReCoNNs \cite{taylor2024regularity} to deal with singularities, and LS-Net \cite{baharlouei2024least} to solve parametric transmission problems. Furthermore, we replace the singular part of the solution in ReCoNNs by an FE eigenvalue solver to enhance accuracy. The proposed framework is built on three essential building blocks, each designed to address a distinct challenge:
\begin{itemize}
    \item \textbf{Designing the loss function for error control:} We design a quadratic loss function that enforces both the PDE residual and the flux continuity, while simultaneously providing a consistent upper bound for the energy-norm error of the approximate solution. Furthermore, the loss is formulated such that singularities do not need to be numerically integrated, which enhances numerical stability. Consequently,  minimizing this loss over the network parameters ensures that the approximate solution is guaranteed to converge toward the exact solution, with dependence on the singular component errors.
    
    \item 
    \textbf{Constructing the LS-ReCoNN solution:} The LS-ReCoNN solution is built from two complementary components: (a) a principal component that contains the smooth and gradient jump parts and is represented by a deep NN; and (b) a singular component, which captures localized features near singularities and is obtained using an FE eigenvalue solver applied to the associated 1D Sturm–Liouville problem. This separation leads to the following representation:
    \begin{align}
    \mathfrak{u}(x ; p) := \sum _{i = 1}^N c_i(p) {u}_i(x) + \sum _{j = 1}^M d_j(p) u^s_j(x, p),
    \qquad \forall p \in \mathbb{P}
    \label{eq2_LSReCoNN_decomposition}
    \end{align}
    which is a modified form of equation \eqref{eq2_PGD_decomposition}. 
    Herein, the functions $u^s_j(x, p)$ are approximated by $M$ eigenvectors, and the coefficients $c_i(p)$ and $d_j(p)$ are determined by an LS solver for each parameter $p$. Compared to the ReCoNN approach, where singular functions are trained in tandem with the rest of the NN, herein, we compute them directly with an FE eigenvalue solver. By relying on a one-dimensional FE eigenvalue problem, which is extremely fast to solve, we replace a costly training procedure with a simple and more accurate linear computation.

    \item \textbf{Handling the parametric dimension:} The LS-ReCoNN incorporates the LS-Net strategy to address the parametric dimension of the problem effectively. In this framework, training the NN focuses on learning and optimizing the space functions that, together with the eigenvalue components, form a low-dimensional representation of the parametric solution manifold. For each parameter instance, an LS solver computes the optimal coefficients within the learned reduced space. This enables accurate and efficient approximations across a broad range of parameters.

\end{itemize}

\noindent During training and inference, the computational workload of LS-ReCoNN is divided into four main tasks: (a) evaluating the NN and its gradients, (b) solving the 1D FE eigenvalue problems, (c) {constructing the LS systems, which can be decomposed into (c1) assembling parameter-independent components (the dominant cost in LS construction) and (c2) constructing a lightweight parameter-dependent combination} and (d) solving the LS systems. Crucially, the costs of (a) and {(c1)} are independent of the number of parameters because the NN is evaluated at fixed spatial points and the LS matrices are assembled from precomputed, parameter-independent components. The remaining parameter-dependent construction in (c2) introduces only a weak dependence on the number of parameters and remains inexpensive for moderately sized systems. While the costs of (b) and (d) scale linearly with the number of parameters, they involve simple 1D computations and small linear systems that are numerically inexpensive. As a result, for a moderate number of parameters, the total cost is dominated by the {NN evaluations and the parameter-independent LS assembly (c1),} making the overall execution time almost independent of the number of parameter instances.

%%%%%%%%%%%%%%%%%%%%%%%%%%%%%%%%%%%%%%%%%%%%%%%%%%%%%%%
%%%%%%%%%%%%%%%%%% New Paragraph %%%%%%%%%%%%%%%%%%%%%%
%%%%%%%%%%%%%%%%%%%%%%%%%%%%%%%%%%%%%%%%%%%%%%%%%%%%%%%

In addition, we observe in the numerical examples that training of the parametric problem often leads to better convergence behavior and superior quality solutions than the non-parametric problem. We believe this is due to the decrease in the number of local minima in the parametric problem. As a result, we observe surprisingly good results for the parametric problem as compared to the non-parametric one at a similar computational cost. This is a unique advantage of NNs in contrast to traditional FEMs.

%%%%%%%%%%%%%%%%%%%%%%%%%%%%%%%%%%%%%%%%%%%%%%%%%%%%%%%
%%%%%%%%%%%%%%%%%% New Paragraph %%%%%%%%%%%%%%%%%%%%%%
%%%%%%%%%%%%%%%%%%%%%%%%%%%%%%%%%%%%%%%%%%%%%%%%%%%%%%%

On the other hand, the LS-ReCoNN method also faces some limitations. It inherits the restrictions of both LS-net and ReCoNN approaches: the former is inherently tailored to linear parametric problems, while the latter relies on precise prior knowledge of low-regularity structures in the solution. As a result, the direct extension of LS-ReCoNN to more general problem classes becomes more challenging and may require additional modeling assumptions or problem-specific adaptations. Also, a direct extension of the proposed method to the 3D case would require the efficient solution of 2D eigenvalue problems.

%%%%%%%%%%%%%%%%%%%%%%%%%%%%%%%%%%%%%%%%%%%%%%%%%%%%%%%
%%%%%%%%%%%%%%%%%% New Paragraph %%%%%%%%%%%%%%%%%%%%%%
%%%%%%%%%%%%%%%%%%%%%%%%%%%%%%%%%%%%%%%%%%%%%%%%%%%%%%%

The remainder of the paper is structured as follows.
Section \ref{Model Problem} defines the parametric elliptic transmission problem and describes the decomposition of the solution. This section also introduces the continuous loss function used for error control. 
Section \ref{LS-ReCoNN Framework} details the LS-ReCoNN framework, covering the NN architecture and the approximation of singular components via an FE eigenvalue solver. Additionally, it presents the derivation of the LS-ReCoNN loss function and its discretization through Monte Carlo sampling. 
Section \ref{Analysis of Computational Cost} provides a quantitative analysis of the framework's computational efficiency. 
Then, Section \ref{Sec_Numerical Examples} introduces three numerical experiments in one and two dimensions, designed to illustrate the advantages and limitations of the proposed LS-ReCoNN method. 
Section \ref{Conclusion} provides concluding remarks and outlines potential directions for future research. Finally, Appendices \ref{Proof of H-Control by Residuals}, \ref{Implementation Details of the Trial Spaces}, \ref{Strong Formulation of the Decomposed Solution}, \ref{Proof of the Error Bound}, and \ref{Loss Discretization Details} present theoretical results, including supporting lemmas and detailed analysis that complement the main developments of this work.

%%%%%%%%%%%%%%%%%%%%%%%%%%%%%%%%%%%%%%%%%%%%%
%%%%%%%%%%%%%%%%%%%%%%%%%%%%%%%%%%%%%%%%%%%%%
%%%%%%%%%%%%%%%%%%%%%%%%%%%%%%%%%%%%%%%%%%%%%
%%%%%%%%%%%%%%%%%%%%%%%%%%%%%%%%%%%%%%%%%%%%%

\section{Model Problem: Parametric Elliptic Transmission Problem}\label{Model Problem}

\subsection{Formulation of the Problem}\label{Formulation of the problem}

\noindent We consider a transmission problem. In a bounded, Lipschitz domain $\Omega\subset\mathbb{R}^d$, we introduce the space 
\begin{equation}
    H^1_0(\Omega)=\{u\in L^2(\Omega):\nabla u\in L^2(\Omega)^d,\, u=0 \text{ on }\partial \Omega\}.
\end{equation}
Given positive $p\in L^\infty(\Omega)$, bounded away from zero, and $\ell^p\in L^2(\Omega)$, we consider the following PDE in weak form: To find $\mathfrak{u}^p\in H^1_0(\Omega)$ such that 
\begin{equation}\label{eqWeakForm}
\int_\Omega p\nabla \mathfrak{u}^p\cdot\nabla v\,dx = \int_\Omega \ell^p v\,dx,
\end{equation}
for all $v\in H^1_0(\Omega)$. The assumptions on $p$ guarantee the existence of a unique solution via the Lax-Milgram theorem~\cite{baharlouei2024least, babuvska1971error, chen2014inf}. The flux is defined to be $p\nabla \mathfrak{u}^p$, and the weak formulation \eqref{eqWeakForm} imposes that its distributional divergence satisfies $\text{div}(p\nabla \mathfrak{u}^p)=\ell^p\in L^2(\Omega)$, so that $p \nabla \mathfrak{u}^p \in H(\text{div},\Omega)=\{v\in L^2(\Omega)^d : \text{div}(v)\in L^2(\Omega)\}$. 

\noindent We consider the case where $p$ is piecewise constant over a polygonal partition of $\Omega$. This partition consists of a finite number of non-overlapping subdomains $\{\Omega_i\}_{i=1}^I$, such that $p(x) = p_i$ for each $\Omega_i$, where $p_i \in (0, \infty)$. 
In the 1D case, $\Omega$ is an interval divided into smaller sub-intervals. In two dimensions, for simplicity, we take $\Omega$ to be a rectangular domain divided into smaller rectangular subdomains aligned with the coordinate axes. In 2D, discontinuities in $p$ (i.e., material discontinuities) can induce singularities at vertices where three or more materials intersect. The set of interior interfaces, i.e., the common boundaries between adjacent subdomains $\Omega_i$ and $\Omega_j$ with $i \ne j$, is denoted by
% \begin{equation}
% \Gamma := \left\{ \gamma_{ij} = \partial \Omega_i \cap \partial \Omega_j \,;\, i \ne j \right\},
% \end{equation}
\begin{equation}
\Gamma := \bigcup_{i \neq j} \gamma_{ij} = \partial\Omega_i \cap \partial\Omega_j,
\end{equation}
where \( \partial \Omega_i \) is the boundary of subdomain \( \Omega_i \).
Note that for each pair \( i \ne j \), the interface satisfies \( \gamma_{ij} = \gamma_{ji} \). Figure~\ref{fig_geometry_1D_2D} provides a schematic illustration of the domain decomposition and associated notation for the one- and two-dimensional settings. 
\begin{figure}[htbp]
    \centering
    \hspace{0.5cm}
    \begin{subfigure}[t]{0.48\textwidth}
        \centering
        \raisebox{1cm}{
            \scalebox{0.85}{\begin{tikzpicture}[
    scale=1.5,
    interface/.style={thick, dashed},
    line/.style={thick}
]

% Main line for the domain
\draw[line] (-3, 0) -- (3, 0);

% Subdomain boundaries (interfaces)
\foreach \x in {-3, -2, -1, 0, 2, 3} {
    \draw[interface] (\x, -0.2) -- (\x, 0.2); % Dashed vertical lines
}

% Domain boundary labels
\node[below] at (-3, -0.3) {\(a \)};
\node[below] at (-3, 0.12) {$\bullet$};
\node[below] at (3, -0.3) {\(b \)};
\node[below] at (3, 0.12) {$\bullet$};

% Interface labels
\foreach \x/\label in {-2/\gamma_{12}, -1/\gamma_{23}, 0/\gamma_{34}, 2/\gamma_{(I-1) I}} {
    \node[below] at (\x, -0.3) {\(\label\)};

\node[below] at (-2, 0.12) {$\bullet$};
\node[below] at (-1, 0.12) {$\bullet$};
\node[below] at (0, 0.12) {$\bullet$};
\node[below] at (2, 0.12) {$\bullet$};

}

% Subdomain labels
\node[above] at (-2.5, -0.4) {\(\Omega_1\)};
\node[above] at (-1.5, -0.4) {\(\Omega_2\)};
 \node[above] at (-0.5, -0.4) {\(\Omega_3\)};
\node[above] at (2.5, -0.4) {\(\Omega_I\)};

\end{tikzpicture}}
        }
        \caption{One-dimensional domain.}
        \label{fig_geometry_1D}
    \end{subfigure}
    \hspace{-0.3cm}
    \begin{subfigure}[t]{0.48\textwidth}
        \centering
        \scalebox{0.85}{\definecolor{My_blue}{RGB}{0,0,102}
\definecolor{My_Orange}{RGB}{255, 178, 102}

\begin{tikzpicture}[
punkt/.style={
   rectangle,
   rounded corners,
   draw=black, thick,
   text width=2em,
   minimum height=1em,
   text centered}
]
    % Draw the square domain Omega
    \filldraw[fill=white, thick] (-2,-2) rectangle (0,0);

    \filldraw[fill=white, thick] (-2,0) rectangle (0,2);

    \filldraw[fill=white, thick] (0,0) rectangle (2,2);

    \filldraw[fill=white, thick] (0, -2) rectangle (2, 0);

    \filldraw[fill=white, thick] (2, -2) rectangle (4, 0);
    \filldraw[fill=white, thick] (2, 0) rectangle (4, 2);
\filldraw (0, 0) circle (2pt) node[anchor=north east] {$x_1$};
\filldraw (2, 0) circle (2pt) node[anchor=north east] {$x_2$};

\node at (2.8, -1) [right] {$\Omega_6$};
\node at (0.8, -1) [right] {$\Omega_5$};
\node at (-1.3, -1) [right] {$\Omega_4$};
\node at (2.8, 1) [right] {$\Omega_3$};
\node at (0.8, 1) [right] {$\Omega_2$};
\node at (-1.3, 1) [right] {$\Omega_1$};
% \node at (-0.3, -3) [right] {$\Omega$};

\node at (0, 1) [right] {$\gamma_{12}$};
\node at (0, -1) [right] {$\gamma_{45}$};
\node at (2, 1) [right] {$\gamma_{23}$};
\node at (2, -1) [right] {$\gamma_{56}$};
\node at (2.8, 0.2) [right] {$\gamma_{36}$};
\node at (0.8, 0.2) [right] {$\gamma_{25}$};
\node at (-1.3, 0.2) [right] {$\gamma_{14}$};

\end{tikzpicture}}
        \caption{Two-dimensional domain.}
        \label{fig:geometry_2D}
    \end{subfigure}
    \caption{Illustration of the domain, subdomains, and relative notations in one and two dimensions.}
    \label{fig_geometry_1D_2D}
\end{figure}
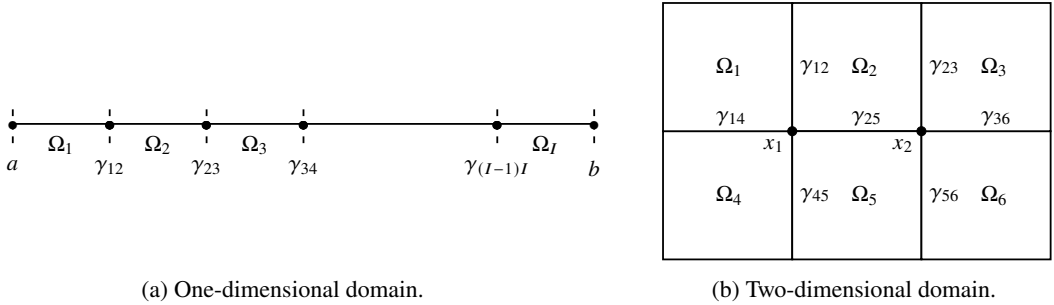

\noindent The discontinuities of $p$ at the set of interfaces $\Gamma$ generally limit the regularity of the solution $u^p$. However, standard elliptic regularity results \cite{evans2022partial} %\footnote{Evans, Partial Differential Equations, Section 6.2}
guarantee that $\mathfrak{u}^p$ is locally in $H^2$ away from the interfaces. In particular, $\mathfrak{u}^p$ satisfies the strong form equation 
\begin{equation}
    -\text{div}(p(x)\nabla \mathfrak{u}^p(x)) = \ell^p(x), 
\end{equation}
away from the interfaces. Furthermore, as the flux $p \nabla \mathfrak{u}^p$ is in $H(\text{div};\Omega)$, its normal component is continuous across interfaces \cite{boffi2013mixed}.
When $\mathfrak{u}^p$ is sufficiently regular on either side of the interface and $\mathbf{n}^{\pm}$ denotes the outward/inward unit normal vector to an interface $\gamma\in\Gamma$ with $\mathbf{n}^+=-\mathbf{n}^-$, the continuity of the flux across the interface may be expressed as 
\begin{equation}
    \left[p\nabla \mathfrak{u}^p\cdot \mathbf{n}\right]:=\lim\limits_{x\to \gamma^+}(p(x)\nabla \mathfrak{u}^p(x)\cdot \mathbf{n}^+(x))+\lim\limits_{x\to\gamma^-}(p(x)\nabla \mathfrak{u}^p(x)\cdot \mathbf{n}^-(x))=0.
\end{equation}
In general, the traces of the flux $p\nabla u\cdot \mathbf{n}$ must be interpreted as elements of $H^{-\frac{1}{2}}(\Gamma)$, however, their jump across the interface may trivially be understood as an element of $L^2$ at the exact solution $\mathfrak{u}^p$. Accordingly, we interpret the system
\begin{equation}\label{eq2_strong}
\begin{array}{r c l l}
-\text{div}(p\nabla \mathfrak{u}^p) &= &\ell ^p,&  \text{in }\Omega\setminus\Gamma,\\
\left[p\nabla \mathfrak{u}^p\cdot \mathbf{n}\right]&=& 0,& \text{in }\Gamma,
\end{array}
\end{equation}
as the strong form of \eqref{eqWeakForm}. It follows that the exact solution to the transmission problem belongs to the following parameter-dependent function space 
\begin{equation}\label{eq2_X_p space}
    \mathbb{X}^p = \{u\in H^1_0(\Omega): (p\nabla u)|_{\Omega_i}\in H(\text{div},\Omega_i), [p\nabla u\cdot \mathbf{n}]\in L^2(\Gamma)\}. 
\end{equation}
Note that in the definition of $\mathbb{X}^p$, we weaken the requirement on the flux $p \nabla u$ to the broken $H(\text{div})$ space. This is equivalent to requiring that $\nabla u$ itself belongs to the broken $H(\text{div})$ space, a condition that is easily built into the architecture in a parameter-independent manner. Similarly, in our discretization of $\mathbb{X}^p$, we will employ trial functions where both one-sided traces are in $L^2(\Gamma)$, removing parameter dependence of the condition that $[p\nabla u\cdot \mathbf{n}]\in L^2(\Gamma)$.
This space is a Hilbert space when endowed with the graph norm
\begin{equation}\label{eq2_graph norm}
\|u\|_g:=\left(\|u\|_{H^1_0}^2+\sum\limits_{i}\| \operatorname{div}(p\nabla u)\|_{L^2(\Omega_i)}^2 + \|[p\nabla u\cdot \mathbf{n}]\|_{L^2(\Gamma)}^2\right)^\frac{1}{2}. 
\end{equation}
\begin{theorem}\label{2_Theorem_1}
Let $u\in \mathbb{X}^p$ and $\Theta>0$. Then, there exists a constant $C>0$, depending only on $\Theta$ and the set of discontinuities $\Gamma$ of $p$, such that 
\begin{equation}
    \|u\|_{H^1_0}\leq \frac{C}{\min(p)}\sqrt{\|\text{div}(p\nabla u)\|_{L^2(\Omega)}^2 + \Theta\|[p\nabla u \cdot n]\|_{L^2(\Omega)}^2}
\end{equation}
\end{theorem}
% \begin{theorem}\label{2_Theorem_1}
% For any $u \in \mathbb{X}_p$, there exists a constant $C > 0$, \blue{independent of $u$ and $p$} and $\Omega$, such that
% \begin{equation}
%     \|u\|_{H^1_0(\Omega)} \leq C \left( \sum_{i} \|  \operatorname{div}(p\nabla u)\|_{L^2(\Omega_i)}^2 + \Theta \|[p\nabla u \cdot \mathbf{n}]\|_{L^2(\Gamma)} \right),
% \end{equation}
% where $\Theta > 0$ is a fixed positive scaling parameter.
% \end{theorem}
\noindent The proof of Theorem \ref{2_Theorem_1} is provided in Appendix \ref{Proof of H-Control by Residuals}. A significant consequence of this result is that $\mathbb{X}^p$ can be equipped with the equivalent norm 
\begin{equation}\label{eq2_Xp_norm}
    \|u\|_{\mathbb{X}^p} = \sqrt{ \|\text{div}(p\nabla u)\|_{L^2(\Omega)}^2 + \Theta\|[p\nabla u\cdot \mathbf{n}]\|_{L^2(\Gamma)}^2},
\end{equation}
% \begin{equation}\label{eq2_Xp_norm}
%     \|u\|_{\mathbb{X}^p}^2 = \sum\limits_{i}\|\text{div}(p\nabla u)\|_{L^2(\Omega_i)}^2 + \Theta\|[p\nabla u\cdot \mathbf{n}]\|_{L^2(\Gamma)}^2,
% \end{equation}
which satisfies $\|u\|_{H^1_0}\leq \frac{C}{\min(p)} \|u\|_{\mathbb{X}^p}$. This norm effectively measures the residuals of the strong form \eqref{eq2_strong}. 
Specifically, we derive the following error estimate for the transmission problem
\begin{equation}\label{eq2_upper bound}
\begin{split}
\|u-\mathfrak{u}^p\|_{H^1_0}^2\lesssim \|u-\mathfrak{u}^p\|_{\mathbb{X}^p}^2 =& \|\text{div}(p\nabla u)-\text{div}(p\nabla \mathfrak{u}^p)\|_{L^2(\Omega)}^2 + \Theta \|[p\nabla u\cdot \mathbf{n}]-[p\nabla \mathfrak{u}^p\cdot \mathbf{n}]\|_{L^2(\Gamma)}^2\\
= & \|\text{div}(p\nabla u)-\ell ^p\|_{L^2(\Omega)}^2 + \Theta\|[p\nabla u\cdot \mathbf{n}]\|_{L^2(\Gamma)}^2
\end{split}
\end{equation}
The inequality \eqref{eq2_upper bound} establishes that the norm $\|u-\mathfrak{u}^p\|_{\mathbb{X}^p}^2$ serves as an upper bound for the $H^1_0$-error. In particular, the evaluation of this norm requires only knowledge of the differential operator and the given data. Consequently, $\|u-\mathfrak{u}^p\|_{\mathbb{X}^p}^2$ can be utilized as a loss function to identify solutions to the underlying PDE.

%%%%%%%%%%%%%%%%%%%%%%%%%%%%%%%%%%%%%%%%%%%%%%%%%%%%
%%%%%%%%%%%%%%%%%%%%%%%%%%%%%%%%%%%%%%%%%%%%%%%%%%%%
%%%%%%%%%%%%%%%%%%%%%%%%%%%%%%%%%%%%%%%%%%%%%%%%%%%%
%%%%%%%%%%%%%%%%%%%%%%%%%%%%%%%%%%%%%%%%%%%%%%%%%%%%
%%%%%%%%%%%%%%%%%%%%%%%%%%%%%%%%%%%%%%%%%%%%%%%%%%%%
%%%%%%%%%%%%%%%%%%%%%%%%%%%%%%%%%%%%%%%%%%%%%%%%%%%%

\subsection{Decomposition of the Solution}\label{Decomposition of the solution}
% We represent the solution to \eqref{eq2_strong} as the sum of two components: a singular component and a principal component. The principal component consists of a smooth part and a gradient jump part. This superposition allows us to capture different physical aspects of the solution independently by sequentially isolating the regularity-limiting behaviors.
% Specifically, for a given parameter $p$, the solution $\mathfrak{u}^p$ admits the representation
% \begin{equation}\label{eq2_exact solution}
% \mathfrak{u}^p= \underbrace{\underbrace{\mathfrak{w}^p}_{\substack{\text{Smooth}\\ \text{part}}} +  \underbrace{\mathfrak{v}^p}_{\substack{\text{Gradient jump} \\ \text{ part}}} }_{\text{Principal component}}+ \underbrace{\mathfrak{s}^p}_{\substack{\text{Singular}\\ \text{ component}}}.    
% \end{equation}
% The spatial support of each component is visualized in Figure \ref{fig_solution_decomposed_1D_2D}. 

In the following, we decompose the exact solution \eqref{eq2_strong} into two components: (a) the principal part, which contains the smooth part of the solution inside the domain and jumps of the gradients at the interfaces, and (b) the singular part, which is associated to singular vertices that appear in two-dimensional transmission problems. We define these as the principal part and the singular part, respectively. The objective of this decomposition is to characterize the underlying structure of the solution, enabling us to encode these analytical properties directly into the NN architecture. This targeted approach significantly enhances the model's capabilities, particularly in regions with low regularity. We formalize these results in the following theorem.

\begin{theorem}\label{thm_decomposition}
Let $\mathfrak{u}^p$ be the solution to the transmission problem \eqref{eq2_strong}.
\begin{itemize}
    \item \textbf{One-dimensional problems:} We may decompose the solution as follows
\begin{equation}\label{eq2_exact solution_1D}
\mathfrak{u}^p= \underbrace{\underbrace{\mathfrak{w}^p}_{\substack{\text{Smooth}\\ \text{part}}} +  \underbrace{\mathfrak{v}^p}_{\substack{\text{Gradient jump} \\ \text{ part}}} }_{\text{Principal component}},    
\end{equation}
where $\mathfrak{w}^p\in H^2(0,1)\cap H^1_0(0,1)$ and $\mathfrak{v}^p\in H^2(0,1\setminus \Gamma)\cap H^1_0(0,1)$ is harmonic in $\Omega \setminus\Gamma$. Furthermore, $\mathfrak{w}^p$ and $\mathfrak{v}^p$ satisfy the following coupled interface condition on $\Gamma$ 
\begin{equation}
    [p\frac{d}{dx}\mathfrak{v}^p\cdot\mathbf{n}]=-[p]\frac{d}{dx}\mathfrak{w}^p\cdot\mathbf{n}.
\end{equation}
\item \textbf{Two-dimensional problems:} There exists a finite-dimensional space $\mathbb{U}_3^p\subset H^1_0(\Omega)\setminus H^2(\Omega\setminus\Gamma)$ depending only on $p$, such that $\text{div}(p\nabla \mathfrak{s})\in L^2(\Omega)$ for all $s\in \mathbb{U}^p_3$. Then, the solution admits the following decomposition
\begin{equation}\label{eq2_exact solution_2D}
\mathfrak{u}^p= \underbrace{\underbrace{\mathfrak{w}^p}_{\substack{\text{Smooth}\\ \text{part}}} +  \underbrace{\mathfrak{v}^p}_{\substack{\text{Gradient jump} \\ \text{ part}}} }_{\text{Principal component}}+ \underbrace{\mathfrak{s}^p}_{\substack{\text{Singular}\\ \text{ component}}},    
\end{equation}
where $\mathfrak{w}^p\in H^2(\Omega)\cap H^1_0(\Omega)$, $\mathfrak{v}^p\in H^2(\Omega\setminus \Gamma)\cap H^1_0(\Omega)$ is harmonic in $\Omega \setminus\Gamma$, and $\mathfrak{s}^p\in\mathbb{U}^p_3$. Furthermore, $\mathfrak{w}^p$ and $\mathfrak{v}^p$ satisfy the following coupled interface condition on $\Gamma$
\begin{equation}
   [p\nabla \mathfrak{v}^p\cdot \mathbf{n}]=-[p]\mathfrak{w}^p\cdot \mathbf{n}. 
\end{equation}
\end{itemize}

\end{theorem}

\noindent The proof of Theorem \ref{thm_decomposition} is provided in Appendix \ref{Solution Decomposition Proof}.  
The physical interpretation of this decomposition is as follows:
\begin{itemize}
\item $\mathfrak{w}^p$ represents the bulk behavior driven by the source term.
\item $\mathfrak{v}^p$ is harmonic away from the interfaces and satisfies a non-homogeneous Neumann-type condition on $\Gamma$. This component captures the discontinuities in the gradient across interfaces.
\item $\mathfrak{s}^p$ captures point singularities that impede the global $H^2(\Omega\setminus\Gamma)$ regularity of the exact solution.
\end{itemize}

\noindent The active regions of the solution decomposition component are visualized in Figure \ref{fig_solution_decomposed_1D_2D}.

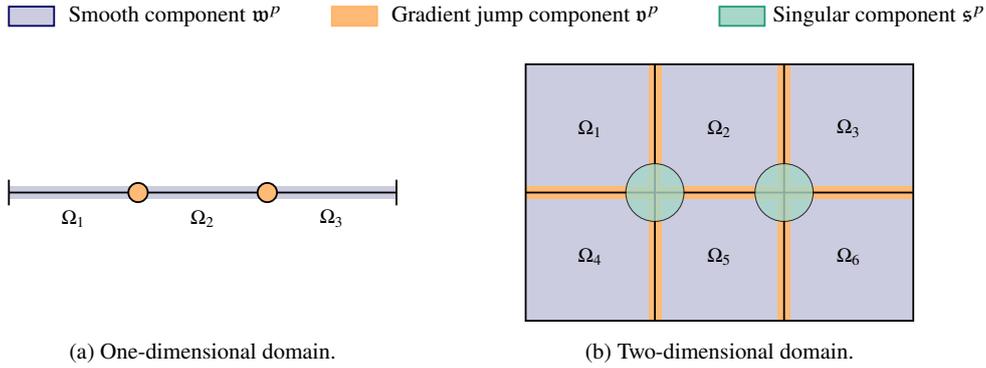
\begin{figure}[htbp]
    \centering
        \centering
        \raisebox{1cm}{
            \scalebox{0.85}{% \definecolor{My_blue}{RGB}{0,0,102}
% \definecolor{My_orange}{RGB}{255, 178, 102}

% \begin{tikzpicture}[scale=1.]
    
%     % --- 2. Subdomains (Smooth Component w^p) ---
%     % Shading the interior to show w^p is active throughout
%     \fill[My_blue!20] (-2,-0.1) rectangle (4,0.1);
    
%     % --- 3. The 1D Domain Line ---
%     \draw[thick] (-2,0) -- (4,0);

%     % --- 4. Interface Points (Jump Component v^p) ---
%     % In 1D, interfaces are points. We circle them as requested.
%     \draw[black, fill=My_orange!50, thick] (0,0) circle (0.15);
%     \draw[black, fill=My_orange!50, thick] (2,0) circle (0.15);
    
%     % Boundary markers
%     \draw[thick] (-2,-0.2) -- (-2,0.2);
%     \draw[thick] (4,-0.2) -- (4,0.2);

%     % --- 5. Labels ---
%     \node at (-1, -0.3) {\small $\Omega_1$};
%     \node at (1, -0.3) {\small $\Omega_2$};
%     \node at (3, -0.3) {\small $\Omega_3$};
    
%     \node[anchor=north] at (0,-0.2) {$x_1$};
%     \node[anchor=north] at (2,-0.2) {$x_2$};

%     % --- 6. Legend ---
%     \begin{scope}[shift={(0,-2)}]
%         \fill[My_blue!20, draw=My_blue, thick] (-1,0) rectangle (-0.3,0.3) node[right, black, shift={(0.2,-0.15)}] { Active range for the smooth component $\mathfrak{w}^p$};
%         \fill[My_orange!40, draw=My_orange, thick] (-1,-0.5) rectangle (-0.3,-0.2) node[right, black, shift={(0.2,-0.15)}] { Active range for the smooth component $\mathfrak{v}^p$};
%     \end{scope}

% \end{tikzpicture}

\definecolor{My_blue}{RGB}{0,0,102}
\definecolor{My_orange}{RGB}{255, 178, 102}
\definecolor{My_green}{RGB}{33, 155, 118}

\begin{tikzpicture}[scale=1.]

    % =========================================================
    % LEFT SIDE: 1D DOMAIN
    % =========================================================
    \begin{scope}[shift={(-6,0)}]
        % Subdomains (Smooth Component w^p)
        \fill[My_blue!20] (-2,-0.1) rectangle (4,0.1);
        
        % The 1D Domain Line
        \draw[thick] (-2,0) -- (4,0);

        % Interface Points (Jump Component v^p) - Point-wise in 1D
        \draw[black, fill=My_orange!90, thick] (0,0) circle (0.15);
        \draw[black, fill=My_orange!90, thick] (2,0) circle (0.15);
        
        % Boundary markers
        \draw[thick] (-2,-0.2) -- (-2,0.2);
        \draw[thick] (4,-0.2) -- (4,0.2);

        % Labels
        \node at (-1, -0.4) {\small $\Omega_1$};
        \node at (1, -0.4) {\small $\Omega_2$};
        \node at (3, -0.4) {\small $\Omega_3$};
        
        % \node[anchor=north] at (0,-0.2) {$x_1$};
        % \node[anchor=north] at (2,-0.2) {$x_2$};
        
        \node at (1, -2.5) {{(a) One-dimensional domain.}};
    \end{scope}

    % =========================================================
    % RIGHT SIDE: 2D DOMAIN
    % =========================================================
    \begin{scope}[shift={(2,0)}]
        % Subdomains (Smooth Component w^p)
        \fill[My_blue!20] (-2,-2) rectangle (4,2);
        
        % Interface Boxes (Gradient Jump Component v^p)
        \foreach \x in {0,2}
            \fill[My_orange!90] (\x-0.1, -2) rectangle (\x+0.1, 2);
        \fill[My_orange!90] (-2, -0.1) rectangle (4, 0.1);

        % Domain Lines
        \draw[thick] (-2,-2) rectangle (4,2);
        \draw[thick] (0,-2) -- (0,2);
        \draw[thick] (2,-2) -- (2,2);
        \draw[thick] (-2,0) -- (4,0);

        % Domain Labels
        \node at (-1, 1) {\small $\Omega_1$}; \node at (1, 1) {\small $\Omega_2$}; \node at (3, 1) {\small $\Omega_3$};
        \node at (-1, -1) {\small $\Omega_4$}; \node at (1, -1) {\small $\Omega_5$}; \node at (3, -1) {\small $\Omega_6$};

        % Singular Components (s_ij around vertices)
        \def\singularityradius{0.45}
        \begin{scope}
            \clip (-2,-2) rectangle (4,2);
            \draw[ fill=My_green!40, opacity=0.8] (0,0) circle (\singularityradius);
            \draw[ fill=My_green!40, opacity=0.8] (2,0) circle (\singularityradius);
        \end{scope}

        % \filldraw (0, 0) circle (2pt) node[anchor=north east] {$x_1$};
        % \filldraw (2, 0) circle (2pt) node[anchor=north east] {$x_2$};
        
        \node at (1, -2.5) {{(b) Two-dimensional domain.}};
    \end{scope}

    % =========================================================
    % SHARED LEGEND (Bottom Centered)
    % =========================================================
\begin{scope}[shift={(-8,-4)}]
    % Smooth Component
    \fill[My_blue!20, draw=My_blue, thick] (0,6.6) rectangle (0.7,6.9);
    \node[right, black] at (0.8,6.75) {Smooth component $\mathfrak{w}^p$};
    
    % Gradient Jump Component
    \fill[My_orange!90, draw=My_orange, thick] (5,6.6) rectangle (5.7,6.9);
    \node[right, black] at (5.8,6.75) {Gradient jump component $\mathfrak{v}^p$};
        
    % Singular Component
    \fill[My_green!40, draw=My_green, thick] (11,6.6) rectangle (11.7,6.9);
    \node[right, black] at (11.7,6.75) {Singular component $\mathfrak{s}
    ^p$};
\end{scope}

\end{tikzpicture}}
        }
    \caption{Active regions of the decomposed solution components: the smooth part $\mathfrak{w}^p$ is active over the entire domain $\Omega$ except on the material interfaces, while the gradient jump part $\mathfrak{v}^p$ and the singular component $\mathfrak{s}^p$ are localized to the material interfaces $\Gamma$ and the singular vertices, respectively.}
    \label{fig_solution_decomposed_1D_2D}
\end{figure}

\noindent Following the framework established by \cite{kellogg1971singularities}, we provide a detailed characterization of the singular functions belonging to $\mathbb{U}^p_3$. They may be described by basis functions $\mathfrak{s}^p_{ij}$ for $j=1,...,K_i$ and $i=1,\ldots, N_s$ indexing each singular vertex. Each function admits a local polar representation of the form
\begin{equation}\label{eq2_sij definition}
\mathfrak{s}^p_{ij}(r, \theta)=r^{\Lambda^p_{ij}}\vartheta^p_{ij}(\theta)\eta (r), \qquad j = 1, \ldots, K_i,
\end{equation}
where $r$ and $\theta$ denote the polar coordinates centered at the vertex $x_i$. Furthermore, $\vartheta^p_{ij}$ is an eigenfunction of the associated Sturm–Liouville problem with corresponding eigenvalue $\Lambda^p_{ij} \in (0, 1)$. That is, $\vartheta^p_{ij}$ and $\Lambda^p_{ij}$ satisfy the following weak formulation
\begin{equation}\label{eq2_SL}
\int_0 ^{2\pi} p(\theta)\Big((\vartheta^p _{ij}(\theta))'v'(\theta) + (\Lambda^p _{ij})^2 \vartheta^p _{ij} (\theta) v(\theta)\big) \, d \theta = 0, \qquad \forall v \in H^1_{\text{per}}(0, 2\pi),
\end{equation}
where $H^1_{\text{per}}(0, 2\pi)$ denotes the Sobolev space of $H^1$ functions on $(0, 2\pi)$ satisfying periodic boundary conditions. In \eqref{eq2_SL}, the parameter $p(\theta)$ denotes the localized angular trace of the coefficient $p$ surrounding the vertex. To ensure local support, we introduce a smooth radial cutoff function $\eta: \mathbb{R}^d \to \mathbb{R}$ defined in polar coordinates as follows
\begin{equation}\label{eq2_eta}
    \eta(r) = \begin{cases} 
    1, & \text{if } r < \delta_1, \\ 
    0, & \text{if } r > \delta_2, 
    \end{cases}
\end{equation}
where $0 < \delta_1 < \delta_2$, with $\delta_1$ and $\delta_2$ such that the support of $\eta$ is strictly contained within the subdomains incident to $x_i$, avoiding overlap with other singular points.

\noindent Crucially, while the singular functions $\mathfrak{s}_{ij}^p$ are harmonic near the vertex ($r < \delta_1$), the use of the cutoff function $\eta(r)$—which we need to keep the effect localized—makes $\mathfrak{s}^p$ non-harmonic in the transition zone $\delta_1 < r < \delta_2$. Consequently, removing the singular part from the total solution $\mathfrak{u}^p$ pushes a new residual into the right-hand side of the equation. This creates a modified source term that affects the smooth and gradient components of the solution. Finally, because these singular functions are constructed with zero flux jump ($[ p \nabla \mathfrak{s}^p_{ij} \cdot \mathbf{n} ] = 0$), they do not interfere with the physical continuity requirements across the interfaces $\Gamma$.

\noindent In our subsequent numerical simulations, we emulate this decomposition by constructing a discretized trial space composed of: smooth functions $\mathbb{U}_1\subset H^2(\Omega)\cap H^1_0(\Omega)$, functions $\mathbb{U}_2\subset H^2(\Omega\setminus\Gamma)\cap H^1_0(\Omega)$ that admit gradient discontinuities across interfaces, and a discretization of $\mathbb{U}^p_3$. The resulting trial space is defined via the direct sum
\begin{equation}\label{eq2_space_decomposition}
    \mathbb{U}^p = \mathbb{U}_1+\mathbb{U}_2+\mathbb{U}^p_3. 
\end{equation}
\noindent Note that $\mathbb{U}^p \subset \mathbb{X}^p$, ensuring that any candidate solution in our constructed space satisfies the global boundary conditions and possesses the necessary local regularity.

\subsection{Continuum loss function}\label{Continuum loss function}
For each $p \in \mathbb{P}$, and a reduced-order subspace $U(p) \subset \mathbb{U}^p$, the corresponding non-parametric loss function is defined by
\begin{equation} \label{eq2_loss-nonparametric}
\mathcal{L}(U(p); p) = \min_{u \in U(p)} \|u-\mathfrak{u}^p\|_{\mathbb{X}^p}^2 
%= \min_{u \in U^p}  \Big(\sum\limits_{i}\|\text{div}(p\nabla u)-\ell ^p\|_{L^2(\Omega_i)}^2 + \|[p\nabla u\cdot \mathbf{n}]\|_{L^2(\Gamma)}^2\Big).
\end{equation}
This formulation simultaneously minimizes the residuals of the governing equation within each subdomain and the flux mismatches across the interfaces. To ensure the reduced subspace is robust across the parameter manifold, we extend the formulation \eqref{eq2_loss-nonparametric} to the entire parameter space $\mathbb{P}$. The global loss function, $\mathcal{L}_\mu(U)$, is defined as the expected approximation norm with respect to the parameter distribution $\mu(p)$
\begin{equation}\label{eq2_continouse loss}
\mathcal{L}_\mu(U) = \int_{\mathbb{P}} \mathcal{L}(U(p); p) \, d\mu(p) 
\end{equation}
The minimization of \eqref{eq2_continouse loss} with respect to $U(p)$ corresponds to finding an optimal, reduced-order space $U(p)$ spanned by a set of functions that are representative of the solution manifold across $\mathbb{P}$. Rather than seeking a single good solution for a specific $p$, this approach identifies low-dimensional, parameter-dependent subspaces over which a good approximation can be found by a low-dimensional least-squares problem.

%%%%%%%%%%%%%%%%%%%%%%%%%%%%%%%%%%%%%%%%%%%%%%%%%%%%
%%%%%%%%%%%%%%%%%%%%%%%%%%%%%%%%%%%%%%%%%%%%%%%%%%%%
%%%%%%%%%%%%%%%%%%%%%%%%%%%%%%%%%%%%%%%%%%%%%%%%%%%%
%%%%%%%%%%%%%%%%%%%%%%%%%%%%%%%%%%%%%%%%%%%%%%%%%%%%
%%%%%%%%%%%%%%%%%%%%%%%%%%%%%%%%%%%%%%%%%%%%%%%%%%%%
%%%%%%%%%%%%%%%%%%%%%%%%%%%%%%%%%%%%%%%%%%%%%%%%%%%%

\section{LS-ReCoNN Framework}\label{LS-ReCoNN Framework}
\noindent In this section, we integrate the FE method, the ReCoNN framework \cite{taylor2024regularity}, and the LS-Net method \cite{baharlouei2024least} to develop an NN-based approach capable of accurately approximating parametric solutions within each subdomain while capturing interactions across interfaces, including complex behavior near singular points. Before delving into the details, we first provide an overview of the LS-ReCoNN strategy. The method represents the solution as a linear combination of \emph{parameter-independent} and \emph{parameter-dependent} functions. The parameter-independent functions are parameterized by NNs. They are designed to capture both smooth behavior across the entire domain and low-regularity features near the interfaces. As a result, they provide a discretization of the solution space associated with the principal part. In contrast, parameter-dependent basis functions are derived from the solution of Sturm–Liouville eigenvalue problems using the FEM to represent the singular components of the solution. The coefficients of these functions are determined by minimizing a quadratic loss, resulting in an LS system.

\noindent The following subsection describes the spatial discretization, which is designed to match the decomposition in equation \eqref{eq2_space_decomposition}. This includes the use of cutoff functions and interfacial features. We then present the formal structure of the LS-ReCoNN approximate solution, followed by the formulation and discretization of its loss function.

\subsection{Solution Space Discretization}
Following our theoretical framework, we discretize the solution space into three distinct subspaces. The first two subspaces, $U_1^\alpha$ and $U_2^\alpha$, which handle the principal components of the solution, are generated by a single NN with trainable parameters $\alpha$ that produce a set of functions that, by definition, span $U_1^\alpha$ and $U_2^\alpha$. This NN is independent of the material parameters. In contrast, the third space, $U_3^p$, specifically designed to capture singularities, is $p$-dependent and is constructed using an FE eigenvalue solver. This gives the discretization:
\begin{equation}
U^\alpha(p) = U_1^\alpha + U_2^\alpha + U_3^p
\end{equation}
The details for building each of these components are provided in the following.
\subsubsection*{Regularity-Conforming Discretization of $\mathbb{U}_1$ and $\mathbb{U}_2$}
Recalling the solution decompositions in \eqref{eq2_exact solution_1D} and \eqref{eq2_exact solution_2D}, we employ a deep NN to learn the reduced-order basis functions for the smooth and gradient jump components. Let $U_1^\alpha \subset \mathbb{U}_1$ and $U_2^\alpha \subset \mathbb{U}_2$ denote finite-dimensional trial subspaces parameterized by the learnable network parameters $\alpha$. Crucially, these subspaces are independent of the parameter $p$, allowing the network to learn a universal ($p$-independent) set of functions that are later enriched by the $p$-dependent singular part. We utilize a single NN architecture that maps a spatial coordinate $x \in \Omega$ to an output vector of dimension $N = N_1 + N_2$, where $N_1$ and $N_2$ correspond to the prescribed dimensions of $U_1^\alpha$ and $U_2^\alpha$, respectively. Figure \eqref{fig2_NN_architecture_1D} and Appendix \ref{Neural Network Structure} provide further details on the network architecture. The network is formally defined as the mapping
\begin{equation}
x \mapsto \left( \bar{\mathbf{w}}^\alpha(x), \bar{\mathbf{v}}^\alpha(x) \right)^T \in \mathbb{R}^{N_1 + N_2},
\end{equation}
where $\bar{\mathbf{w}}^\alpha(x) = (\bar{w}^\alpha_1, \dots, \bar{w}^\alpha_{N_1})$ and $\bar{\mathbf{v}}^\alpha(x) = (\bar{v}^\alpha_1, \dots, \bar{v}^\alpha_{N_2})$.
To ensure that these outputs conform to the required functional regularity and boundary conditions, each component is transformed through specialized cutoff functions. Specifically, we incorporate three distinct cutoffs as follows:

\begin{itemize}
    \item \textbf{Boundary cutoff: } To enforce homogeneous Dirichlet boundary conditions, the network outputs are multiplied component-wise by a smooth cutoff function $\mathcal{B}(x)$  \cite{berrone2023enforcing}. Further details regarding $\mathcal{B}(x)$ are provided in Appendix \ref{Boundary Cutoff Function}.

    \item \textbf{Gradient jump cutoff:} To accommodate gradient discontinuities without triggering Gibbs-type oscillations, the jump component basis functions $\bar{\mathbf{v}}^\alpha$ are multiplied component-wise by a specialized cutoff vector $\Psi = (\psi_1, \psi_2, \ldots, \psi_{N_2})^T$. Following the approach in \cite{berrone2023enforcing}, each basis function $\bar{v}^\alpha_n$ is associated with a cutoff $\psi_n \in C(\Omega) \cap C^\infty(\Omega \setminus \Gamma_n)$, where $\Gamma_n \subset \Gamma$ denotes a specific subregion of the interfaces. These functions are designed to satisfy $\psi_n = 0$ on $\partial\Omega$ and exhibit a unit normal derivative jump across $\Gamma_n$, i.e.,
    \begin{equation}
    \lim_{y \to x^\pm} \nabla \psi_n(y) \cdot \mathbf{n}(x) = \pm 1, \quad \text{for } x \in \Gamma_n.
    \end{equation}
    By utilizing these properties, the one-sided traces of the normal derivatives for the resulting basis functions can be evaluated analytically
    \begin{equation}
    \begin{split}
    \lim\limits_{y\to x^\pm} \nabla(\psi_n(y)\bar{v}^\alpha_n(y))\cdot \mathbf{n}(y)=&\lim\limits_{y\to x^\pm} \bar{v}^\alpha_n(y)\nabla\psi_n(y)\cdot \mathbf{n}(y)+\psi_n(y)\nabla \bar{v}^\alpha_n(y)\cdot \mathbf{n}(y)\\
    =& \bar{v}^\alpha_n(x)\lim\limits_{y\to x^\pm} \nabla\psi_n(y)\cdot \mathbf{n}(y)+ \psi_n(x)\lim\limits_{y\to x^\pm}\nabla \bar{v}^\alpha_n(y)\cdot \mathbf{n}(y)\\
    =&\bar{v}^\alpha_n(x).
    \end{split}
    \end{equation}
    This construction allows the normal derivative jumps—which are essential for the $\mathbb{X}^p$-norm loss evaluation—to be represented and calculated directly from the underlying NN outputs $\bar{v}^\alpha_n$. By distributing these jump-enriched functions across segments of $\Gamma$, we ensure a sufficiently rich representation of the flux behavior and achieve partial explainability.  The particular construction and formula for the employed $\Psi$ are outlined in detail in Appendix \ref{Gradient jump cutoff}.
    
    \item \textbf{Singularity exclusion cutoff:} Since the singular functions $\mathfrak{s}_{ij}^p$ incorporate the function $\eta$ to resolve local junction behaviors, we introduce the vector-valued cutoff functions $\Phi_1$ and $\Phi_2$ to modify $\bar{\mathbf{w}}^\alpha$ and $\bar{\mathbf{v}}^\alpha$, respectively. These cutoffs ensure that the NN outputs vanish in the immediate vicinity of singular vertices. By masking the NN at these junctions, we effectively exclude the global trial functions from attempting to resolve the high-gradient kink characteristic of the solution. The specific geometric constructions and formulas for $\Phi_1$ and $\Phi_2$ are detailed in Appendix \ref{Singularity-Exclusion cutoff}.
\end{itemize}
Consequently, the final NN architecture is defined as follows
\begin{equation} \label{eq2_walpha-valpha}
\left(
\begin{array}{c}
\mathbf{w}^\alpha(x) \\
\mathbf{v}^\alpha(x) 
\end{array}
\right) = \mathcal{B}(x) \cdot 
\left(
\begin{array}{c}
\bar{\mathbf{w}}^\alpha(x) \odot \Phi_1(x) \\ \bar{\mathbf{v}}^\alpha(x) \odot \Psi(x) \odot \Phi_2(x) \end{array}\right), \end{equation}
where $\mathbf{w}^\alpha(x) = (w^\alpha_1, \dots, w^\alpha_{N_1})$ and $\mathbf{v}^\alpha(x) = (v^\alpha_1, \dots, v^\alpha_{N_2})$. Furthermore, the trial parameter-independent subspaces are defined as follows
\begin{equation}
U_1^\alpha = \text{span} \{ w^\alpha_n \}_{n=1}^{N_1}, \quad \text{and} \quad U_2^\alpha = \text{span} \{ v^\alpha_n \}_{n=1}^{N_2}.
\end{equation}

\begin{figure}[htbp]

	\begin{center}
	\definecolor{My_blue}{RGB}{0,0,102}
\definecolor{My_Orange}{RGB}{255, 178, 102}
\begin{tikzpicture}[
punkt/.style={
   rectangle,
   rounded corners,
   draw=black, thick,
   text width=2em,
   minimum height=1em,
   text centered}
]
%\node[punkt, minimum width = 24em, minimum height = 12em] (block) at (3.5, 1.5) {};
%\node[punkt, minimum width = 3.5em, minimum height = 12em] (block) at (8.4, 1.5) {};
%\draw (0,0.35) node{Input};

%\draw(9,0.35) node {Output};

% \Vertex[x=10,y=1.5,style={opacity=0.0,color=red!0},size=0.00000001]{inv1}
% \Vertex[x=10,y=0,style={opacity=0.0,color=white!100}]{inv2}

\Vertex[x=0,y=1.5,label=$ x$,color=white,size=1,fontsize=\normalsize]{X}
\draw (0,0.7) node {Input};

\Vertex[x=2,y=3.5,label=$\sigma (\cdot)$,color=white,size=0.85,fontsize=\scriptsize]{A11}
\Vertex[x=4,y=3.5,label=$\sigma (\cdot)$,color=white,size=0.85,fontsize=\scriptsize]{A21}
\Vertex[x=2,y=2.25,label=$\sigma (\cdot)$,color=white,size=0.85,fontsize=\scriptsize]{A12}
\Vertex[x=4,y=2.25,label=$\sigma (\cdot)$,color=white,size=0.85,fontsize=\scriptsize]{A22}
\Vertex[x=2,y=0,label=$\sigma (\cdot)$,color=white,size=0.85,fontsize=\scriptsize]{A13}
\Vertex[x=4,y=0,label=$\sigma (\cdot)$,color=white,size=0.85,fontsize=\scriptsize]{A23}

\Vertex[x=7,y=4.5,label=$\bar{w}_{1}^\alpha \hspace{-0.05cm}(x)$ ,color=white,size=0.85,fontsize=\scriptsize]{A31}
\Vertex[x=7,y=2.5,label=$\bar{w}_{N_1}^\alpha \hspace{-0.05cm}(x)$ ,color=white,size=0.85,fontsize=\scriptsize]{A32}
\Vertex[x=7,y=1.2,label=$\bar{v}_{1}^\alpha \hspace{-0.05cm}(x)$ ,color=white,size=0.85,fontsize=\scriptsize]{A33}
\Vertex[x=7,y=-0.8,label=$\bar{v}_{N_2}^\alpha \hspace{-0.05cm}(x)$ ,color=white,size=0.85,fontsize=\scriptsize]{A34}

\Vertex[x=9,y=4.5,label=$w_{1}^\alpha \hspace{-0.05cm}(x)$ ,color=red!20!white,size=0.85,fontsize=\scriptsize]{A41}
\Vertex[x=9,y=2.5,label=$w_{N_1}^\alpha \hspace{-0.05cm}(x)$ ,color=red!20!white,size=0.85,fontsize=\scriptsize]{A42}
\Vertex[x=9,y=1.2,label=$v_{1}^\alpha \hspace{-0.05cm}(x)$ ,color=red!20!white,size=0.85,fontsize=\scriptsize]{A43}
\Vertex[x=9,y=-0.8,label=$v_{N_2}^\alpha \hspace{-0.05cm}(x)$ ,color=red!20!white,size=0.85,fontsize=\scriptsize]{A44}

% \Vertex[x=9,y=3.5,label=$\bar{w}_{1}^\alpha \hspace{-0.05cm}(x)$ ,color=red!20!white,size=0.85,fontsize=\scriptsize]{A41}
% \Vertex[x=9,y=2.25,label=$\bar{w}_{2}^\alpha \hspace{-0.05cm}(x)$ ,color=red!20!white,size=0.85,fontsize=\scriptsize]{A42}
% \Vertex[x=9,y=0,label=$\bar{w}_{N}^\alpha \hspace{-0.05cm}(x)$ ,color=red!20!white,size=0.85,fontsize=\scriptsize]{A43}

\draw (2,4.2) node {$L_1$};
\draw (4,4.2) node {$L_2$};
\draw (7,5.5) node {$\bar{\mathbf{w}}^\alpha (x)$};
\draw (7,-1.8) node {$\bar{\mathbf{v}}^\alpha (x)$};
\draw (9,5.5) node {${\color{red}\mathbf{w}^\alpha (x)}$};
\draw (9,-1.8) node {${\color{red}\mathbf{v}^\alpha (x)}$};

\Vertex[x=5.5,y=3,style={color=white},label=$\color{black}\hdots$,size=0.75]{Ah1}
\Vertex[x=5.5,y=2,style={color=white},label=$\color{black}\hdots$,size=0.75]{Ah2}
\Vertex[x=5.5,y=0,style={color=white},label=$\color{black}\hdots$,size=0.75]{Ah3}

\Edge[color=black, Direct,lw=0.5pt](X)(A11)
\Edge[color=black, Direct,lw=0.5pt](X)(A12)
\Edge[color=black, Direct,lw=0.5pt](X)(A13)

\Edge[color=white,label={$\color{black}\vdots$}](A12)(A13)
\Edge[color=white,label={$\color{black}\vdots$}](A22)(A23)
\Edge[color=white,label={$\color{black}\vdots$}](Ah2)(Ah3)

\Edge[color=white,label={$\color{black}\vdots$}](A31)(A32)

\Edge[color=white,label={$\color{black}\vdots$}](A33)(A34)

\Edge[color=black, Direct,lw=0.5pt](A11)(A21)
\Edge[color=black, Direct,lw=0.5pt](A11)(A22)
\Edge[color=black, Direct,lw=0.5pt](A11)(A23)
\Edge[color=black, Direct,lw=0.5pt](A11)(A21)
\Edge[color=black, Direct,lw=0.5pt](A12)(A21)
\Edge[color=black, Direct,lw=0.5pt](A12)(A22)
\Edge[color=black, Direct,lw=0.5pt](A12)(A23)
\Edge[color=black, Direct,lw=0.5pt](A13)(A21)
\Edge[color=black, Direct,lw=0.5pt](A13)(A22)
\Edge[color=black, Direct,lw=0.5pt](A13)(A23)

\Edge[color=black, Direct,lw=0.5pt](A23)(Ah1)
\Edge[color=black, Direct,lw=0.5pt](A23)(Ah2)
\Edge[color=black, Direct,lw=0.5pt](A23)(Ah3)
\Edge[color=black, Direct,lw=0.5pt](A21)(Ah1)
\Edge[color=black, Direct,lw=0.5pt](A21)(Ah2)
\Edge[color=black, Direct,lw=0.5pt](A21)(Ah3)
\Edge[color=black, Direct,lw=0.5pt](A22)(Ah1)
\Edge[color=black, Direct,lw=0.5pt](A22)(Ah2)
\Edge[color=black, Direct,lw=0.5pt](A22)(Ah3)

\Edge[color=black, Direct,lw=0.5pt](Ah1)(A31)
\Edge[color=black, Direct,lw=0.5pt](Ah1)(A32)
\Edge[color=black, Direct,lw=0.5pt](Ah1)(A33)
\Edge[color=black, Direct,lw=0.5pt](Ah2)(A31)
\Edge[color=black, Direct,lw=0.5pt](Ah2)(A32)
\Edge[color=black, Direct,lw=0.5pt](Ah2)(A33)
\Edge[color=black, Direct,lw=0.5pt](Ah3)(A31)
\Edge[color=black, Direct,lw=0.5pt](Ah3)(A32)
\Edge[color=black, Direct,lw=0.5pt](Ah3)(A33)

\Edge[color=black, Direct,lw=0.5pt](A31)(A41)
\Edge[color=black, Direct,lw=0.5pt](A32)(A42)
\Edge[color=black, Direct,lw=0.5pt](A33)(A43)
\Edge[color=black, Direct,lw=0.5pt](A34)(A44)

\Edge[color=white,label={$\color{black}\vdots$}](A41)(A42)
\Edge[color=white,label={$\color{black}\vdots$}](A43)(A44)

\node[punkt, minimum width = 4em, minimum height = 8.9em] (block) at (7,3.5) {};

\node[punkt, minimum width = 4em, minimum height = 8.9em] (block) at (7,0.2) {};

\node[punkt, minimum width = 4em, minimum height = 8.9em] (block) at (9,3.5) {};

\node[punkt, minimum width = 4em, minimum height = 8.9em] (block) at (9,.2) {};

\draw (8,5.7) node{ $ \downarrow $};
% \draw (8,6.5) node{ $\Psi$,$\Phi_1$, and $\Phi_2$};
\draw (8,6.2) node{Applying cutoff functions};
\node[punkt, minimum width = 11em, minimum height = 2.em] (block) at (8,6.2) {};

\end{tikzpicture}
	\end{center}
	\caption{NN architecture for the parameter-independent trial functions. The network maps spatial coordinates to a vector of raw outputs that represent the smooth and gradient jump components of the solution.}
	\label{fig2_NN_architecture_1D}
\end{figure}
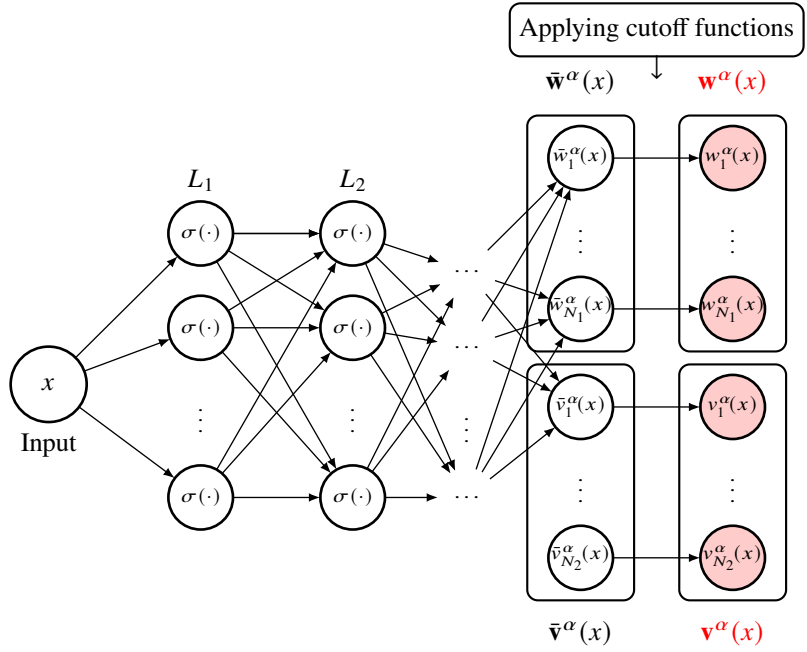

\subsubsection*{Discretization of $\mathbb{U}^p_3$}
To approximate the singular component defined in \eqref{eq2_exact solution_2D}, we introduce a finite-dimensional trial subspace $U_3^p \subset \mathbb{U}^p_3$ of dimension $N_3$. This subspace is spanned by the basis functions $\mathbf{s}^{p}(x) = ((s_{ij}^p)_{i=1}^{N_s})_{j=1}^{N_3}$, which are specifically designed to capture localized behavior near material junctions
\begin{equation}
U_3^p = \text{Span}\{ s_{ij}^p \mid 1 \le i \le N_s,\; 1 \le j \le N_3 \}.
\end{equation}
For each singular point $x_i$, the function $s_{ij}^p$ is expressed in local polar coordinates $(r, \theta)$ as
\begin{equation}\label{eq2_sjp}
s_{ij}^p(r, \theta) = r^{\lambda_i^p} \mu_i^p(\theta) \eta(r), \qquad j = 1, \ldots, N_3,
\end{equation}
where $\lambda_i^p$ and $\mu_i^p(\theta)$ denote the eigenvalues and angular eigenfunctions, respectively. These are determined by solving the associated Sturm–Liouville eigenvalue problem via an FE discretization. Crucially, these eigenfunctions conform to the material geometry; they are globally continuous and smooth everywhere except at the interfaces. Given a set of piecewise-polynomial trial functions $\{\varphi_j\}_{j=1}^{N_3}$, the corresponding stiffness and mass matrices, $G^p$ and $B^p$, are assembled as
\begin{equation}\label{eq2_FE_matrices}
G^p_{ij} = \int_0^{2\pi} p(\theta) \frac{d \varphi_i}{d\theta} \frac{d \varphi_j}{d\theta} d\theta, \qquad B^p_{ij} = \int_0^{2\pi} p(\theta) \varphi_i(\theta) \varphi_j(\theta) d\theta.
\end{equation}
The detailed construction of the basis functions $\{\varphi_j\}_{j=1}^{N_3}$ is provided in Appendix \eqref{Implementation of the FE Eigenvalue Solver}. This yields the following generalized eigenvalue problem
\begin{equation}\label{eq2_FE_eigenvalue_solver}
G^p \rho_i^p = \lambda_i^p B^p \rho_i^p ,
\end{equation}
where $\rho_i$ are the corresponding eigenvectors. The numerical procedure for solving the eigenvalue problem \eqref{eq2_FE_eigenvalue_solver} is detailed in Appendix \ref{Implementation of the FE Eigenvalue Solver}. We specifically select the subset of eigenvalues satisfying $0 < \lambda_i < 1$, as these correspond to the low-regularity singular behavior inherent in transmission problems. Finally, the angular eigenfunctions are reconstructed as a linear combination of the FE basis functions, i.e.,
\begin{equation}
\mu_i^p(\theta) = \sum_{j=1}^{N_3} \rho_{ij}^p \varphi_j, \qquad i = 1, \ldots, N_s.
\end{equation}

\noindent Since the singular functions $s_{ij}^p$ are computed via the FE method rather than determined analytically, they are technically non-conforming. Indeed, unlike the exact singular functions that are harmonic near the junctions (for $r < \delta_1$), any numerical approximation introduces a non-zero residual in this region. Because the Laplacian scales as $r^{\lambda-2}$, with $\lambda<1$, it fails to be in $L^2$, which theoretically causes the continuum $\mathbb{X}^p$-loss to become infinite. A detailed discussion on solving this issue is deferred to Subsection \ref{Loss Discretization}, where we demonstrate how an appropriate numerical approximation of $\|u - \mathfrak{u}^p\|_{\mathbb{X}^p}$ mitigates these errors and quantifies the propagation of errors into the $H^1_0$-error.

%%%%%%%%%%%%%%%%%%%%%%%%%%%%%%%%%%%%%%%%%%%%%
%%%%%%%%%%%%%%%%%%%%%%%%%%%%%%%%%%%%%%%%%%%%%
%%%%%%%%%%%%%%%%%%%%%%%%%%%%%%%%%%%%%%%%%%%%%
%%%%%%%%%%%%%%%%%%%%%%%%%%%%%%%%%%%%%%%%%%%%%

\subsection{LS-ReCoNN Approximate Solution}\label{LS-ReCoNN Approximate Solution}
\subsubsection*{One-Dimensional Problems}
As discussed in Subsection \ref{Decomposition of the solution}, one-dimensional problems exhibit no singularities, since the solution only develops discontinuities in its first derivative at material interfaces. This absence of singular components simplifies the structure of the solution and allows the LS-ReCoNN method to approximate it using the smooth and gradient jump components. The LS-ReCoNN solution is then given by
\begin{align}\label{eq2_LS-ReCoNN_solution_1D}
\begin{array}{rl}
     \mathfrak{u}^p(x) \approx u^{p, \alpha}(x) =& \hspace{-2mm}w^{p, \alpha}(x) + v^{p, \alpha}(x)\\
      := &\hspace{-2mm}\displaystyle \sum_{n = 1}^{N_1} a^{\alpha, p}_n w^{\alpha}_n(x) + \sum_{m = 1}^{N_2} b^{\alpha, p}_m v^{\alpha}_m(x)\\[3mm]
     = &\hspace{-2mm}\displaystyle \underbrace{\mathbf{a}^{p,\alpha} \cdot \mathbf{w}^{\alpha}(x)}_{\substack{\text{Smooth}\\ \text{ component}}} + \underbrace{\mathbf{b}^{p,\alpha} \cdot \mathbf{v}^{\alpha}(x)}_{\substack{\text{Gradient jump}\\ \text{ component}}}, \qquad \forall x \in \Omega \subset \mathbb{R},\\
\end{array}
\end{align}
where $u^{p, \alpha}$ is the LS-ReCoNN solution over the domain $\Omega$ and $\mathbf{w}^{\alpha}$ and $\mathbf{v}^{\alpha}$ are defined in \eqref{eq2_walpha-valpha}. The vectors $\mathbf{a}^{p,\alpha}:= (a_n^{p,\alpha})_{n = 1}^{N_1}$ and $\mathbf{b}^{p,\alpha}:= (b_n^{p,\alpha})_{m = 1}^{N_2}$ denote the coefficients of the optimal approximate solution within the approximation subspace $U_1^\alpha$ and $U_2^\alpha$. For any given parameter $p$, they can be determined by solving an LS optimization problem. The construction of the LS system is detailed in Subsection \ref{Loss Discretization} and Appendix \ref{Matrix Definitions}.
% \begin{equation}\label{eq2_optimal_coeff_1D}
%     (\mathbf{a}^{p,\alpha}, \mathbf{b}^{p,\alpha}) =  \arg\min_{(\mathbf{a}, \mathbf{b})} \left( \Theta \, 
%     \left \Vert p\frac{d^2}{dx^2} (\mathbf{a}\cdot \mathbf{w}^{\alpha} + \mathbf{b}\cdot \mathbf{v}^{\alpha}) - l^p \right\Vert^2_{L^2(\Omega)} 
%     + \left\Vert [p \frac{d}{dx} (\mathbf{a}\cdot \mathbf{w}^{\alpha} + \mathbf{b}\cdot \mathbf{v}^{\alpha}) \cdot \mathbf{n}] \right\Vert^2_{L^2(\Gamma)} 
%     \right).
% \end{equation}
% \begin{remark}
% In one spatial dimension, the $L^2$ norm over the interface $\Gamma$ appearing in the second term of equation~\eqref{eq2_optimal_coeff_1D} simplifies to a finite sum of squared point evaluations at the interface locations. This significantly reduces the computational cost of evaluating the interface residual in 1D problems.
% \end{remark}
 A schematic illustration of the LS-ReCoNN architecture for the one-dimensional case is provided in Figure~\ref{fig_LS_ReCoNN_1D}.

\begin{figure}[htbp]
\begin{center}
\scalebox{1}{\usetikzlibrary{decorations.pathmorphing} % Required for zigzag lines
\begin{tikzpicture}[
punkt/.style={
   rectangle,
   rounded corners,
   draw=black, thick,
   text width=2em,
   minimum height=1em,
   text centered}
]

\Vertex[x=-2.3,y=0,label=$ x$,color=white,size=0.7,fontsize=\normalsize]{X}
\draw (-3.2,0) node {Input};
\draw[thick, ->] (-1.95, 0) -- (-1.3, 0); 
\node[punkt, minimum width = 4.5em, minimum height = 2em] (block) at (-0.4, 0) {};
% \node[punkt, minimum width = 3em, minimum height = 3.4em] (block) at (1.7, 0) {};
\draw (-0.4, 0) node{Deep NN};
\draw[thick, ->] (0.4, 0) -- (1, 0.3); 
\draw[thick, ->] (0.4, 0) -- (1, -0.3); 
\draw (1.6, 0.3) node{${\color{red}\mathbf{w}^{\alpha}}(x)$};
\draw (1.6, -0.3) node{${\color{red}\mathbf{v}^{\alpha}}(x)$};

% \draw[thick] (-2.3, -0.35) -- (-2.3, -1); % Vertical line with arrow
% \draw[thick, ->] (-2.3, -1) -- (1, -1);  % Horizontal line with arrow
% \draw (1.5, -1) node{$\Psi(x)$};

% \draw[thick, ->] (2.1, 0) -- (2.5, -0.4);
% \draw[thick, ->] (2.1, -1) -- (2.5, -0.6);

% \Vertex[x=3.1,y=-0.5,label=${\color{red}\mathbf{w}^{\alpha}}(x)$,color=white,size=1,fontsize=\normalsize]{X}
% \draw (3.1, -0.5) node{${\color{red}\mathbf{w}^{\alpha}}(x)$};

\node[punkt, minimum width = 17em, minimum height = 4.2em, draw, black] (block) at (-.8, 0) {};

\draw (4.2,0) node{ $p, {\color{red}\mathbf{w}^\alpha}, {\color{red}\mathbf{v}^\alpha} $};
\node[punkt, minimum width = 5.5em, minimum height = 2em] (block) at (6.6, 0) {};
\draw (6.6, 0) node{LS solver};
\draw[thick, ->] (5, 0) -- (5.5, 0);

\draw[thick, ->] (7.7, 0) -- (8.2, 0); % 

\node[punkt, minimum width = 18.8em, minimum height = 4.2em, draw, black] (block) at (6.7, 0) {};
\draw (9.1, 0) node{${\color{blue}\mathbf{a}^{p, \alpha}}$, ${\color{blue}\mathbf{b}^{p, \alpha}}$};

%%%%%%%%%%%%%%%%%%%%%%%%%%%%%%%%%%%%%%%%%%%%%%%%%

% Add the flash-like directed lines from the left blue dashed box
\draw[thick] (-1, -.75) -- (-1, -1.9); % Vertical line with arrow
\draw[thick, ->] (-1, -1.9) -- (.2, -1.9);  % Horizontal line with arrow

% Add the flash-like directed lines from the right blue dashed box
\draw[thick] (7.4, -.75) -- (7.4, -1.9); % Vertical line with arrow
\draw[thick, ->] (7.4, -1.9) -- (6.2, -1.9); % Horizontal line with arrow

%%%%%%%%%%%%%%%%%%%%%%%%%%%%%%%%%%%%%%%%%%
\node[punkt, minimum width = 16.6em, minimum height = 2.5em] (block) at (3.2, -1.9) {};
\draw (3.2, -1.9) node{$\displaystyle u^{p, \alpha}(x) = {\color{blue}\mathbf{a}^{p,\alpha}} \cdot {\color{red}\mathbf{w}^\alpha}(x) + {\color{blue}\mathbf{b}^{p,\alpha}} \cdot {\color{red}\mathbf{v}^\alpha}(x)$};

\end{tikzpicture}}
\end{center}		
\caption{Schematic of the LS-ReCoNN architectures for solving one-dimensional parametric transmission problems.}\label{fig_LS_ReCoNN_1D}
\end{figure}
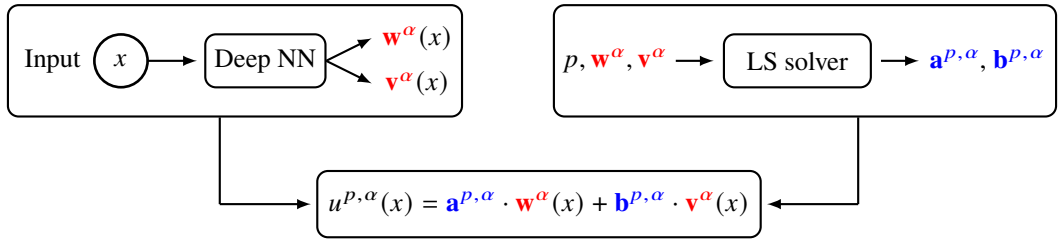

%%%%%%%%%%%%%%%%%%%%%%%%%%%%%%%%%%%%%%%%%%%%%
%%%%%%%%%%%%%%%%%%%%%%%%%%%%%%%%%%%%%%%%%%%%%
%%%%%%%%%%%%%%%%%%%%%%%%%%%%%%%%%%%%%%%%%%%%%
%%%%%%%%%%%%%%%%%%%%%%%%%%%%%%%%%%%%%%%%%%%%%

\subsubsection*{Two-Dimensional Problems}
In two-dimensional problems, in addition to discontinuities in the first derivative across domain interfaces, singular points may also arise. According to the solution structure given in \eqref{eq2_exact solution_1D} and \eqref{eq2_exact solution_2D}, it is essential to incorporate singular functions into the approximate solution.

\noindent Recalling the definition of $\mathbf{w}^\alpha(x)$ and $\mathbf{v}^\alpha(x)$ from Equation \eqref{eq2_walpha-valpha} and $s^p_{ij}(x)$ from \eqref{eq2_sjp}, the LS-ReCoNN approximate solution for two-dimensional problems is expressed as follows

\begin{align}\label{eq2_LS-ReCoNN_solution_2D}
\begin{array}{rl}
     \mathfrak{u}^p(x) \approx u^{p, \alpha}(x) =&w^{p, \alpha}(x) + v^{p, \alpha}(x) + s^{p, \alpha}(x)\\:=&\hspace{-0.2cm} \displaystyle \sum_{n = 1}^{N_1} a^{\alpha, p}_n w^{\alpha}_n(x) \hspace{0.2cm} + \sum_{m = 1}^{N_2} b^{\alpha, p}_m v^{\alpha}_m(x) \hspace{0.2cm}+ \displaystyle \sum _{i=1}^{N_s}\sum _{j=1}^{N_3} {c_{ij}^{\alpha, p}} s_{ij}^{p}(x) \\[3mm]
     =& \hspace{-0.2cm} \displaystyle \underbrace{\mathbf{a}^{p,\alpha} \cdot \mathbf{w}^{\alpha}(x)}_{\substack{\text{Smooth}\\ \text{ component}}} + \underbrace{\mathbf{b}^{p,\alpha} \cdot \mathbf{v}^{\alpha}(x)}_{\substack{\text{Gradient jump}\\ \text{ component}}} +\underbrace{{\sum _{i=1}^{N_s}\mathbf{c}^{\alpha, p}_i} \cdot \mathbf{s}^{p}_i(x)}_{\substack{\text{Singular}\\ \text{ component}}}, \qquad \forall x \in \Omega \subset \mathbb{R}, \\[2mm]
\end{array}
\end{align}
where $\mathbf{a}^{p, \alpha} = (a_n^{p, \alpha})_{n=1}^{N_1}$, $\mathbf{b}^{p, \alpha} = (b_m^{p, \alpha})_{m=1}^{N_2}$ and {$\mathbf{c}^{\alpha, p}_i = (c_{ij}^{\alpha, p})_{j=1}^{N_3}$} are the coefficient vectors, which are unknown constants and considered as the weights of the optimal approximate solution within the approximation subspaces $U_1^\alpha$, $U_2^\alpha$, $U_3^p$, respectively. For each parameter $p$, they can be determined by solving an LS optimization problem. The construction of the LS system is detailed in Subsection \ref{Loss Discretization} and Appendix \ref{Matrix Definitions}.
% \begin{equation}\label{eq2_optimal_coeff_2D}
% \begin{array}{rl}
%     (\mathbf{a}^{p,\alpha}, {\mathbf{b}^{p, \alpha}} , {\mathbf{c}^{p, \alpha}}) =  \displaystyle \arg\min_{(\mathbf{a}, \mathbf{b}, \mathbf{c})} \Big( \hspace{-2mm}&\Theta \, 
%     \left\Vert p \,( \Delta (\mathbf{a}\cdot\mathbf{w}^{\alpha} + \mathbf{b}\cdot\mathbf{v}^{\alpha} )+ \mathbf{c}\cdot \mathbf{S}^{p}) - l^p \right\Vert^2_{L^2(\Omega)} 
%     +\\
%     &\left\Vert \left[ p \, \nabla (\mathbf{a}\cdot\mathbf{w}^{\alpha} + \mathbf{b}\cdot\mathbf{v}^{\alpha}) \cdot \mathbf{n} \right] \right\Vert^2_{L^2(\Gamma)} 
%     \Big),
% \end{array}
% \end{equation}
A schematic illustration of the LS-ReCoNN architecture for the two-dimensional case is provided in Figure~\ref{fig_LS_ReCoNN_2D}.
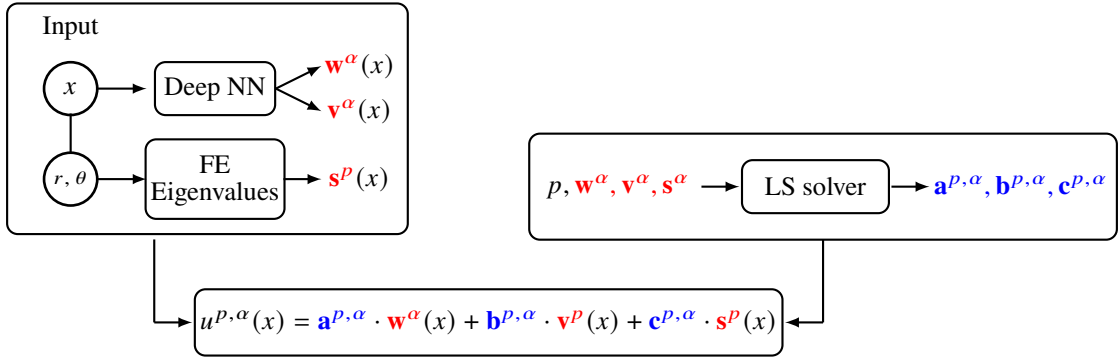
\begin{figure}[H]
\begin{center}
\scalebox{1}{\usetikzlibrary{decorations.pathmorphing} % Required for zigzag lines
\begin{tikzpicture}[
punkt/.style={
   rectangle,
   rounded corners,
   draw=black, thick,
   text width=2em,
   minimum height=1em,
   text centered}
]

\Vertex[x=-2.3,y=0,label=$ x$,color=white,size=0.7,fontsize=\normalsize]{X}
\draw (-2.3,.8) node {Input};
\draw[thick, ->] (-1.95, 0) -- (-1.3, 0); 
\node[punkt, minimum width = 4.5em, minimum height = 2em] (block) at (-0.4, 0) {};
\draw (-0.4, 0) node{Deep NN};
\draw[thick, ->] (0.4, 0) -- (1, 0.3); 
\draw[thick, ->] (0.4, 0) -- (1, -0.3); 
\draw (1.5, 0.3) node{${\color{red}\mathbf{w}^{\alpha}}(x)$};
\draw (1.5, -0.3) node{${\color{red}\mathbf{v}^{\alpha}}(x)$};
%%%%%%%%%%%%%%%%%%%%%%%%%%%%%%%%%%%%%
% \draw[thick] (-2.1, -0.3) -- (-2.1, -0.8); % Vertical line with arrow
% \draw[thick, ->] (-2.1, -0.8) -- (1, -0.8);  % Horizontal line with arrow
% \draw (1.5, -0.8) node{$\bar{\Psi}(x)$};
% %%%%%%%%%%%%%%%%%%%%%%%%%%%%%%%%%%%%%
% \draw[thick] (-2.2, -0.33) -- (-2.2, -1.5); % Vertical line with arrow
% \draw[thick, ->] (-2.2, -1.5) -- (1, -1.5);  % Horizontal line with arrow
% \draw (1.5, -1.5) node{$\Phi(x)$};
%%%%%%%%%%%%%%%%%%%%%%%%%%%%%%%%%%%%%

% \draw[thick, ->] (2.1, 0) -- (2.5, -0.5);
% \draw[thick, ->] (2.1, -0.8) -- (2.5, -0.8);
% \draw[thick, ->] (2.1, -1.5) -- (2.5, -1.1);
% \draw (3.2, -0.8) node{${\color{red}\mathbf{w}^{\alpha}}(x)$};

%%%%%%%%%%%%%%%%%%%%%%%%%%%%%%%%%%%%%

\draw[thick] (-2.3, -0.33) -- (-2.3, -.9); % Vertical line with arrow
\Vertex[x=-2.3, y=-1.2, label={$r,\theta$}, size=0.7, color=white]{r}
% \draw (1.5, -1.5) node{$\Phi(x)$};
%%%%%%%%%%%%%%%%%%%%%%%%%%%%%%%%%%%%%
\draw[thick, ->] (-1.95, -1.2) -- (-1.3, -1.2); 
\node[punkt, minimum width = 5.2em, minimum height = 3em] (block) at (-0.4, -1.2) {};
\draw (-0.4, -1) node{FE};
\draw (-0.4, -1.4) node{Eigenvalues};
\draw[thick, ->] (0.5, -1.2) -- (1, -1.2); 
\draw (1.5, -1.2) node{${\color{red}\mathbf{s}^{p}}(x)$};
%%%%%%%%%%%%%%%%%%%%%%%%%%%%%%%%%%%%%
% \draw[thick] (-0.4, -2.7) -- (-0.4, -2.9); % Vertical line with arrow
% \draw[thick, ->] (-0.4, -2.9) -- (1, -2.9);  % Horizontal line with arrow
% \draw (1.5, -2.9) node{$\lambda$};
%%%%%%%%%%%%%%%%%%%%%%%%%%%%%%%%%%%%%
% \draw[thick] (-2.3, -2.55) -- (-2.3, -3.6); % Vertical line with arrow
% \draw[thick, ->] (-2.3, -3.6) -- (1, -3.6);  % Horizontal line with arrow
% \draw (1.5, -3.6) node{$\eta(r)$};
%%%%%%%%%%%%%%%%%%%%%%%%%%%%%%%%%%%%%

% \draw[thick, ->] (2.1, -2.2) -- (2.5, -2.5);
% \draw[thick, ->] (2.1, -2.8) -- (2.5, -2.8);
% \draw[thick, ->] (2.1, -3.5) -- (2.5, -3);
% \draw (3.2, -2.8) node{${\color{red}\mathbf{s}^{p}}(x)$};
%%%%%%%%%%%%%%%%%%%%%%%%%%%%%%%%%%%%%
\node[punkt, minimum width = 15em, minimum height = 8.7em, draw, black] (block) at (-0.5, -.4) {};

% \node[punkt, minimum width = 3.1em, minimum height = 8em, draw, red] (block) at (3.2, -1.8) {};
% \draw (3.2, .1) node{${\color{red}\mathbf{u}^{p}}(x)$};
%%%%%%%%%%%%%%%%%%%%%%%%%%%%%%%%%%%%%

\draw (4.9,-1.3) node{ $p, {\color{red}\mathbf{w}^\alpha, \mathbf{v}^\alpha, \mathbf{s}^\alpha} $};
\node[punkt, minimum width = 5.5em, minimum height = 2em] (block) at (7.5, -1.3) {};
\draw (7.5, -1.3) node{LS solver};
\draw[thick, ->] (6, -1.3) -- (6.5, -1.3); 
\draw[thick, ->] (8.5, -1.3) -- (9., -1.3); % 

\node[punkt, minimum width = 22em, minimum height = 4em, draw, black] (block) at (7.6, -1.3) {};
\draw (10.2, -1.3) node{${\color{blue}\mathbf{a}^{p, \alpha}, \mathbf{b}^{p, \alpha}, \mathbf{c}^{p, \alpha}}$};

%%%%%%%%%%%%%%%%%%%%%%%%%%%%%%%%%%%%%%%%%%%%%%%%%

% Add the flash-like directed lines from the left blue dashed box
\draw[thick] (-1.2, -2) -- (-1.2, -3.1); % Vertical line with arrow
\draw[thick, ->] (-1.2, -3.1) -- (-0.7, -3.1);  % Horizontal line with arrow

% Add the flash-like directed lines from the right blue dashed box
\draw[thick] (7.6, -2) -- (7.6, -3.1); % Vertical line with arrow
\draw[thick, ->] (7.6, -3.1) -- (7.1, -3.1); % Horizontal line with arrow

%%%%%%%%%%%%%%%%%%%%%%%%%%%%%%%%%%%%%%%%%%
\node[punkt, minimum width = 22em, minimum height = 2.5em] (block) at (3.2, -3.1) {};
\draw (3.2, -3.1) node{$\displaystyle u^{p, \alpha}(x) = {\color{blue}\mathbf{a}^{p,\alpha}} \cdot {\color{red}\mathbf{w}^\alpha}(x) + {\color{blue}\mathbf{b}^{p, \alpha}} \cdot {\color{red}\mathbf{v}^p}(x)+ {\color{blue}\mathbf{c}^{p, \alpha}} \cdot {\color{red}\mathbf{s}^p}(x)$};

\end{tikzpicture}}
\end{center}		
\caption{Schematic of the LS-ReCoNN architectures for solving two-dimensional parametric transmission problems.}\label{fig_LS_ReCoNN_2D}
\end{figure}

\vspace{-.5cm}
\subsection{LS-ReCoNN Loss Function}\label{LS-ReCoNN Loss Function}
As previously noted, the singular functions $\{s_j^p\}_{j=1}^{N_3}$ are computed using an FE eigenvalue solver, rendering the methodology non-conforming.  Specifically, the small residuals at the junctions associated with numerical approximation would cause the exact $\mathbb{X}^p$-norm \eqref{eq2_continouse loss} to become infinite. To address this, we formulate a specific loss function tailored to the discretized space.

\noindent 
Using the decompositions \eqref{eq2_exact solution_1D} and \eqref{eq2_exact solution_2D}, and the derivations in Appendix \ref{Strong Formulation of the Decomposed Solution}, the strong formulation \eqref{eq2_strong} is reformulated as follows: 
Find the stress intensity factors $\kappa^p_{ij} \in \mathbb{R}$ and the components $\mathfrak{w}^p \in \mathbb{U}_1$ and $\mathfrak{v}^p \in \mathbb{U}_2$ such that
\begin{align}
   \label{eq2_transmission elliptic problem}p\left(\Delta (\mathfrak{w}^p + \mathfrak{v}^p)+ \sum _{i = 1}^{N_s}\sum _{j = 1}^{K_i}\kappa^p_{ij} \mathcal{S}^p_{ij}\right) = l^p,& \qquad \text{in}\,\, \Omega\setminus\Gamma,\\
\label{eq2_Continuity-flux}
[p \nabla \mathfrak{v}^p \cdot \mathbf{n} ] + [p] (\nabla \mathfrak{w}^p \cdot \mathbf{n}) = 0,
% \left[p \nabla \mathfrak{w}^p\cdot\mathbf{n}\right]=0,
& \qquad \text{on}\,\, \Gamma,
\end{align}
Since the singular basis functions $\mathfrak{s}^p_{ij}$ and the cutoff function $\eta$ are predetermined for $i = 1, \ldots, N_s$ and $j = 1, \ldots, K_i$, we define $\mathcal{S}^{p}_{ij}$ in polar coordinates $(r, \theta)$ centered at the singular point $x_j$, as follows
\begin{equation}\label{eq2_mathcal{S}ij_main}
\mathcal{S}^p_{ij}(r, \theta):= \Delta (\mathfrak{s}^p_{ij} \eta) =
\begin{cases}
2\, \partial_r \mathfrak{s}^p_{ij}(r, \theta)\, \eta'(r) + \mathfrak{s}^p_{ij}(r, \theta)\left( \eta''(r) + \frac{1}{r}\eta'(r) \right), & \text{if } r \in (\delta_1, \delta_2), \\
0, & \text{otherwise}.
\end{cases}
\end{equation}
%Moreover, provided $\eta$ is a $C^2$ function, $\mathcal{S}^p_{ij}$ belongs to $L^2$ for $i = 1, \ldots, N_s$ and $j = 1, \ldots, K_i$.
Note that $\mathcal{S}^p_{ij}$ is supported exclusively on the annulus $\delta_1 < r < \delta_2$ corresponding to the transition region of the singularity. Outside this interval, the term vanishes: for $r \le \delta_1$, the singular function is exactly harmonic ($\Delta \mathfrak{s}^p_{ij} = 0$), while for $r \ge \delta_2$, the product is zero as the cutoff function $\eta$ vanishes.

\noindent Based on the continuous loss formulation in \eqref{eq2_continouse loss}, the strong formulation \eqref{eq2_strong}, and the LS-ReCoNN solutions given in \eqref{eq2_LS-ReCoNN_solution_1D} and \eqref{eq2_LS-ReCoNN_solution_2D}, the LS-ReCoNN loss function is expressed as
\begin{equation}\label{eq2_loss_total_LS-ReCoNN}
\begin{array}{rl}
{\mathcal{J}}_\mu(U^{\alpha})= \displaystyle \int_\mathbb{P} \,\min_{u \in U^{\alpha}(p)}\Big( & \hspace{-2mm}\displaystyle \,\left\Vert p \left( \Delta \left(\mathbf{a}^{p, \alpha}\cdot\mathbf{w}^{\alpha} + \mathbf{b}^{p, \alpha}\cdot\mathbf{v}^{\alpha}\right) + \sum _{i=1}^{N_s}\sum _{j=1}^{N_3} {c_{ij}^{\alpha, p}} \mathbf{S}_{ij}^{p} \right) - l^p \right\Vert^2_{L^2(\Omega)} + \\
&\hspace{-2mm}\Theta \left\Vert [p \nabla (\mathbf{a}^{p, \alpha}\cdot \mathbf{w}^{p} + \mathbf{b}^{p, \alpha}\cdot \mathbf{v}^{p}) \cdot \mathbf{n}] \right\Vert_{L^2(\Gamma)}^2 
\Big)\,d\mu,
\end{array}   
\end{equation}
where $U^{\alpha}(p) = U_1^\alpha + U_2^\alpha + U_3^p$. In our discrete implementation, we define the singular source term $\mathbf{S}_{ij}^{p}$ using the following radial representation rather than evaluating the Laplacian of the FE-based singular function directly
\begin{equation}\label{eq2_mathbf{S}j}
\mathbf{S}_{ij}^{p}(r, \theta) =
\begin{cases}
2\, \partial_r s^p_{ij}(r, \theta)\, \eta'(r) + s^p_{ij}(r, \theta)\left( \eta''(r) + \dfrac{1}{r} \eta'(r) \right), & \text{if } r \in (\delta_1, \delta_2), \\
0, & \text{otherwise}.
\end{cases}
\end{equation}
If the singular functions $s_{ij}^p$ were exact, the representation \eqref{eq2_mathbf{S}j} would be identical to \eqref{eq2_mathcal{S}ij_main}. However, adopting this form for our FE-computed bases effectively avoids the numerical integration of singularities at the junctions. Indeed, by ignoring the Laplacian in the immediate vicinity of the vertex ($r \leq \delta_1$), we mitigate the ill-conditioned behavior associated with the non-conforming approximation and ensure that the loss remains finite and stable.

\noindent In the following, we state that the loss function ~\eqref{eq2_loss_total_LS-ReCoNN} serves as an upper bound for the approximation error in the \(H_0^1(\Omega)\) norm. Crucially, while our use of FE approximations for the singular components introduces numerical errors, following theorem allows us to quantify these effects. This ensures that we maintain rigorous control over the $H_0^1(\Omega)$ error, even with a non-conforming trial space. Its proof is deferred to
Appendix \ref{Proof of the Error Bound}.

\begin{theorem}\label{Theorem_main}
There exists a constant $C > 0$, independent of the LS-ReCoNN solution $u^{p, \alpha}$, such that
\begin{equation}\label{eq2_upper_bound_H01_total}
\int_{\mathbb{P}} \left\| u^{p, \alpha} - \mathfrak{u}^p \right\|_{H_0^1(\Omega)}^2 \, d\mu(p)
\leq C \left( \mathcal{J}_\mu(U^{p, \alpha})
+ \int_{\mathbb{P}} \| \mathbf{c}^{p, \alpha}\|_{\ell ^2}^2 \left( \| \mu^p - \vartheta ^p \|_{L^2(0, 2\pi)}^2 + \| \lambda^p - \Lambda^p \|_{\ell ^2}^2 \right)\, d\mu(p) \right) ,
\end{equation}
where for any $z \in\{ \vartheta^p, \Lambda^p,\mathbf{c}^{p, \alpha}, \mu^p,  \lambda^p\} $, we have  $z = \left((z_{ij})_{i=1}^{N_s}\right)_{j=1}^{N_3}$.
\end{theorem}

\noindent Theorem~\ref{Theorem_main} establishes that minimizing the loss functional \( \mathcal{J}_\mu(U^{p, \alpha}) \) on space \( U^{p, \alpha} \), together with reducing the approximation errors in \( \lambda^p \) and \( \vartheta^p \) for each $p \in \mathbb{P}$, leads to an accurate approximation of the exact solution. By isolating the errors associated with the singular components, the theorem confirms that the $H^1_0$ error remains bounded as long as the FE eigenvalue solver and the loss minimization are sufficiently accurate. Moreover, the result confirms the intuition that when the stress intensity factors are large, errors in the singular functions have a more pronounced impact on the overall \( H^1_0 \)-norm error. Significantly, the upper bound in~\eqref{eq2_upper_bound_H01_total} depends only on the $L^2$ errors of the angular functions and not their $H^1$ errors.

\subsection{Loss Discretization}\label{Loss Discretization}
\noindent For a practical implementation, the loss function \eqref{eq2_loss_total_LS-ReCoNN} must be discretized. This involves two steps: approximating the integral over the parameter space and discretizing the $L^2$-norms. In this work, we utilize the Monte Carlo approach for both steps.

\subsubsection*{ Integral Approximation over the Parameter Space}
The loss function ${\mathcal{J}}_\mu(U^{p, \alpha})$ is approximated using a Monte Carlo estimate using a set of i.i.d. samples $P \subset \mathbb{P}$. The specific sampling procedures are detailed in Appendix \ref{Parameter Sampling Details}. Accordingly, the discrete loss is computed as the mean of the PDE residuals and interface jumps across the sampled parameter space
\begin{equation}\label{eq_parameter-discretization_loss}
\begin{array}{rl}
\displaystyle
{\mathcal{J}}_\mu(U^{p, \alpha}) \approx  \frac{1}{|P|} \sum_{p \in P} \min_{u \in U^{p, \alpha}}\Big( &\hspace{-2mm} \displaystyle  \,\left\Vert p \left( \Delta \left(\mathbf{a}^{p, \alpha}\cdot\mathbf{w}^{\alpha} + \mathbf{b}^{p, \alpha}\cdot\mathbf{v}^{\alpha}\right) + {\mathbf{c}^{p, \alpha}}\cdot \mathbf{S}^{p} \right) - l^p \right\Vert^2_{L^2(\Omega)} + \\
&\hspace{-2mm}\Theta\left\Vert [p \nabla (\mathbf{a}^{p, \alpha}\cdot \mathbf{w}^{p} + \mathbf{b}^{p, \alpha}\cdot \mathbf{v}^{p}) \cdot \mathbf{n}] \right\Vert_{L^2(\Gamma)}^2 
\Big).
\end{array}
\end{equation}

\subsubsection*{Integral Approximation over the computational domain}
The $L^2$-norms in the loss function \eqref{eq_parameter-discretization_loss} are approximated using Monte Carlo over the domain $\Omega$ and the interface $\Gamma$. For a fixed network parameter $\alpha$, the discrete loss is expressed in matrix form as 
\begin{equation}\label{eq2_loss_D}
    {\mathcal{J}}_D(U^{p, \alpha}) =  \frac{1}{|P|} \sum_{p \in P}  \min_{u \in U^{p, \alpha}} \left( \left\Vert \mathbf{B}^{p, \alpha} \mathbf{y}^{p, \alpha} - \mathbf{l}^p\right\Vert_2^2 \right).
\end{equation}
where $\mathbf{y}^{p, \alpha} = (\mathbf{a}^{p, \alpha}, \mathbf{b}^{p, \alpha}, \mathbf{c}^{p, \alpha})^T$. The specific construction of matrix $\mathbf{B}^{p, \alpha}$ and vector $\mathbf{l}^p$ are detailed in Appendix \ref{Matrix Definitions}.
\begin{remark}
    Based on the discrete formulations presented in \eqref{eq2_loss_D}, the optimal coefficients $\mathbf{a}^{p, \alpha}$, $\mathbf{b}^{p, \alpha}$, and $\mathbf{c}^{p}$, are obtained by solving the following LS solver
\begin{equation}\label{eq2_optimal_coeff_discrete}
    \mathbf{y}^{p,\alpha} =  \arg\min_{\mathbf{y}} \Vert\mathbf{B}^{p,\alpha}\mathbf{y}-\mathbf{l}^p\Vert_2^2, \qquad \forall p \in P
\end{equation}
We utilize a Cholesky decomposition to solve the resulting normal equations and minimize the residual.
\end{remark}
\noindent Finally, the computation of the loss function is outlined in Algorithm \ref{Algorithm_1}. 

\begin{algorithm}[H]\label{Algorithm_1}
{
\begingroup

\SetAlgoLined
Sample a finite batch of parameter values $P$.\\
Construct $\{\mathbf{B}^{p,\alpha},\mathbf{l}^p\}_{p\in P}$.\\
Compute all vectors of optimal coefficients $\{\mathbf{y}^{p,\alpha}\}_{p\in P}$, by solving the LS problems \eqref{eq2_optimal_coeff_discrete}.\\
Estimate the resulting minimum-residual average using equation \eqref{eq2_loss_D}.
\caption{Implementation for computing the loss function}\label{train_step_guidelines}
\endgroup
}
\end{algorithm}

\begin{remark}
    The matrix $\mathbf{B}^{p,\alpha}$ can be represented as a linear combination of parameter-independent matrices with parameter-dependent coefficients. These components are computed once per epoch using high-precision integration, after which the matrices $\mathbf{B}^{p,\alpha}$ for each batch of parameters are assembled efficiently via parameter-dependent linear combinations. This strategy makes the assembly of the linear systems nearly independent of the number of parameters in each iteration while ensuring that $\mathbf{B}^{p,\alpha}$ is updated consistently for each batch, thus significantly reducing the overall computational cost.
\end{remark}

%%%%%%%%%%%%%%%%%%%%%%%%%%%%%%%%%%%%%%%%%%%%%
%%%%%%%%%%%%%%%%%%%%%%%%%%%%%%%%%%%%%%%%%%%%%
%%%%%%%%%%%%%%%%%%%%%%%%%%%%%%%%%%%%%%%%%%%%%
%%%%%%%%%%%%%%%%%%%%%%%%%%%%%%%%%%%%%%%%%%%%%
%%% Analysis of Computational Cost %%%%%%%%%%
%%%%%%%%%%%%%%%%%%%%%%%%%%%%%%%%%%%%%%%%%%%%%
%%%%%%%%%%%%%%%%%%%%%%%%%%%%%%%%%%%%%%%%%%%%%
%%%%%%%%%%%%%%%%%%%%%%%%%%%%%%%%%%%%%%%%%%%%%
%%%%%%%%%%%%%%%%%%%%%%%%%%%%%%%%%%%%%%%%%%%%%

\section{Analysis of Computational Cost}\label{Analysis of Computational Cost}

\noindent A primary advantage of the LS-ReCoNN framework is its computational efficiency when addressing parametric problems. While the training cost of many deep learning approaches for PDEs scales linearly with the number of parameter instances, our methodology utilizes a separated representation of the solution for the principal part to ensure that the computational cost associated with the increase in the number of parameters becomes negligible.

\noindent To quantify this efficiency, we denote the number of trainable NN parameters by $N_\alpha$, the number of parameter instances by $N_p$, the number of basis functions (related to principle and singular parts) in the LS system by $N_{LS}$, the number of basis functions in the 1D FE eigenvalue solver by $N_{FE}$, and the total number of integration points by $N_x$. The LS-ReCoNN solution is constructed as a combination of basis functions—approximated by the NN and the 1D FE eigenvalue solver—and parameter-dependent coefficients determined by an LS solver. This structure allows the computational cost of loss evaluation per iteration to be partitioned into parameter-independent and parameter-dependent components as follows
\begin{equation}\label{eq2_cost}
    \text{Cost} \approx \overbrace{\underbrace{C_1 (N_\alpha \cdot N_x)}}^{\substack{\text{Evaluating the NN}\\ \text{ and its gradients}}}_{\text{parameter-independent}} + \underbrace{\overbrace{C_2 (N_p \cdot N_{LS}^3)}^{\substack{\text{Solving the LS system}}} + \overbrace{C_3 (N_p \cdot N_{FE}^3\cdot N_s)}^{\substack{\text{Solving the FE }\\ \text{eigenvalue problem}}}}_{\text{parameter-dependent}}
\end{equation}
% \begin{equation}\label{eq2_cost}
% \text{Cost} \approx \underbrace{\overbrace{C_1 (N_\alpha \cdot N_x)}^{\substack{\text{Evaluating the NN} \ \text{and gradients}}} + \overbrace{C_2 (N_{x} \cdot N_x)}^{\text{Constructing the LS sub-matrices}}}_{\text{parameter-independent}} + \underbrace{\overbrace{C_3 (N_p \cdot N_{LS}^3)}^{\substack{\text{Solving the} \ \text{LS system}}} + \overbrace{C_4 (N_p \cdot N_{FE}^3 \cdot N_s)}^{\substack{\text{Solving the FE} \ \text{eigenvalue solver}}}}_{\text{parameter-dependent}}
% \end{equation}
where $C_1, C_2, C_3,$ and $C_4$ are constants. In this formulation, the parameter-independent part is performed at the $N_x$ quadrature points only once per epoch, independently of the number of parameter samples $N_p$ within the batch. Consequently, the resulting values are reused across the entire parameter space. In contrast, for parameter-dependent terms, $N_{LS}$ is kept to a reasonable size, and the 1D FE eigenvalue solver requires a very small number of degrees of freedom. Given that $N_\alpha$ and $N_x$ are typically large, the total computational effort for any reasonable $N_p$ is heavily dominated by parameter-independent NN evaluations. This hierarchy ensures that the overall execution time remains nearly invariant with the number of parameter instances, providing a significant advantage over traditional methods.

%%%%%%%%%%%%%%%%%%%%%%%%%%%%%%%%%%%%%%%%%%%%%
%%%%%%%%%%%%%%%%%%%%%%%%%%%%%%%%%%%%%%%%%%%%%
%%%%%%%%%%%%%%%%%%%%%%%%%%%%%%%%%%%%%%%%%%%%%
%%%%%%%%%%%%%%%%%%%%%%%%%%%%%%%%%%%%%%%%%%%%%
%%%%%%%%%%% Numerical Examples %%%%%%%%%%%%%%
%%%%%%%%%%%%%%%%%%%%%%%%%%%%%%%%%%%%%%%%%%%%%
%%%%%%%%%%%%%%%%%%%%%%%%%%%%%%%%%%%%%%%%%%%%%
%%%%%%%%%%%%%%%%%%%%%%%%%%%%%%%%%%%%%%%%%%%%%
%%%%%%%%%%%%%%%%%%%%%%%%%%%%%%%%%%%%%%%%%%%%%

\section{Numerical Examples}\label{Sec_Numerical Examples}
\subsection{One-Dimensional Problems}\label{One-Dimensional Problems}
Let $\Omega = (0, \pi)$. We consider a one-dimensional parametric transmission problem governed by equation \eqref{eq2_strong}, with a source term given by $l^p (x) = 25 \sin(5x)$. The domain is partitioned into $I=5$ disjoint subdomains, as illustrated in Figure \ref{fig_geometry_1D}. Each subdomain is defined by
$\Omega_i = ((i - 1)\pi/5, i\pi/5)$ for $ i = 1, \ldots, 5.$
The parameter $p:\Omega\to\mathbb{R}$ is assumed to be piecewise constant, taking values in the interval $(0.01, 50)$ on each subdomain, i.e.,
\begin{equation}\label{eq_sigma}
p(x)=p_i \in (0.01, 50), \qquad \forall x\in \Omega_i, \qquad i = 1, \ldots, 5.
\end{equation}
The exact solution of the problem is given by
$ \mathfrak{u}^p(x)= (\sin{5x})/{p_i}$ for $ x\in \Omega_i$ and $i = 1, \ldots, 5$. 
We employ an NN with 
three hidden layers, each consisting of 10 neurons, 
followed by an output layer with 50 neurons ($N_1 =10$ and  $N_2=40$). 
A discretized loss function \eqref{eq2_loss_D} with $\Theta = 1$ and 100 collocation points was used in each subdomain to train $\mathbf{u}^{p, \alpha}$. Furthermore, the validation loss is evaluated using 103 quadrature points in each subdomain. In each training epoch, a total of \( 500 \) samples are generated per parameter. All parameter values are drawn independently from a uniform distribution over the interval \( (0.01, 50) \). The validation set is constructed using the same sampling approach, comprising $500$ samples for each parameter, which are fixed before the start of training.  The network is trained for $10^3$ iterations using the Adam optimizer, with a learning rate decreased linearly from $10^{-2}$ to $10^{-4}$.

\noindent Figure \ref{fig_1D_Loss} shows the evolution of the loss for the LS-ReCoNN method on both the training and validation datasets throughout the training process. We see a strong alignment between the training and validation losses, indicating that no overfitting has occurred. 

\begin{figure}[H]
    \centering
    \vspace{-0.5cm}
        \includegraphics[width=0.65\textwidth]{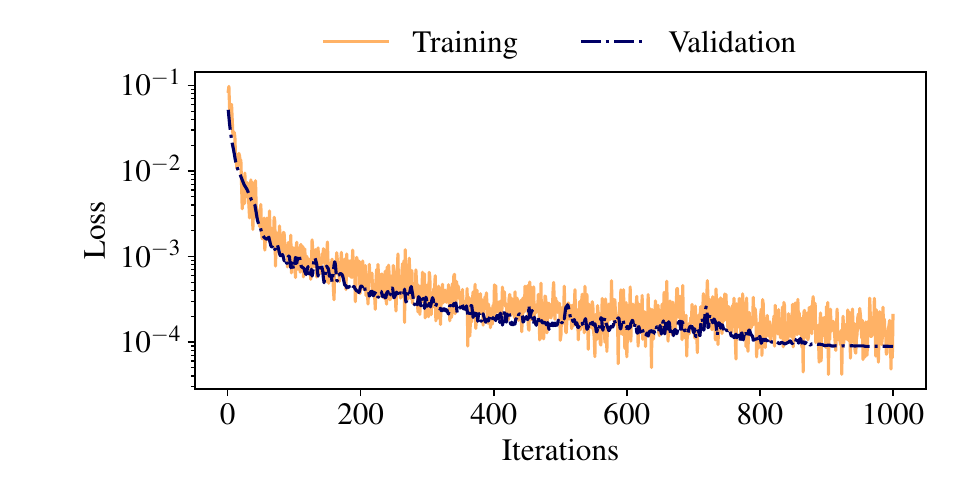} 
        \vspace{-0.3cm}
        \caption{Loss evolution of the LS-Net method on the training and validation datasets for the one-dimensional transmission problem  described in Subsection \ref{One-Dimensional Problems}.}\label{fig_1D_Loss}
    
\end{figure}
\noindent  Furthermore, Figure~\ref{fig_Distribution_1D} shows the relative \( L^2 \)-errors (in \%) of the approximate solutions and their first derivatives, evaluated over \( 10^3 \) randomly sampled parameters, by solving the corresponding least-squares problems both before and after training the NN. As illustrated, training substantially reduces the relative $L^2$-errors. In particular, the average relative errors of the approximate solutions and their fluxes decrease significantly from approximately the range \( (10^1\%, 10^3\%) \) to \( (10^{-5}\%, 10^{-3}\%) \) and \( (5\%, 10^3\%) \) to \( (5\times10^{-4}\%, 10^{-3}\%) \), respectively.

\begin{figure}[H]
    \centering
        \includegraphics[width=0.75\textwidth]{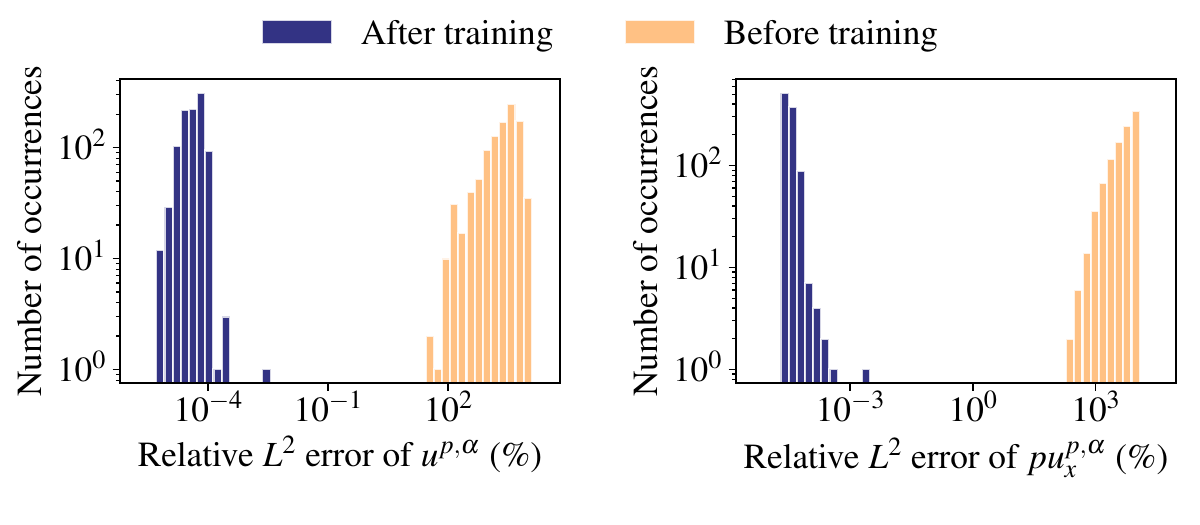} 
        \vspace{-0.5cm}
        \caption{Distribution of the relative $L^2$-errors (in $\%$) for the LS-ReCoNN solution and its flux of the one-dimensional transmission problem.  The results are obtained by solving the corresponding LS problem before and after training the NN.}\label{fig_Distribution_1D}
    
\end{figure}

\paragraph{Comparison to single-instance problem:} 
To evaluate the efficacy of our parametric formulation, we compare its performance against a non-parametric model, which we interpret within our framework by defining the parameter measure as a Dirac delta distribution, $\mu = \delta(p - p_0)$. In this single-instance case, the training is conducted with a batch size of one for a fixed parameter set $p_0 = (1, 4, 0.2, 30, 49)$. To ensure a fair comparison, both models are trained for exactly 1,000 epochs. As shown in Figure \ref{fig_sample_1D} for the one-dimensional transmission problem, the parametric approach is more accurate. Specifically, the absolute solution error and the flux error for the parametric case are on the order of $10^{-6}$, whereas the non-parametric errors are ten times higher, at $10^{-5}$. Beyond the improvement in accuracy, the most significant advantage of the parametric approach is its efficiency. In this method, the NN is trained only once. After the training phase, we can find solutions for any new parameter values almost instantly by using a simple LS solver. In contrast, the non-parametric approach requires the model to be retrained from scratch for every new problem or parameter change, which is computationally expensive and slow.

% \noindent Figures ~\ref {fig_sample_1D_1} and ~\ref {fig_sample_1D_2} present the exact solutions and the corresponding LS-ReCoNN approximations for two parameter sets, \( p = (1, 0.1, 3, 0.3, 0.5) \) and \( p = (10, 45, 2, 30, 15) \), respectively. The results indicate that the method accurately approximates the solutions and fluxes. 

\begin{figure}[H]
    \centering

        \includegraphics[width=0.9\textwidth]{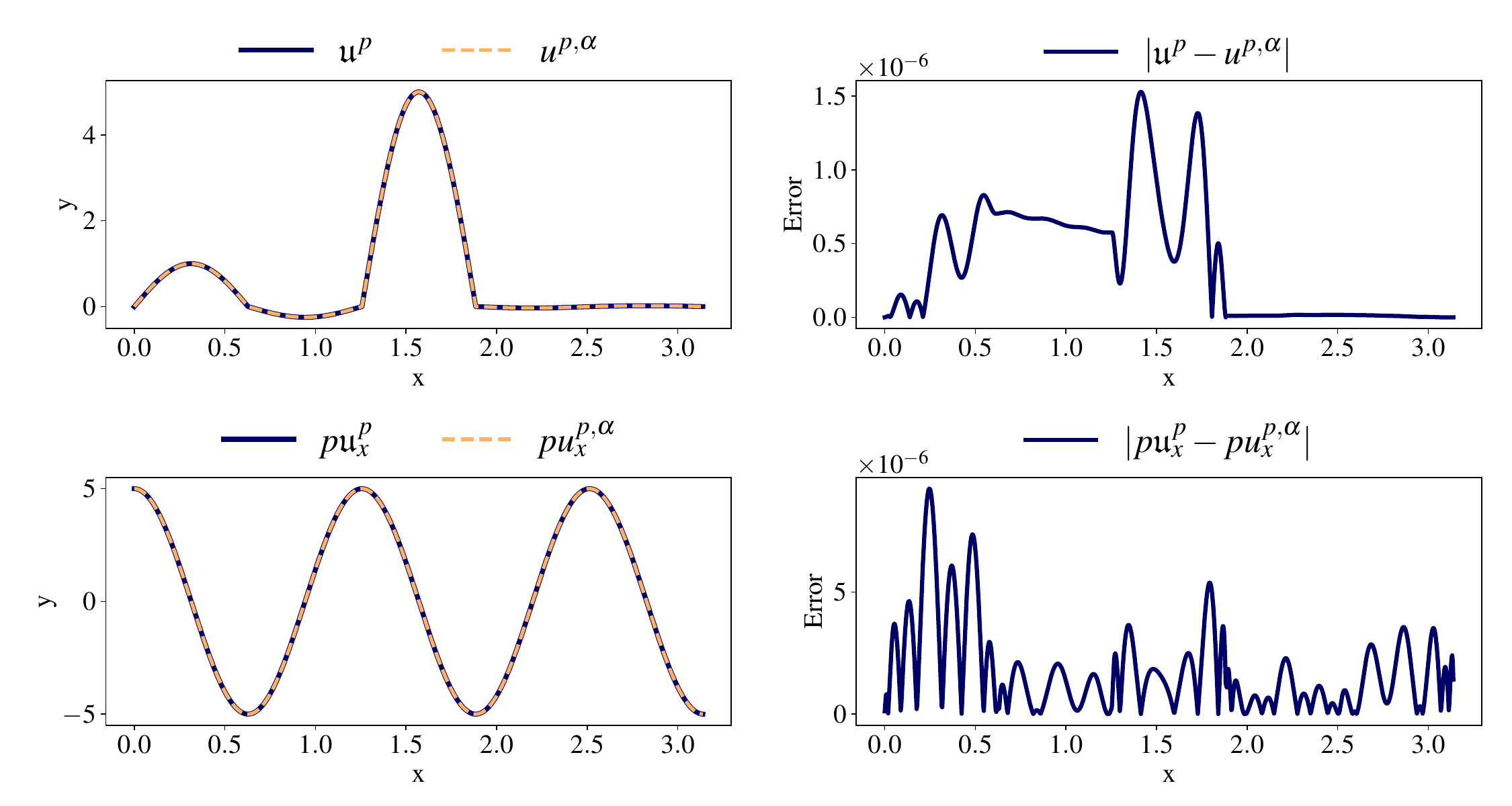}
        \label{fig:high_freq}

\caption{Comparison of the exact and the LS-ReCoNN solutions for the 1D transmission problem with $p_0 = (1, 4, 0.2, 30, 49)$ after 1000 epochs. These results are obtained by the parametric approach that uses a pre-trained network and an LS solver.}
    \label{fig_sample_1D}
\end{figure}
\begin{figure}[H]
    \centering  
        \includegraphics[width=0.9\textwidth]{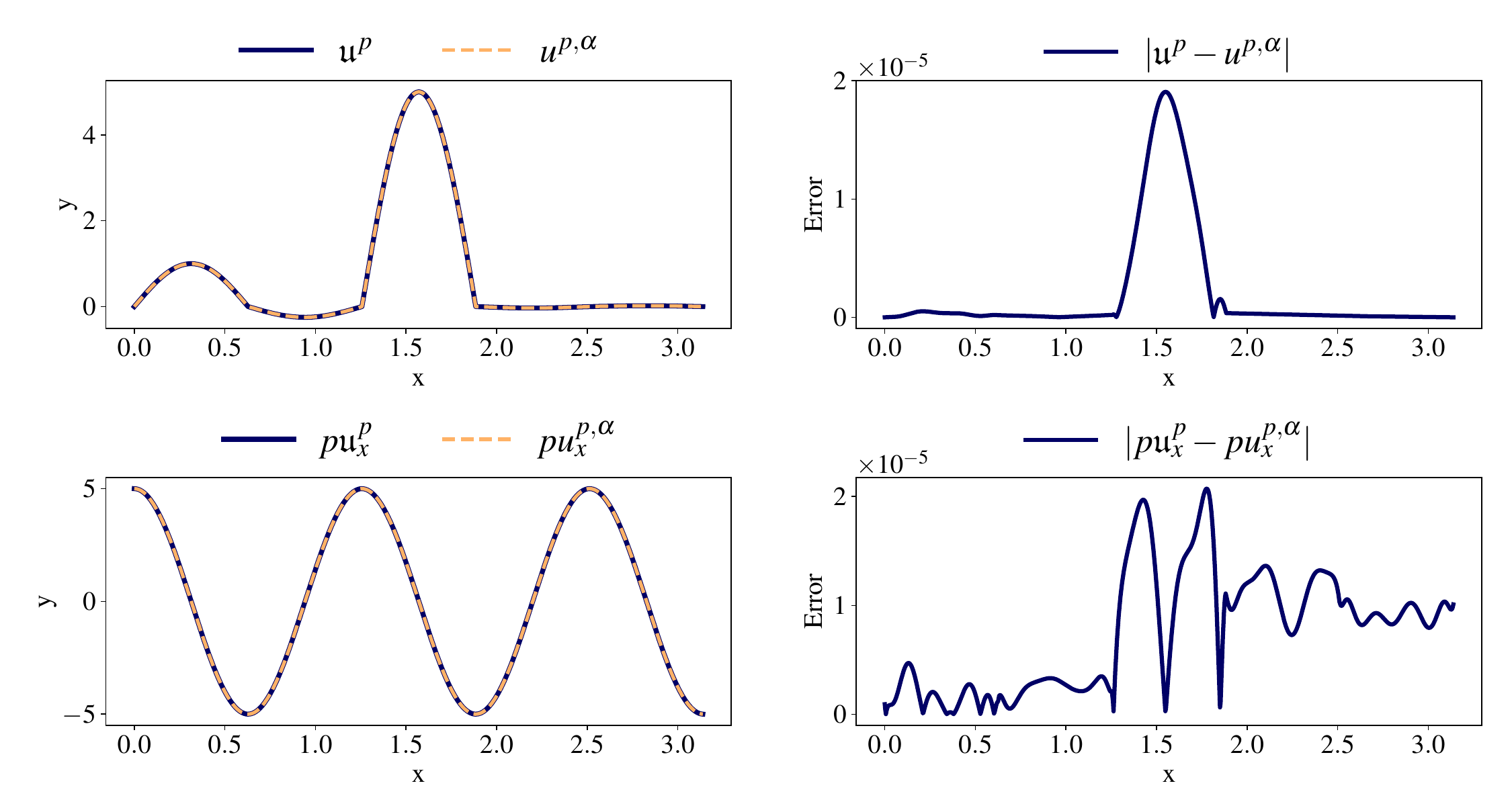}
        \label{fig:low_freq}
    \caption{Comparison of the exact and the LS-ReCoNN solutions for the 1D transmission problem with $p_0 = (1, 4, 0.2, 30, 49)$ after 1000 epochs. These results are obtained by the non-parametric approach that requires full retraining.}
    \label{fig_sample_1D}
\end{figure}

% \begin{figure}[H]
%     \centering
%     \begin{subfigure}[b]{1\textwidth}
%         \centering
%         \includegraphics[width=0.9\textwidth]{u_subplot_1D_2.pdf}
%         \caption{Parametric training. }
%         \label{fig:high_freq}
%     \end{subfigure}
%     \vspace{1cm}
%     \begin{subfigure}[b]{1\textwidth}
%         \centering
        
%         \includegraphics[width=0.9\textwidth]{u_subplot_1D_non_parametric.pdf}
%         \caption{Non-parametric training}
%         \label{fig:low_freq}
%     \end{subfigure}
%     \vspace{-1cm}
%     \caption{Comparison of exact and LS-ReCoNN solutions for the 1D transmission problem with $p_0 = (1, 4, 0.2, 30, 49)$ after 1000 epochs. The parametric approach (a) uses a pre-trained network and an LS solver, whereas the non-parametric approach (b) requires full retraining.}
%     \label{fig_sample_1D}
% \end{figure}

\subsection{Two-Dimensional Problems}\label{Two-Dimensional Problems}
\subsubsection{$2 \times 2$ Material Configuration}\label{2 by 2 Material Configuration}
Consider the two-dimensional parametric transmission problem described by equation \eqref{eq2_strong}, defined on the square domain \( \Omega = (-1, 1)^2 \). The domain \( \Omega \) is partitioned into four non-overlapping square subdomains \( \Omega_i \) for \( i = 1, \ldots, 4 \), defined as follows
\begin{align}
    \Omega_1 = (-1, 0) \times (-1, 0), \quad
    \Omega_2 = (0, 1) \times (-1, 0),\quad
    \Omega_3 = (-1, 0) \times (0, 1), \quad
    \Omega_4 = (0, 1) \times (0, 1).  
\end{align}
The geometric configuration of the domain and subdomains is illustrated in Figure~\ref{fig_2D_geometry_Example}. 

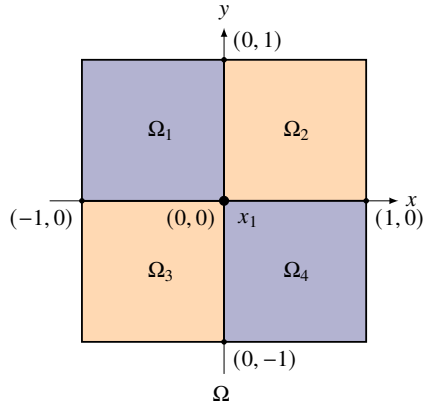
\begin{figure}[H]
    \centering
    \scalebox{0.85}{\definecolor{My_blue}{RGB}{0,0,102}
\definecolor{My_Orange}{RGB}{255, 178, 102}
\begin{tikzpicture}[
punkt/.style={
   rectangle,
   rounded corners,
   draw=black, thick,
   text width=2em,
   minimum height=1em,
   text centered}
]
    % Draw the square domain Omega
    \filldraw[fill=My_Orange!50!white, thick] (-2.2,-2.2) rectangle (0,0);
    %\draw[dashed] (-2.1,-2.1) rectangle (-0.1,-0.1);

    \filldraw[fill=My_blue!30!white, thick] (-2.2,0) rectangle (0,2.2);
    %\draw[dashed] (-2.1,0.1) rectangle (-0.1,2.1);

    \filldraw[fill=My_Orange!50!white, thick] (0,0) rectangle (2.2,2.2);
    %\draw[dashed] (0.1,0.1) rectangle (2.1,2.1);

    \filldraw[fill=My_blue!30!white, thick] (0, -2.2) rectangle (2.2, 0);
    %\draw[dashed] (0.1, -2.1) rectangle (2.1, 0.1);
    
    % Labels for the axes
    \node at (2.7, 0) [right] {$x$};
    \node at (0, 2.7) [above] {$y$};

    % Axes
    \draw[->] (-2.7,0) -- (2.7,0) node[anchor=north west] {};
    \draw[->] (0,-2.7) -- (0,2.7) node[anchor=south east] {};

    \filldraw (-2.2, 0) circle (1pt) node[anchor=north east] {$(-1, 0)$};
    \filldraw (2.2, 0) circle (1pt) node[anchor=north west] {$(1, 0)$};
    \filldraw (0, -2.2) circle (1pt) node[anchor=north west] {$(0, -1)$};
    \filldraw (0, 2.2) circle (1pt) node[anchor=south west] {$(0, 1)$};

    \filldraw (0, 0) circle (2pt) node[anchor=north east] {$(0, 0)$};

\node at (0.8, 1.1) [right] {$\Omega_2$};
\node at (-1.3, 1.1) [right] {$\Omega_1$};
\node at (0.8, -1.1) [right] {$\Omega_4$};
\node at (-1.3, -1.1) [right] {$\Omega_3$};
\node at (-0.3, -3) [right] {$\Omega$};

\node at (0.1, -0.3) [right] {$x_1$};

% \node at (0, 1.1) [right] {$\gamma_{3,4}$};
% \node at (0, -1.1) [right] {$\gamma_{1,2}$};
% \node at (0.8, 0.2) [right] {$\gamma_{2,4}$};
% \node at (-1.3, 0.2) [right] {$\gamma_{1,3}$};

% % Add two arrows above gamma_{2,3} at x = -1
% \draw[->, thick] (0.05, -1.6) -- (0.7, -1.6);
% \draw[->, thick] (-0.05, -1.6) -- (-0.7, -1.6);

% \node[above] at (0.8, -2.2) {\(\mathbf{n}_{12}^+\)};
% \node[above] at (-0.7, -2.2) {\(\mathbf{n}_{12}^-\)};

\end{tikzpicture}}
    \caption{Geometry of the two-dimensional domain \( \Omega \) with a $2 \times 2$ material configuration.}
    \label{fig_2D_geometry_Example}
\end{figure}

\noindent
The parameter function \( p: \Omega \to \mathbb{R} \) is defined as a piecewise-constant function, taking distinct constant values on each subdomain \( \Omega_i \), i.e., $p(x) = p_i$ for $x \in \Omega_i$ and $i = 1, \ldots, 4$, where \( p_i \in [1, 10] \). Furthermore, the right-hand side function is defined as
\begin{equation}\label{eq2_rhs_2D}
    l^p(x, y) = p \cdot \sqrt{r} \left( \cos\left(\frac{\theta}{2}\right) -3 \sin\left(\frac{\theta}{2}\right)  \right), \quad (x, y) \in \Omega,
\end{equation}
where \( \theta = \arctan(y/x) \) and \( r = \sqrt{x^2 + y^2} \) denote the angular and radial components in polar coordinates, respectively. The right-hand side in~\eqref{eq2_rhs_2D}, motivated by the construction in~\cite{brenner2008mathematical}, is specifically chosen to enhance the manifestation of the solution’s singular behavior.

\noindent The LS-ReCoNN solution \( \mathbf{u}^{p, \alpha} \) is constructed according to equation~\eqref{eq2_LS-ReCoNN_solution_2D}, using an NN with four layers, each consisting of 75 neurons with $N_1 = 25$ and $N_2 = 50$ in the last layer. The training process minimizes the discretized loss function \eqref{eq2_loss_D} with $\Theta = 200$. During each training epoch, the model is evaluated on 200 randomly generated parameter samples, where each parameter \( p_i \) is independently selected from the interval \( [1, 10] \), as detailed in Appendix \ref{Parameter Sampling Details}. We employ a stochastic sampling rule to generate collocation points.  Specifically, in each training epoch, we draw a batch consisting of $40 \times 40$ points inside the domain, together with 40 points along each interface.  Validation collocation points are generated using the same strategy, with three additional points added in each axis. Validation points are fixed during training. The network is trained over 50,000 iterations using the Adam optimizer, with a learning rate that decays linearly from \( 10^{-2} \) to \( 10^{-5} \).

{\noindent For the FE eigenvalue solver, the angular domain $(0, 2\pi)$ is partitioned into four equal elements. The approximation space is constructed using piecewise polynomials of degree up to four, which are supplemented by global hat functions spanning neighboring elements to enforce $C^{0}$ continuity across the interfaces. This hierarchical construction yields a total of sixteen basis functions per singular point. Numerical integration of the stiffness and mass matrices is performed using a 5-point Gauss-Legendre quadrature rule, ensuring exact integration for the resulting polynomial products. A comprehensive description of the basis function construction and the implementation of the eigenvalue algorithm is provided in Appendix D.}

\noindent We employ the FEM to construct accurate reference solutions for selected parameter configurations used in the comparison study. The FEM solver discretizes the domain using a uniform mesh consisting of \(100 \times 100\) square elements. In each element, the solution is approximated using linear basis functions.

\noindent Figure~\ref{fig_2D_Loss_parametric} illustrates the evolution of the loss values for the LS-ReCoNN method on both the training and validation datasets throughout the training process. The close alignment of the two curves indicates good generalization performance and suggests that the model does not suffer from overfitting.

\begin{figure}[H]
    \centering
        \includegraphics[width=0.6\textwidth]{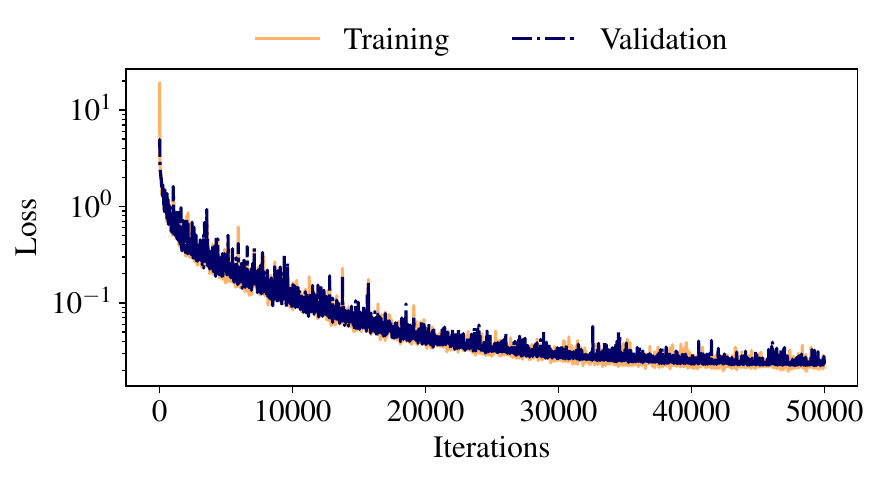} 
        \caption{Loss evolution of the LS-Net method on the training and validation datasets for the two-dimensional parametric transmission problem described in Subsection \ref{2 by 2 Material Configuration}.}\label{fig_2D_Loss_parametric}
    
\end{figure}

\noindent Figure~\ref{fig_Distribution_2D} illustrates the distribution of relative \( L^2 \)-errors (in \%) for 200 randomly sampled parametric problems, both in the solutions and their flux, obtained by solving the corresponding LS problems before and after training the network. Since the LS-ReCoNN is discretization-invariant, we utilized this feature to generate more accurate results. Specifically, we solved the final LS system using $100 \times 100$ collocation points inside the domain and  100 points along interfaces. For this evaluation, we used a midpoint quadrature rule—dividing the domain into uniform squares and evaluating the solution at the center of each. As evident from the figure, training significantly reduces the relative errors. In particular, the average relative \( L^2 \)-error of the approximate solutions decreases significantly, from approximately \( 100\% \) to a range of \( (0.1\%, 1\%) \). Also, the error associated with the gradient is reduced from the interval \((60\%, 100\%) \) to approximately \((1\%, 4\%) \). These results demonstrate the effectiveness of the LS-ReCoNN framework in learning accurate solution representations and their gradients. 

\begin{figure}[H]
    \centering
        \includegraphics[width=0.75\textwidth]{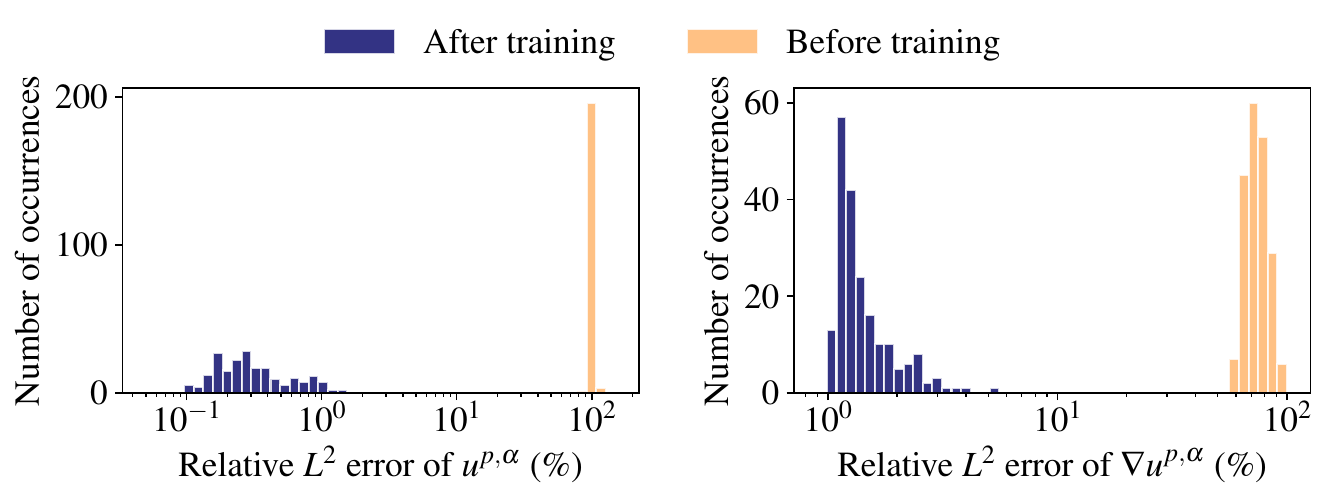} 
        \caption{Distribution of the relative $L^2$-errors (in $\%$) for the LS-ReCoNN solution and its first derivative of the two-dimensional transmission problem described in Subsection \ref{2 by 2 Material Configuration}.  The results are obtained by solving the corresponding LS problem before and after training the NN.}\label{fig_Distribution_2D}
    
\end{figure}

\noindent To verify that our numerical framework is accurate, we compared the FE eigenvalue solver's results against semi-analytic solutions based on the method in \cite{taylor2024regularity}. Figure \ref{fig_Distribution_eigens} presents the relative $L^2$-errors (in \%) for both the eigenvalues and eigenfunctions. The negligible error margins reported in this figure confirm that the FE discretization is highly reliable and provides a precise basis for the singular functions.

% \begin{figure}[H]
%         \centering
%         \includegraphics[width=\textwidth]{u_subplot_1D_1.pdf}
%     \caption{The exact and LS-ReCoNN solutions, along with their fluxes and the corresponding absolute errors, are presented for the one-dimensional transmission problem corresponding parameter set $p = (1, 0.1, 3, 0.3, 0.5)$.}
%     \label{fig_sample_1D_1}
% \end{figure}
% \begin{figure}[H]
%     \vspace{1cm}
%         \vspace{1cm}
%         \includegraphics[width=\textwidth]{u_subplot_1D_2.pdf}
%     \caption{The exact and LS-ReCoNN solutions, along with their fluxes and the corresponding absolute errors, are presented for the one-dimensional transmission problem \ref{One-Dimensional Problems} corresponding to parameter set $p = (10, 45, 2, 30, 15)$.}
%     \label{fig_sample_1D_2}
% \end{figure}

%%%%%%%%%%%%%%%%%%%%%%%%%%%%%%%%%%%%%%%%%%%%%%%%%%%%%%%%%%%%%%
%%%%%%%%%%%%%%%%%%%%%%%%%%%%%%%%%%%%%%%%%%%%%%%%%%%%%%%%%%%%%%
%%%%%%%%%%%%%%%%%%%%%%%%%%%%%%%%%%%%%%%%%%%%%%%%%%%%%%%%%%%%

\noindent Figure~\ref{fig_2D_solution_Loss_parametric} compares the FEM and LS-ReCoNN solutions for the parameter configuration \( p_1 = 2 \), \( p_2 = 10 \), \( p_3 = 7 \), and \( p_4 = 1 \). The corresponding absolute error is also shown, demonstrating strong agreement between the two approaches. \noindent Figure~\ref{fig_2D_gradient_Loss_parametric} shows the fluxes computed from both the FEM and LS-ReCoNN solutions, including their vector fields and the corresponding absolute error. These results further demonstrate the accuracy and consistency of the LS-ReCoNN approximation.  Figure~\ref{fig_2D_Error_Loss_parametric} shows a zoomed-in view of the absolute errors in the fluxes of the LS-ReCoNN and FEM solutions near the origin, allowing for a closer examination of the solution behavior at the singularity point. To better interpret the results, we emphasize that the FEM is inherently incapable of accurately capturing the singular behavior at the origin. Therefore, the reference FEM solution exhibits unbounded error near the singularity located at the origin. However, the FEM solution still represents the best possible approximation in the energy norm within the selected discretized space, even though it does not accurately resolve the singularity. 
In contrast, other standard NN-based approaches typically attempt to approximate the singular structure directly. Such attempts in existing literature can introduce instability near singular points and often lead to oscillatory behavior or the Gibbs phenomenon. This contrast underscores a fundamental distinction between traditional FEM and typical NN-based methods found in other studies when dealing with singularities.
%In contrast, the NN-based approaches explicitly attempt to approximate the singular structure. Such attempts can introduce instability near singular points and often lead to oscillatory behavior or the Gibbs phenomenon. This contrast underscores a fundamental distinction between traditional FEM and NN-based methods in the presence of singularities.

\begin{figure}[H]
    \centering
        \includegraphics[width=0.75\textwidth]{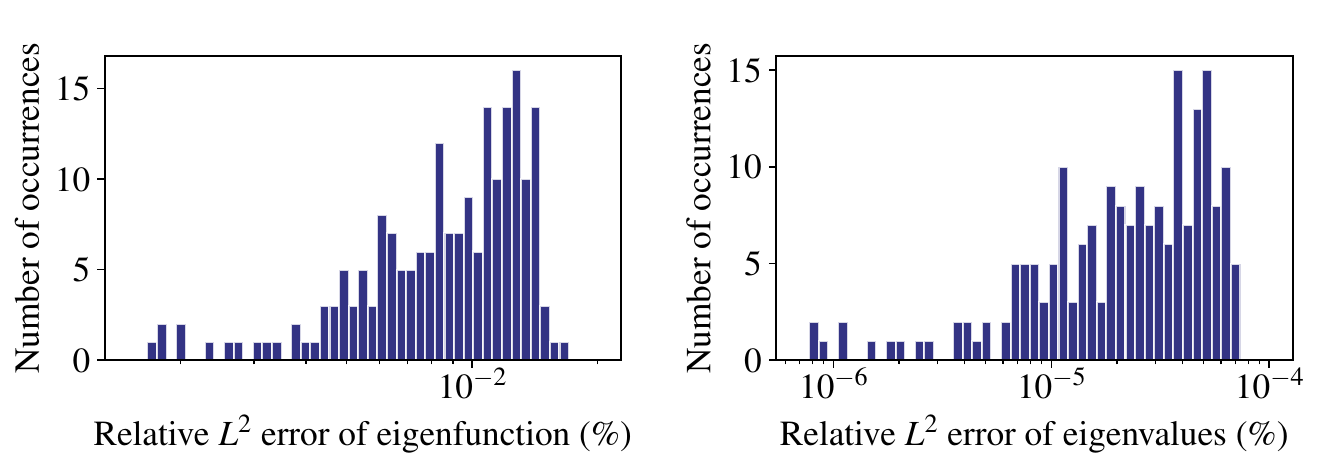} 
        \caption{Relative $L^2$-errors (in \%) of eigenfunctions and eigenvalues of the two-dimensional transmission problem described in Subsection \ref{2 by 2 Material Configuration}. Numerical results were validated against semi-analytic solutions.  }\label{fig_Distribution_eigens}
    
\end{figure}

%%%%%%%%%%%%%%%%%%%%%%%%%%%%%%%%%%%%%%%%%%%%%
%%%%%%%%%%%%%%%%%%%%%%%%%%%%%%%%%%%%%%%%%%%%%
%%%%%%%%%%%%%%%%%%%%%%%%%%%%%%%%%%%%%%%%%%%%%
%%%%%%%%%%%%%%%%%%%%%%%%%%%%%%%%%%%%%%%%%%%%%
%%%%% Square Domain with $5 \times 5$ %%%%%%%
%%%%%%%%%%%%%%%%%%%%%%%%%%%%%%%%%%%%%%%%%%%%%
%%%%%%%%%%%%%%%%%%%%%%%%%%%%%%%%%%%%%%%%%%%%%
%%%%%%%%%%%%%%%%%%%%%%%%%%%%%%%%%%%%%%%%%%%%%
%%%%%%%%%%%%%%%%%%%%%%%%%%%%%%%%%%%%%%%%%%%%%

\subsubsection{$4 \times 4$ Material Configuration}\label{4 by 4 Material Configuration}

Consider the two-dimensional parametric transmission problem introduced in Subsection~\ref{2 by 2 Material Configuration}. Let the domain \( \Omega \) be partitioned into 16 non-overlapping square subdomains \( \Omega_i \), arranged in a uniform \( 4 \times 4 \) grid, where \( i = 1, \ldots, 16 \). The geometric configuration of this partition is depicted in Figure~\ref{fig_2D_geometry_Example_5x5}.

\noindent
The parameter function \( p: \Omega \to \mathbb{R} \) is defined as a piecewise constant function, taking distinct values \( p_i \in [1, 10] \) on each subdomain \( \Omega_i \), i.e., $p(x) = p_i$ for  $x \in \Omega_i$ and $i = 1, \ldots, 16$, where $p_i \in [1, 10]$. The right-hand side 
\begin{figure}[H]
    \centering
    \begin{subfigure}[b]{0.85\textwidth}
        \centering
        \includegraphics[width=\textwidth]{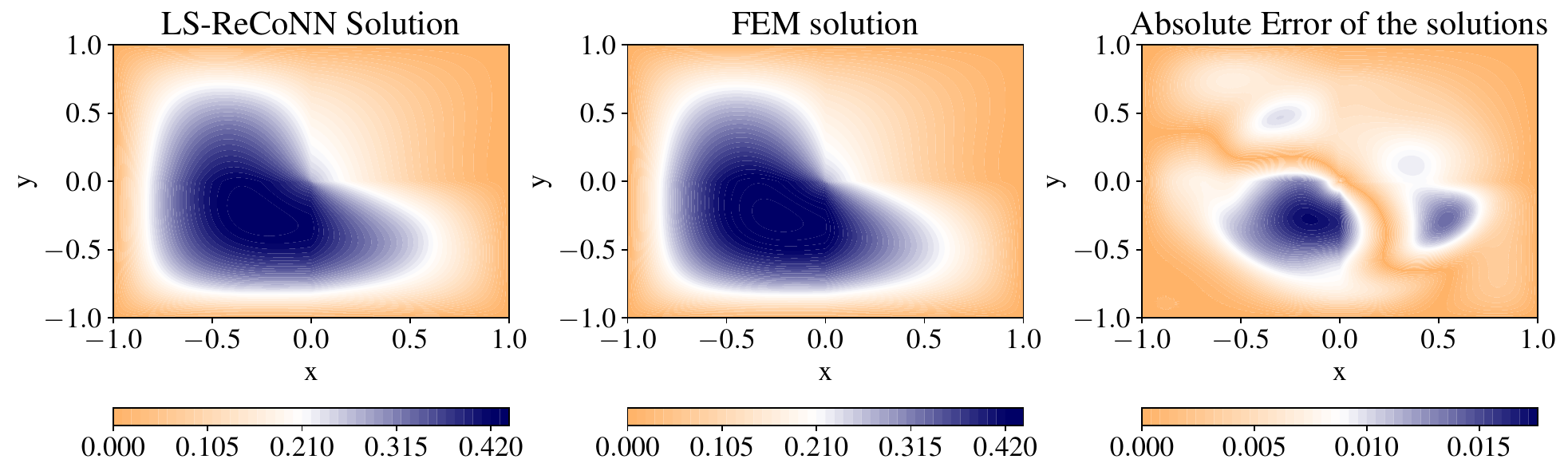}
        \caption{Comparison of LS-ReCoNN (left) and FEM (middle) solutions and their absolute error (right).}
        \label{fig_2D_solution_Loss_parametric}
    \end{subfigure}
    
    \vspace{1cm}

    \begin{subfigure}[b]{0.85\textwidth}
        \centering
        \includegraphics[width=\textwidth]{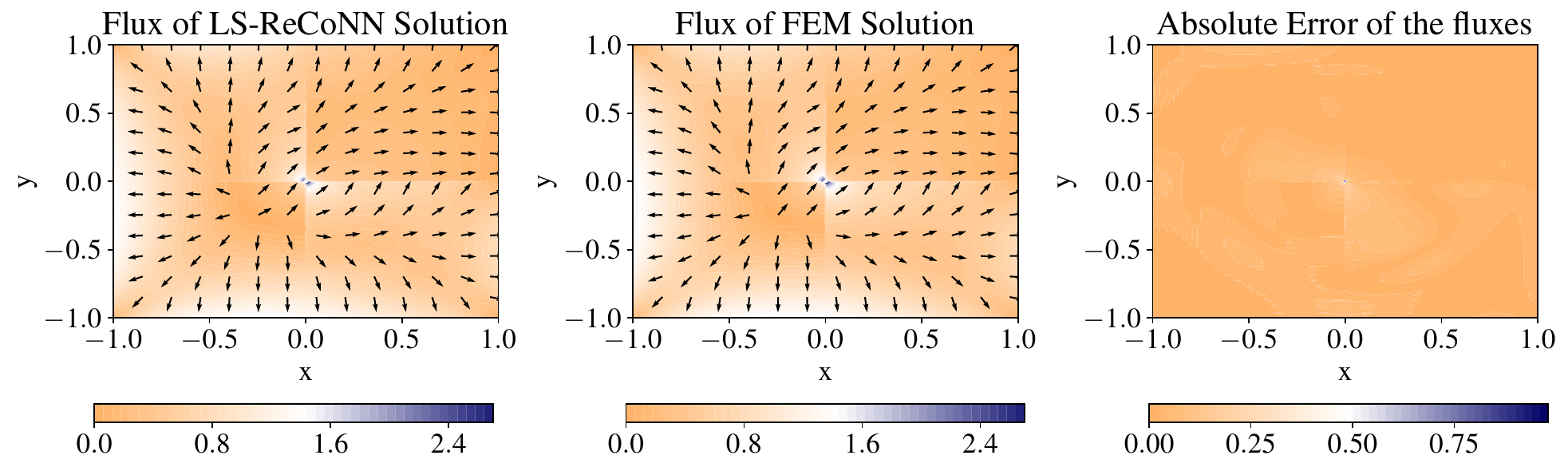}
        \caption{Comparison of the LS-ReCoNN (left) and FEM (middle) fluxes and their absolute error (right). Note that, for clarity of presentation, the flux arrows are normalized to unit length.}
        \label{fig_2D_gradient_Loss_parametric}
    \end{subfigure}
    
    \vspace{1cm}
    
    \begin{subfigure}[b]{0.85\textwidth}
        \centering
        \includegraphics[width=\textwidth]{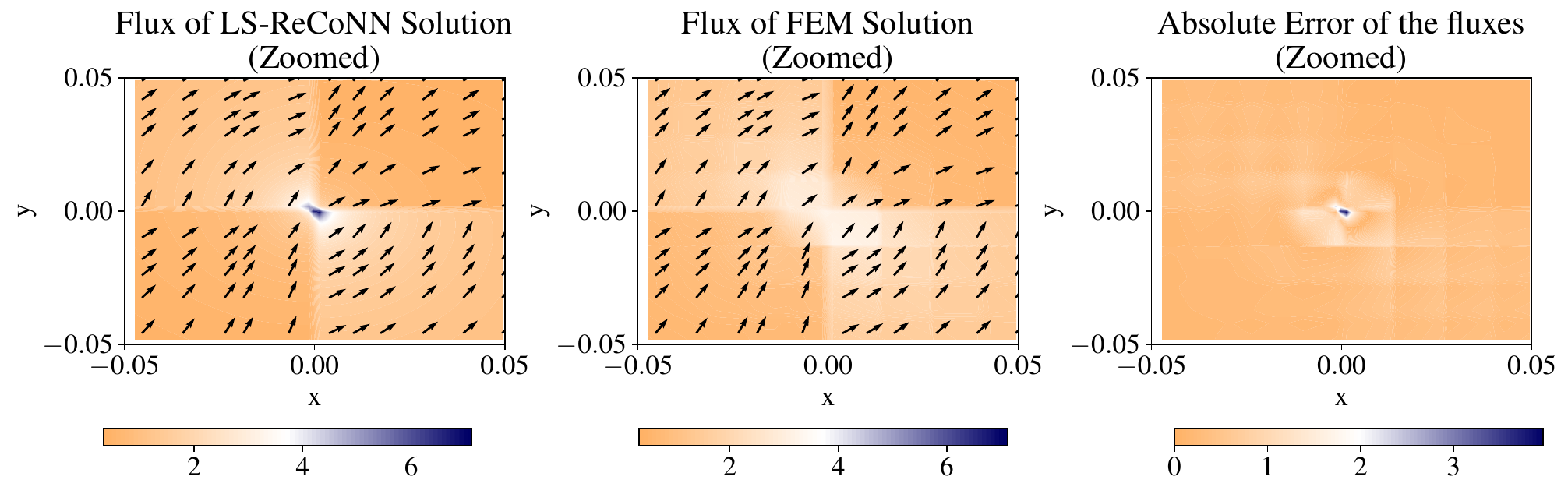}
        \caption{Comparison of LS-ReCoNN (left) and FEM (middle) solutions, and their absolute error (right) near the singularity point with zoomed-in view. Note that, for clarity of presentation, the flux arrows are normalized to unit length.}
        \label{fig_2D_Error_Loss_parametric}
    \end{subfigure}
    
    \caption{Comparison of FEM and LS-ReCoNN solutions for the two-dimensional parametric transmission problem described in Subsection \ref{2 by 2 Material Configuration} after 50,000 training iterations. We show the approximate solutions, their flux magnitudes and vector fields, and the corresponding absolute errors.}
    \label{fig:combined}
\end{figure}
\noindent function is given as
\begin{equation}\label{eq2_rhs_2D_3by3}
    l^p(x, y) = \sum_{i = 1}^9 p \cdot \sqrt{r_i} \left(\cos\left(\frac{\theta_i}{2}\right) - 3 \sin\left(\frac{\theta_i}{2}\right) \right), \quad (x, y) \in \Omega,
\end{equation}
where \( \theta_i \) and \( r_i \) denote the angular and radial components centered at the singular point $x_i$, respectively. 

\noindent The approximation \( \mathbf{u}^{p,\alpha} \) produced by LS-ReCoNN follows the representation in~\eqref{eq2_LS-ReCoNN_solution_2D}. 
The NN architecture employed in this experiment consists of five layers. The hidden layers contain 180 neurons each, while the fifth layer comprises 270 neurons with $N_1 = 90$ and $N_2 = 180$. Training is carried out by minimizing the discretized loss functional~\eqref{eq2_loss_D} with $\Theta = 10$.  At each epoch, 50 parameter samples are drawn independently from \([1,10]\), as detailed in Appendix \ref{Parameter Sampling Details}, and the model is updated based on them.  Collocation points are chosen using a first-order
\begin{figure}[H]
    \centering
    \scalebox{1}{\definecolor{My_blue}{RGB}{0,0,102}
\definecolor{My_Orange}{RGB}{255, 178, 102}

\begin{tikzpicture}[
    scale=2.5,
    every node/.style={font=\small},
    punkt/.style={
       rectangle,
       rounded corners,
       draw=black, thick,
       text width=2em,
       minimum height=1em,
       text centered}
]

% Draw the 4x4 grid and label each subdomain
\newcounter{domain}
\setcounter{domain}{1}

\foreach \j in {3,...,0} {
    \foreach \i in {0,...,3} {
        \pgfmathsetmacro{\x}{-1 + 0.5*\i}
        \pgfmathsetmacro{\y}{-1 + 0.5*\j}
        \pgfmathsetmacro{\xc}{\x + 0.25}
        \pgfmathsetmacro{\yc}{\y + 0.25}
        \pgfmathparse{mod(\i+\j,2)==0 ? "My_Orange!50!white" : "My_blue!30!white"}
        \edef\fillcolor{\pgfmathresult}
        \filldraw[fill=\fillcolor, thick] (\x,\y) rectangle ({\x+0.5}, {\y+0.5});
        \node at (\xc, \yc) {$\Omega_{\arabic{domain}}$};
        \stepcounter{domain}
    }
}

% Label the internal grid points (3x3 = 9 points)

\setcounter{point}{1}

\foreach \j in {3,...,1} {
    \foreach \i in {1,...,3} {
        \pgfmathsetmacro{\x}{-1 + 0.5*\i}
        \pgfmathsetmacro{\y}{-1 + 0.5*\j}
        
        % Draw a small dot at the grid point
        \filldraw (\x, \y) circle (0.5pt);
        
        % Label shifted by (0.1, -0.1)
        \node at ({\x + 0.12}, {\y - 0.1}) {$x_{\arabic{point}}$};
        
        \stepcounter{point}
    }
}

% Corners of the domain
\filldraw (-1, -1) circle (0.3pt) node[anchor=north east] {$(-1, -1)$};
\filldraw (1, 1) circle (0.3pt) node[anchor=south west] {$(1, 1)$};
\filldraw (1, -1) circle (0.3pt) node[anchor=north west] {$(1, -1)$};
\filldraw (-1, 1) circle (0.3pt) node[anchor=south east] {$(-1, 1)$};

% Label domain
\node at (0, -1.3) {$\Omega$};

\end{tikzpicture}}
    \caption{Geometry of the two-dimensional domain \( \Omega \) with a $4 \times 4$ material configuration for the transmission problem.}
    \label{fig_2D_geometry_Example_5x5}
\end{figure}
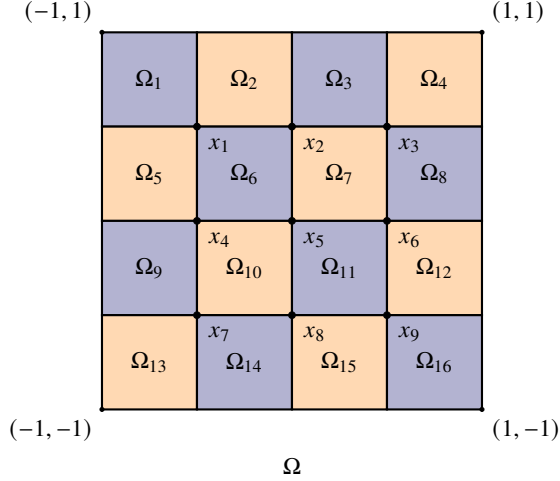
\noindent  stochastic rule: each batch consists of \(80\times80\) points placed in the interior of the computational domain, complemented by $80$ additional points along every interface. 
For validation, collocation points are generated in the same manner, except that three extra points are added along each axis; these validation sets remain fixed during training.  The optimization is performed with the Adam algorithm over $2\times10^4$ iterations, employing a linearly decaying learning rate that starts at \(10^{-1}\) and decreases to \(5\times10^{-3}\).

{\noindent For the computation of the singular basis functions at the nine internal singular points, the FE eigenvalue solver is implemented as described in Subsection \ref{2 by 2 Material Configuration}. Consistent with the previous example, we utilize the hierarchical construction of sixteen basis functions per singular point, the details of which are provided in Appendix D.}

\noindent Figure~\ref{fig_2D_Loss_parametric_4by4} illustrates the evolution of the loss values for the LS-ReCoNN method on both the training and validation datasets throughout the training process. The close alignment of the two curves indicates good generalization performance and suggests that the model does not suffer from overfitting.

\begin{figure}[H]
    \centering
        \includegraphics[width=0.6\textwidth]{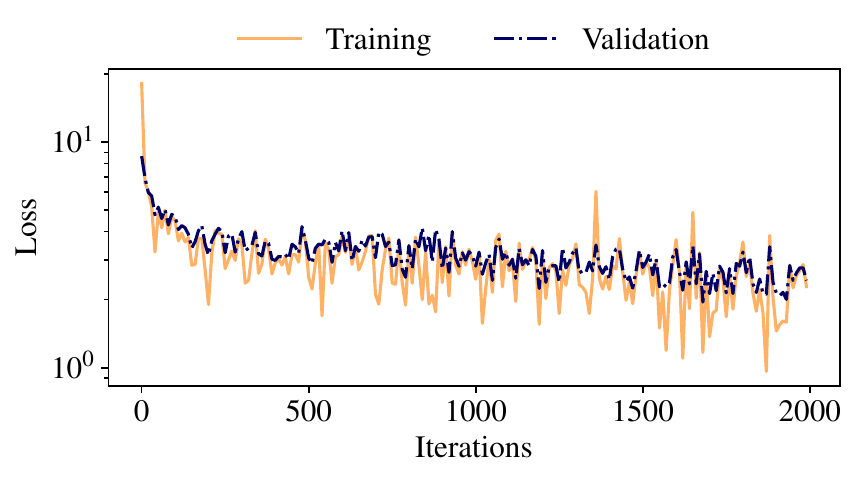} 
        \caption{Loss evolution of the LS-Net method on the training and validation datasets for the two-dimensional parametric transmission problem described in Subsection \ref{4 by 4 Material Configuration}.}\label{fig_2D_Loss_parametric_4by4}
    
\end{figure}

\noindent Figure~\ref{fig_Distribution_2D_4by4} presents the distribution of relative \( L^2 \)-errors (in \%) computed for 200 randomly chosen parametric test cases. As in the previous example, we solved the final system using $200 \times 200$ collocation points inside the domain and 200 collocation points along each interface with a midpoint quadrature rule. The comparison includes both the solution fields and their fluxes, obtained by solving the corresponding LS problems before and after training. As in the previous example, a FEM is employed to generate reliable reference solutions for selected parameter instances used in the evaluation. The FEM model discretizes the computational domain with a uniform grid of \(200 \times 200\) square elements, and linear shape functions are used within each element to approximate the solution. Based on Figure \ref{fig_Distribution_2D_4by4}, it is clear that the training step substantially lowers the error levels. More precisely, the mean relative \( L^2 \)-error of the approximated solutions is reduced from nearly \(10^4\%\) down to nearly \(10\%\). Similarly, the gradient error decreases markedly, moving from \(10^5\%\) before training to approximately \(10\%\) afterward. These outcomes confirm that the LS-ReCoNN approach is highly effective in capturing both the solution and its derivative information with high accuracy.

\begin{figure}[H]
    \centering
        \includegraphics[width=0.75\textwidth]{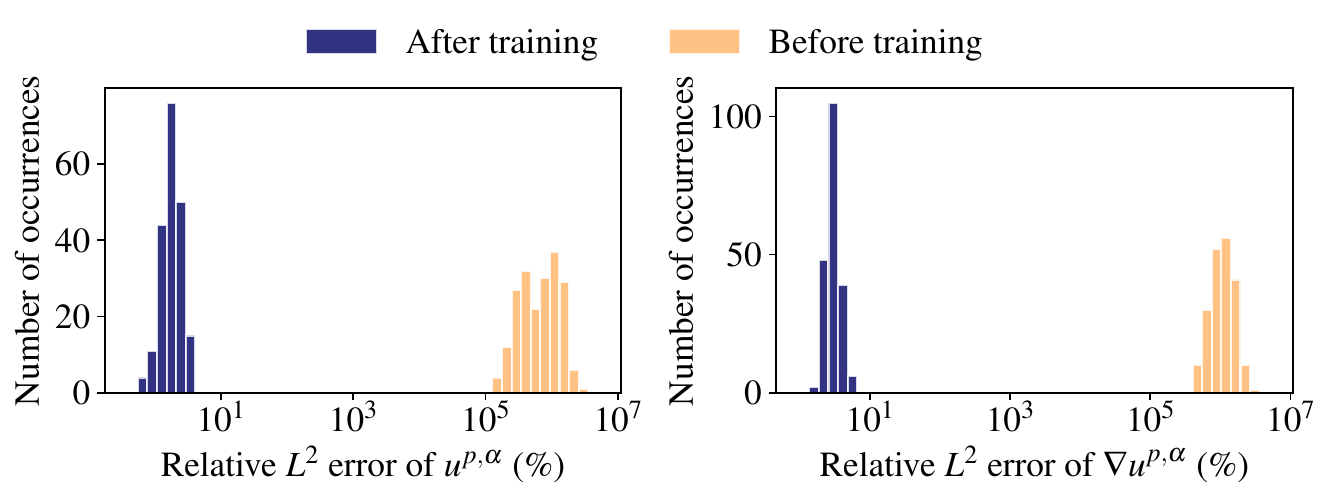} 
        \caption{Distribution of the relative $L^2$-errors (in $\%$) for the LS-ReCoNN solution and its flux of the two-dimensional transmission problem described in Subsection \ref{4 by 4 Material Configuration}.  The results are obtained by solving the corresponding LS problem before and after training the NN.}\label{fig_Distribution_2D_4by4}
    
\end{figure}

\noindent Figure~\ref{fig_2D_solution_Loss_parametric_4by4} compares the FEM and LS-ReCoNN solutions for a specific set of parameter 
\begin{equation}
    p = (1.5,\, 9,\, 2,\, 3,\, 4,\, 7,\, 2,\, 4,\, 1,\, 5,\, 8,\, 1,\, 9,\, 3,\, 5,\, 2).
\end{equation}
The corresponding absolute error is also displayed, again showing close agreement between the two methods.  Moreover, Figure~\ref{fig_2D_gradient_Loss_parametric_4by4} depicts the fluxes obtained from both approaches together with their vector fields and the absolute error. As before, the results confirm the accuracy of the LS-ReCoNN approximation.

%%%%%%%%%%%%%%%%%%%%%%%%%%%%%%%%%%%%%%%%%%%%%
%%%%%%%%%%%%%%%%%%%%%%%%%%%%%%%%%%%%%%%%%%%%%
%%%%%%%%%%%%%%%%%%%%%%%%%%%%%%%%%%%%%%%%%%%%%
%%%%%%%%%%%%%%%%%%%%%%%%%%%%%%%%%%%%%%%%%%%%%
%%%%%%%%%%%%%%% Conclusion %%%%%%%%%%%%%%%%%%
%%%%%%%%%%%%%%%%%%%%%%%%%%%%%%%%%%%%%%%%%%%%%
%%%%%%%%%%%%%%%%%%%%%%%%%%%%%%%%%%%%%%%%%%%%%
%%%%%%%%%%%%%%%%%%%%%%%%%%%%%%%%%%%%%%%%%%%%%
%%%%%%%%%%%%%%%%%%%%%%%%%%%%%%%%%%%%%%%%%%%%%

\section{Conclusion}\label{Conclusion}
In this work, we presented the LS-ReCoNN, a novel hybrid framework that integrates an eigenvalue solver, a Regularity-Conforming Neural Network (ReCONN) architecture, and a Least-Squares Neural Network (LS-Net) strategy to solve parametric transmission problems exhibiting singularities. By decomposing the solution into a principal component approximated by a deep NN and a singular component derived from an FE eigenvalue solver, the method effectively addresses the challenges of localized singularities and parametric complexity.

\noindent The LS-ReCoNN framework offers several key advantages over existing classical and deep learning methods. In terms of accuracy and stability, the method avoids Gibbs-type instabilities and achieves higher precision in regions of reduced regularity by embedding known regularity information and utilizing an FE solver for singular components. Furthermore, the proposed method provides significant computational efficiency through its separated representation of the principal component of the solution, which enables the precomputation of parameter-independent matrices. This structural advantage ensures that the cost of handling a large number of parameters $N_p$ is dominated by NN evaluations, rendering the total execution time nearly invariant to the number of parameter instances. We also proposed a loss function that provides an explicit upper bound for the energy-norm error, ensuring that the training process minimizes a physically meaningful quantity. Numerical experiments in both one- and two-dimensional settings demonstrate the accuracy and robustness of the proposed method. Indeed, the results show that the LS-ReCoNN successfully captures complex solution behaviors, such as gradient jumps and singular structures, which are typically difficult for standard NN methods. Notably, we found that training the model as a parametric problem leads to much better convergence than training a non-parametric version for a single case. In the parametric problem, the network navigates a smoother optimization and avoids the local minima that often trap non-parametric models. This results in a more robust and reliable approximation of both the solution and the flux across various parameter ranges.

\noindent This work opens several promising directions for future research. For example, extending the LS-ReCoNN to 3D problems. However, this requires improving the modeling of singular functions for complex geometries and reducing the error of eigenfunction evaluation, which are key areas for exploration. Furthermore, applying the LS-ReCoNN framework to nonlinear or time-dependent transmission problems could further broaden its applicability in the scientific and engineering aspects.

% \begin{figure}[H]
%     \centering
%         \includegraphics[width=0.75\textwidth]{Error_Eigens_4by4.pdf} 
%         \caption{Relative $L^2$-errors (in \%) of eigenfunctions and eigenvalues of the two-dimensional transmission problem described in Subsection \ref{4 by 4 Material Configuration}. Numerical results were validated against semi-analytic solutions.  }\label{fig_Distribution_eigens_4by4}
    
% \end{figure}

\begin{figure}[H]
    \centering
    \begin{subfigure}[b]{0.85\textwidth}
        \centering
        \includegraphics[width=\textwidth]{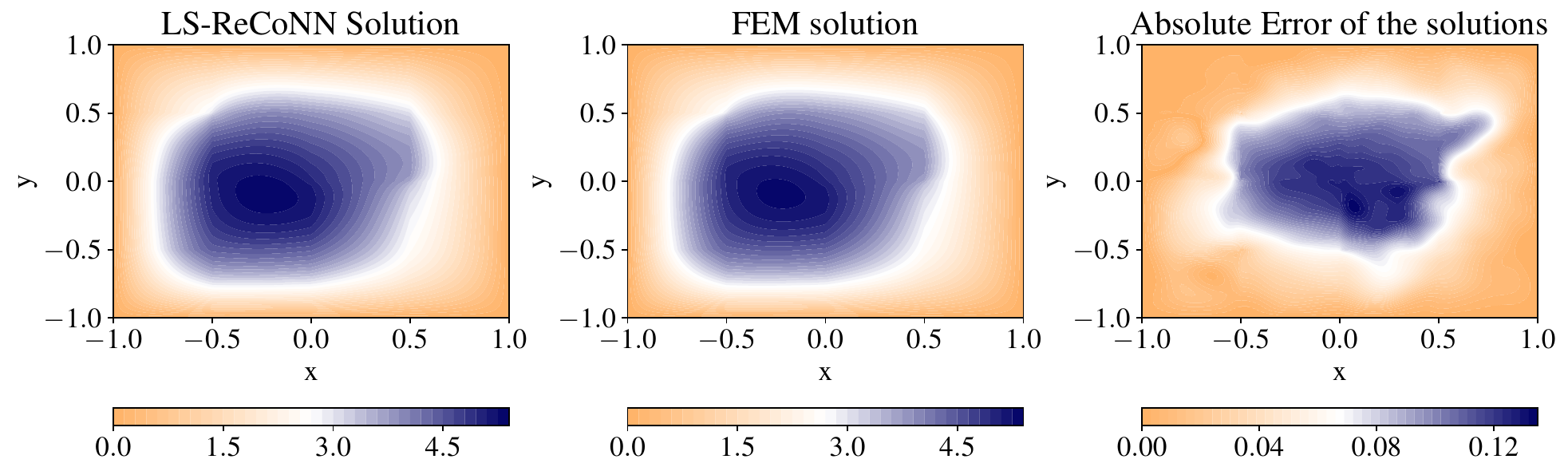}
        \caption{Comparison of LS-ReCoNN (left) and FEM (middle) solutions and their absolute error (right).}
        \label{fig_2D_solution_Loss_parametric_4by4}
    \end{subfigure}
    \vspace{1cm}

    \begin{subfigure}[b]{0.85\textwidth}
        \centering
        \includegraphics[width=\textwidth]{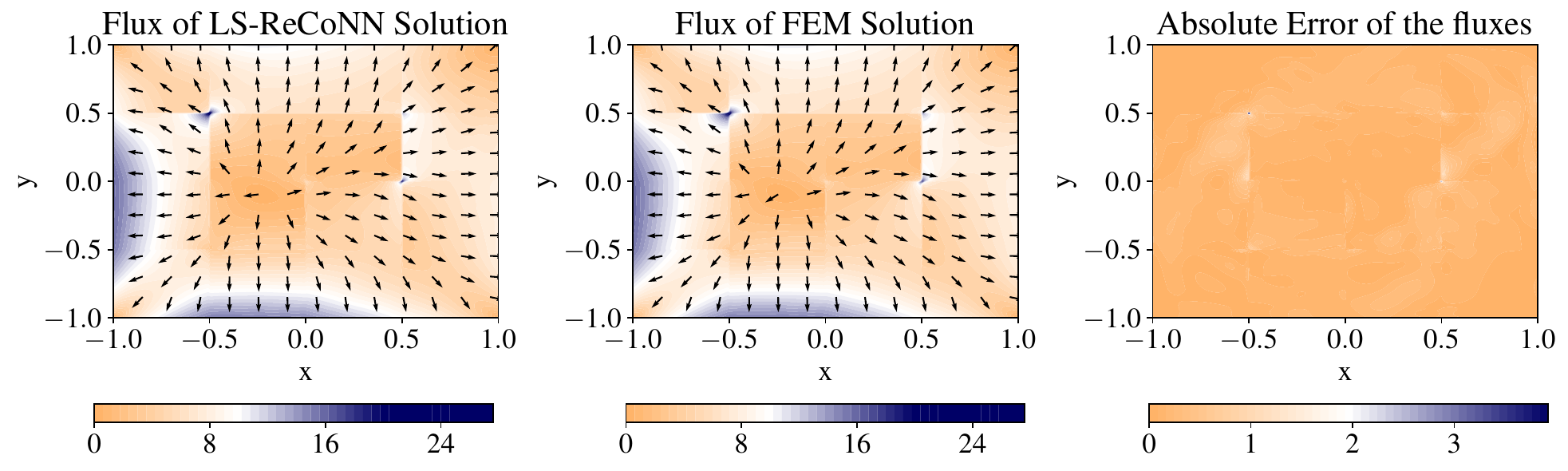}
        \caption{Comparison of the LS-ReCoNN (left) and FEM (middle) fluxes and their absolute error (right). Note that, for clarity of presentation, the flux arrows are normalized to unit length.}
        \label{fig_2D_gradient_Loss_parametric_4by4}
    \end{subfigure}
    
    \caption{Comparison of FEM and LS-ReCoNN solutions for the two-dimensional parametric transmission problem described in Subsection \ref{4 by 4 Material Configuration} after 10,000 training iterations. We show the approximate solutions, their flux magnitudes and vector fields, and the corresponding absolute errors for the parameter set $p = (1.5,\, 9,\, 2,\, 3,\, 4,\, 7,\, 2,\, 4,\, 1,\, 5,\, 8,\, 1,\, 9,\, 3,\, 5,\, 2)$.}
    \label{fig:combined_4by4}
\end{figure}

%%%%%%%%%%%%%%%%%%%%%%%%%%%%%%%%%%%%%%%%%%%%%
%%%%%%%%%%%%%%%%%%%%%%%%%%%%%%%%%%%%%%%%%%%%%
%%%%%%%%%%%%%%%%%%%%%%%%%%%%%%%%%%%%%%%%%%%%%
%%%%%%%%%%%%%%%%%%%%%%%%%%%%%%%%%%%%%%%%%%%%%
%%%%%%%%%%%%% Acknowledgments %%%%%%%%%%%%%%%
%%%%%%%%%%%%%%%%%%%%%%%%%%%%%%%%%%%%%%%%%%%%%
%%%%%%%%%%%%%%%%%%%%%%%%%%%%%%%%%%%%%%%%%%%%%
%%%%%%%%%%%%%%%%%%%%%%%%%%%%%%%%%%%%%%%%%%%%%
%%%%%%%%%%%%%%%%%%%%%%%%%%%%%%%%%%%%%%%%%%%%%

\section{Acknowledgments}\label{Acknowledgments}
This work has received funding from the following Research Projects/Grants: 
European Union’s Horizon Europe research and innovation programme under the Marie Sklodowska-Curie Action MSCA-DN-101119556 (IN-DEEP). 
PID2023-146678OB-I00 funded by MICIU/AEI /10.13039/501100011033 and by FEDER, EU;
BCAM Severo Ochoa accreditation of excellence CEX2021-001142-S funded by MICIU/AEI/10.13039/501100011033; 
Basque Government through the BERC 2022-2025 program;
BEREZ-IA (KK-2023/00012) and RUL-ET(KK-2024/00086), funded by the Basque Government through ELKARTEK;
Consolidated Research Group MATHMODE (IT1866-26) of the UPV/EHU given by the Department of Education of the Basque Government; 
BCAM-IKUR-UPV/EHU, funded by the Basque Government IKUR Strategy and by the European Union NextGenerationEU/PRTR.

%%%%%%%%%%%%%%%%%%%%%%%%%%%%%%%%%%%%%%%%%%%%%
%%%%%%%%%%%%%%%%%%%%%%%%%%%%%%%%%%%%%%%%%%%%%
%%%%%%%%%%%%%%%%%%%%%%%%%%%%%%%%%%%%%%%%%%%%%
%%%%%%%%%%%%%%%%%%%%%%%%%%%%%%%%%%%%%%%%%%%%%
%%%%%%%%%%%%%%% Appendix A %%%%%%%%%%%%%%%%%%
%%%%%%%%%%%%%%%%%%%%%%%%%%%%%%%%%%%%%%%%%%%%%
%%%%%%%%%%%%%%%%%%%%%%%%%%%%%%%%%%%%%%%%%%%%%
%%%%%%%%%%%%%%%%%%%%%%%%%%%%%%%%%%%%%%%%%%%%%
%%%%%%%%%%%%%%%%%%%%%%%%%%%%%%%%%%%%%%%%%%%%%
\makeatletter
\renewcommand{\@seccntformat}[1]{%
  \ifstrequal{#1}{section}{%
    \appendixname~\csname the#1\endcsname:\quad
  }{%
    \csname the#1\endcsname\quad
  }%
}
\makeatother

\appendix
\section{Proof of $H^1_0$-Control by Residuals}\label{Proof of H-Control by Residuals}

In this section, we provide the proof for Theorem \ref{2_Theorem_1}, establishing that the $H^1_0$-norm is controlled by the norm of the residuals associated with the strong form \eqref{eq2_strong}.

\paragraph{Proof of Theorem \ref{2_Theorem_1}:} Define the energy norm by 
$$
||u||_{H^1_p}^2=\int_\Omega p|\nabla u|^2\,dx. 
$$
We define the $H^1_0$ trace and Poincar\'e constants via 
\begin{equation}
\begin{split}
C_{tr}=& \sup\limits_{v\in H^1_0}\frac{||v||_{L^2(\Gamma)}}{||v||_{H^1_0}}\\
C_P = &\sup\limits_{v\in H^1_0}\frac{||v||_{L^2(\Omega)}}{||v||_{H^1_0}}
\end{split}
\end{equation}
We note that as $\Gamma$ is a finite union of straight lines, $C_{tr}$ is finite. Since $p_{\min}\leq p$ almost everywhere, it holds that for $v\in H^1_0$, $||v||_{H^1_p}\geq \sqrt{p_{\min}}||v||_{H^1_0}$, and thus
\begin{equation}
\begin{split}
\frac{||v||_{L^2(\Gamma)}}{||v||_{H^1_p}}\leq &\frac{C_{tr}}{\sqrt{p_{\min}}},\\
\frac{||v||_{L^2(\Omega)}}{||v||_{H^1_p}}\leq &\frac{C_P}{\sqrt{p_{\min}}},\\
\end{split}
\end{equation}
for all $v\in H^1_0$. Then, we may estimate the $H^1_0$-norm of $u$ via
\begin{equation}
\begin{split}
\sqrt{p_{\min}}||u||_{H^1_0}\leq ||u||_{H^1_p}=&\sup\limits_{v\in H^1_0}\frac{1}{||v||_{H^1_p}}\int_\Omega p\nabla u\cdot \nabla v\,dx\\
=&\sup\limits_{v\in H^1_0}\frac{1}{||v||_{H^1_p}}\left(\int_{\Gamma}[p\nabla u\cdot n]v\,dS -\int_{\Omega\setminus \Gamma} \text{div}(p\nabla u)v\,dx\right)\\
\leq & \sup\limits_{v\in H^1_0}\frac{1}{||v||_{H^1_p}}\left(||[p\nabla u\cdot n]||_{L^2(\Gamma)}||v||_{L^2(\Gamma)}+||\text{div}(p\nabla u)||_{L^2(\Omega)}||v||_{L^2(\Omega)}\right)\\
\leq & \frac{C_{tr}}{\sqrt{p_{\min}}}||[p\nabla u\cdot n]||_{L^2(\Omega)}+\frac{C_P}{\sqrt{p_{\min}}}||\text{div}(p\nabla u)||_{L^2(\Omega)}. 
\end{split}
\end{equation}
We note that the integration by parts step is permitted as $u\in \mathbb{X}^p$. This yields 
\begin{equation}
\begin{split}
||u||_{H^1_0}\leq & \frac{1}{p_{\min}}\left(\frac{C_{tr}\sqrt{\Theta}}{\sqrt{\Theta}}||[p\nabla u \cdot n]||_{L^2(\Gamma)}+C_P||\text{div}(p\nabla u)||_{L^2(\Omega)}\right)\\
\leq & \frac{\max\left(\frac{C_tr}{\sqrt{\Theta}},C_P\right)}{p_{\min}}\left(\sqrt{\Theta}||[p\nabla u \cdot n]||_{L^2(\Gamma)}+||\text{div}(p\nabla u)||_{L^2(\Omega)}\right)\\
\leq & \frac{\sqrt{2}\max\left(\frac{C_tr}{\sqrt{\Theta}},C_P\right)}{p_{\min}}\sqrt{\Theta||[p\nabla u \cdot n]||_{L^2(\Gamma)}^2+||\text{div}(p\nabla u)||_{L^2(\Omega)}^2}
\end{split}
\end{equation}

\section{Solution Decomposition Proof}\label{Solution Decomposition Proof}
\paragraph{Proof of Theorem \ref{thm_decomposition}:}
We first consider the 1D case. Let $\mathfrak{w}^p$ be the unique weak solution to the Poisson equation $\frac{d^2}{dx^2}\mathfrak{w}^p(x)=\frac{\ell^p(x)}{p(x)}$ in $H^1_0(0,1)$. Since the right-hand side belongs to $L^2(0,1)$, standard elliptic regularity implies that $\mathfrak{w}^p\in H^2(\Omega)$. We then define $\mathfrak{v}^p=\mathfrak{u}^p-\mathfrak{w}^p$. It follows that $\mathfrak{v}^p\in H^2((0,1)\setminus\Gamma)\cap H^1_0(0,1)$, as both $\mathfrak{u}^p$ and $\mathfrak{w}^p$ are in this space. The second derivative of $\mathfrak{v}^p$ is therefore well defined almost everywhere, and we observe that 
\begin{equation}
    \frac{d^2}{dx^2}\mathfrak{v}^p=\frac{d^2}{dx^2}\mathfrak{u}^p-\frac{d^2}{dx^2}\mathfrak{w}^p=\frac{\ell^p}{p(x)}-\frac{\ell^p}{p(x)}=0,
\end{equation}
almost everywhere. To obtain its Neumann-type condition, we have that 
\begin{equation}
\begin{split}
[p\frac{d}{dx}\mathfrak{v}^p\cdot\mathbf{n}] = [p\frac{d}{dx}\mathfrak{u}^p\cdot\mathbf{n}]-[p\frac{d}{dx}\mathfrak{w}^p\cdot\mathbf{n}]=-[p]\frac{d}{dx}\mathfrak{w}^p\cdot\mathbf{n},
\end{split}
\end{equation}
invoking the fact that $\mathfrak{u}^p$ satisfies the transmission condition on the interfaces and $w^p\in H^2(\Omega)$, , ensuring its derivative is continuous across $\Gamma$. 

\noindent For the two-dimensional case, the existence of the finite dimensional space of singular functions $\mathbb{U}^p$, which satisfies $\text{div}(p\nabla \mathfrak{s})\in L^2(\Omega)$ is established in \cite{kellogg1971singularities}. Furthermore, there exists $\mathfrak{s}^p\in \mathbb{U}^p$ such that 
\begin{equation}
\mathfrak{u}^p-\mathfrak{s}^p \in H^2(\Omega\setminus\Gamma)\cap H^1_0(\Omega). 
\end{equation}
We introduce the modified source term, $f^p=\frac{1}{p}\text{div}(p(\nabla \mathfrak{u}^p-\nabla \mathfrak{s}^p)) = \frac{\ell^p}{p}-\frac{1}{p}\text{div}(p\nabla \mathfrak{s}^p)\in L^2(\Omega)$. Analogous to the one-dimensional case, we define $\mathfrak{w}^p$ to be the unique weak solution in $H^1_0(\Omega)$ to 
\begin{equation}
   \int_\Omega \nabla \mathfrak{w}^p\cdot\nabla\varphi -f^p\varphi\,dx=0,
\end{equation}
for all $\varphi\in H^1_0(\Omega)$. By elliptic regularity on convex polygonal domains, $\mathfrak{w}^p\in H^2(\Omega)$. Consequently, its Laplacian is well defined almost everywhere and satisfies 
\begin{equation}\label{eqLaplacianWp}
\Delta \mathfrak{w}^p = f^p=\Delta\left(\mathfrak{u}^p-\mathfrak{s}^p\right)
\end{equation}
away from the interfaces. We define $\mathfrak{v}^p=\mathfrak{u}^p-\mathfrak{s}^p-\mathfrak{w}^p$. Since $\mathfrak{u}^p-\mathfrak{s}^p\in H^2(\Omega\setminus\Gamma)\cap H^1_0(\Omega)$ and $\mathfrak{u}^p\in H^2(\Omega)\cap H^1_0(\Omega)$, it is evident that $\mathfrak{v}^p\in H^2(\Omega\setminus\Gamma)\cap H^1_0(\Omega)$. By \eqref{eqLaplacianWp}, the Laplacian of $\mathfrak{v}^p$ vanishes almost everywhere. Regarding the interface condition, we have 
\begin{equation}
[p\nabla \mathfrak{v}^p\cdot \mathbf{n}]=[p\nabla (\mathfrak{u}^p-\mathfrak{s}^p-\mathfrak{w}^p)\cdot \mathbf{n}]=[p\nabla \mathfrak{u}^p\cdot \mathbf{n}]-[p\nabla \mathfrak{s}^p\cdot \mathbf{n}]-[p\nabla \mathfrak{w}^p\cdot \mathbf{n}].
\end{equation}
Since both $\mathfrak{u}^p$ and $\mathfrak{s}^p$ satisfy the homogeneous transmission condition ($[p\nabla \cdot \mathbf{n}] = 0$), and owing to the $H^2(\Omega)$ regularity of $\mathfrak{w}^p$ which implies continuity of its gradient, the expression reduces to
\begin{equation}
[p\nabla \mathfrak{v}^p \cdot \mathbf{n}] = -[p] \nabla \mathfrak{w}^p \cdot \mathbf{n}.
\end{equation}
This completes the proof.

\section{Implementation Details of the Trial Spaces $U_1$, $U_2$, and $U_3^p$}\label{Implementation Details of the Trial Spaces}

\subsection{Neural Network Structure}\label{Neural Network Structure}
Let us consider the following fully-connected NN illustrated in Figure \ref{fig2_NN_architecture_1D}, with depth $D$ and the set of learnable parameters $\alpha$

\begin{equation}\label{eq2_NN_before ReCoNN}
(\bar{\mathbf{w}}^\alpha, \bar{\mathbf{v}}^\alpha) = L_{D} \circ L_{D - 1} \circ \ldots  \circ  L_1 \circ L_0, 
\end{equation}
where $(\bar{\mathbf{w}}^\alpha, \bar{\mathbf{v}}^\alpha) = (\bar{w}^\alpha_1, \ldots \bar{w}^\alpha_{N_1},\bar{v}^\alpha_1, \ldots, \bar{v}^\alpha_{N_2})^T$ and $L_0 = x$ are the output and input vectors of the network, respectively. $L_d$ is the layer function and it is defined as follows
\begin{equation}\label{eq_layers}
L_d =\sigma( A_d L_{d-1} + b_d), \qquad d = 1, \ldots, D, 
\end{equation}
where $\sigma(\cdot)$ is an activation function that acts component-wise on vectors. Here, we utilize a $\tanh$ activation function. The matrices $A_d\in\mathbb R^{M_d} \times \mathbb R^{M_{d-1}}$ and the vectors $b_d\in\mathbb{R}^{M_{d}}$ are the so-called weights and biases, respectively, where their components collectively form the set $\alpha$. In this representation, $M_0$ is the dimension of the problem, $M_{D} = N_1 + N_2$ is the dimension of the output, and $M_{d}$ for $ d =1, \ldots, D-1$ denotes the number of neurons in the $d$-th layer. Note that the success of an NN heavily depends on the architecture and proper selection of hyperparameters (e.g., initialization, learning rate), as different initializations can yield distinct results. However, a detailed analysis of these factors is beyond the scope of this work.

\subsection{Implementation of the Boundary Cutoff}\label{Boundary Cutoff Function}
In the context of using NNs to solve PDEs, Dirichlet boundary conditions can be enforced strongly by incorporating them directly into the network’s output structure. This approach ensures that the resulting solution space inherently satisfies the prescribed boundary conditions. For homogeneous {Dirichlet} boundary conditions, we apply a cut-off function $\mathcal{B}$, where $\mathcal{B}(x) = 0$ on the segments of the boundary for which the solution space includes the boundary conditions \cite{berrone2023enforcing}. Given that any nonhomogeneous PDE problem can be transformed into a homogeneous one, this methodology can be readily adapted to handle inhomogeneous boundary conditions by applying an appropriate lifting. 

\subsection{Implementation of the Gradient Jump Cutoff}\label{Gradient jump cutoff}
The gradient jump cutoff function $\Psi(x) = (\psi_1, \psi_2, \ldots, \psi_{N_2})^T \in \mathbb{R}^{N_2}$ is constructed for one- and two-dimensional problems as follows:
\begin{itemize}
    \item \textbf{One-dimensional problems:} All components of the vector-valued function are set to a common scalar function, $\psi_n = \bar{\psi}$ for $n = 1, \dots, N_2$. In this setting, the $M = I-1$ material interfaces are located at positions $\gamma_{i(i+1)}$. Based on the framework of normalized Approximate Distance Functions (ADFs) \cite{berrone2023enforcing}, the scalar component $\bar{\psi}(x)$ is defined as
    \begin{equation}
    \bar{\psi}(x) = \sqrt{\sum_{i = 1}^M \frac{1}{(x - \gamma_{i(i + 1)})^2}}.
    \end{equation}

    \item \textbf{Two-dimensional problems:} The vector-valued function $\Psi(x)$ is partitioned into two equal groups of functions as follows
    \begin{equation}\label{eq2_Psi_2D}
    \psi_n =
    \begin{cases}
    \bar{\psi}_1, & 1 \le n \le N_2/2, \\[2pt]
    \bar{\psi}_2, & N_2/2 < n \le N_2,
    \end{cases}
    \qquad n = 1,\dots,N_2.
    \end{equation}
    Following the construction of normalized ADFs \cite{berrone2023enforcing}, the scalar functions $\bar{\psi}_1$ and $\bar{\psi}_2$ are given by
    \begin{equation}\label{eq2_psi}
    \bar{\psi}_1(x)
    =
    \left(\sum_{\gamma \in \Gamma_1}\frac{1}{(x-\gamma)^2}\right)^{1/2},
    \qquad
    \bar{\psi}_2(x)
    =
    \left(\sum_{\gamma \in \Gamma_2}\frac{1}{(x-\gamma)^2}\right)^{1/2},
    \end{equation}
    where $\Gamma_1$ and $\Gamma_2$ denote the sets of vertical and horizontal interfaces, respectively,
    and $\Gamma = \Gamma_1 \cup \Gamma_2$.
\end{itemize}

\subsection{Implementation of the Singularity-Exclusion Cutoff}\label{Singularity-Exclusion cutoff}
We define the singularity-exclusion cutoffs $\Phi_1(x)\in \mathbb{R}^{N_1}$ and $\Phi_2(x)\in \mathbb{R}^{N_2}$ for one- and two-dimensional problems as follows:
\begin{itemize}
    \item \textbf{One-dimensional problems:} Since 1D transmission problems do not exhibit singularities, the corresponding exclusion cutoffs are constant. Consequently, for all $x \in \Omega$, we set $\Phi_1(x) = \mathbf{1}_{N_1}$ and $\Phi_2(x) = \mathbf{1}_{N_2}$, where $\mathbf{1}_n$ denotes a vector of ones of dimension $n$.

    \item \textbf{Two-dimensional problems:} In 2D settings, we define the cutoff function $\Phi_1(x)$ for the smooth component as a partitioned vector
    \begin{equation}\label{eq2_Phi_1}
    \Phi_1(x) = \left( \underbrace{\phi_1(x), \dots, \phi_1(x)}_{\frac{N_1}{N_s}}, \underbrace{\phi_2(x), \dots, \phi_2(x)}_{\frac{N_1}{N_s}}, \ldots, \underbrace{\phi_{N_s}(x), \dots, \phi_{N_s}(x)}_{\frac{N_1}{N_s}} \right)^T,
    \end{equation}
    where $N_s$ represents the number of singular points. For each singular point $x_n$, the component function $\phi_n(x)$ is defined using the smooth function $\eta$ as follows
    \begin{equation}\label{eq2_phi_k}
    \phi_n(x) = 1 - \eta(|x - x_n|), \quad n = 1, \dots, N_s.
    \end{equation}
    Similarly, $\Phi_2(x) \in \mathbb{R}^{N_2}$, is constructed to ensure local nullification at each singular points, i.e.,
    \begin{equation}\label{eq2_Phi_2}
    \Phi_2(x) = \begin{pmatrix} \phi(x), \ \phi(x) \end{pmatrix}, \quad \text{with }\quad \phi(x) = \left( \underbrace{\phi_1(x), \dots, \phi_1(x)}_{\frac{N_2}{2N_s}}, \underbrace{\phi_2(x), \dots, \phi_2(x)}_{\frac{N_2}{2N_s}}, \ldots, \underbrace{\phi_{N_s}(x), \dots, \phi_{N_s}(x)}_{\frac{N_2}{2N_s}} \right)^T.
    \end{equation}
\end{itemize}

\subsection{Implementation of the FE Eigenvalue Solver} \label{Implementation of the FE Eigenvalue Solver}
\subsubsection{Construction of Basis Functions}
The angular domain $[0, 2\pi]$ is partitioned into four subdomains, defined by the intervals $[\frac{i\pi}{2}, \frac{(i+1)\pi}{2}]$ for $i = 0, \dots, 3$. For computational convenience, the angular domain is mapped to a reference domain $\xi \in [0, 4]$. Within this framework, we define two distinct sets of basis functions to span the approximation space:

\begin{itemize}
    \item \textbf{Internal Polynomials:} Within each element $e \in \{0, 1, 2, 3\}$, we define higher-order basis functions using fourth-degree polynomials. To ensure these functions vanish at the element nodes (maintaining independence from boundary constraints), we utilize a windowing factor $\xi(1-\xi)$. The internal basis functions are defined as follows
    \begin{itemize}
        \item $\varphi_{3e}(\xi) = \xi(1-\xi)$
        \item $\varphi_{3e+1}(\xi) = 5\xi(1-\xi)(\xi-0.5)$
        \item $\varphi_{3e+2}(\xi) = 20\xi(1-\xi)(\xi-0.5)^2$
    \end{itemize}
    
    \item \textbf{Global Hat Functions:} To enforce $C^0$ continuity between boundary of elements, we include linear hat functions as follows
    \begin{equation}
        \varphi_{e + 12}(\xi) = \frac{1}{4} \max(0, 1 - |\xi - k|), \qquad e = 0, \ldots, 3
    \end{equation}
    Because the domain is periodic, these functions wrap around the boundary.
\end{itemize}

\noindent This construction yields a total of 16 basis functions per singular point.

\subsubsection{Numerical Integration}
The assembly of the matrices $G^p$ and $B^p$ in Equation \eqref{eq2_FE_matrices} involves the integration of polynomial products weighted by the material parameter $p(\theta)$. We employ a 5-point Gauss-Legendre quadrature rule on each element. Since the basis functions are polynomials of degree 4, the products $\varphi_i \varphi_j$ and $\frac{d\varphi_i}{d\theta} \frac{d\varphi_j}{d\theta}$ result in polynomials of at most degree 8. A 5-point Gaussian rule integrates polynomials up to degree $9$ exactly; thus, the quadrature introduces no numerical integration error for the chosen space.

\subsubsection{Algorithm}

The generalized eigenvalue problem \eqref{eq2_FE_eigenvalue_solver} is solved by transforming it into a standard symmetric eigenvalue problem and then using the \texttt{tensorflow.linalg} suite. This procedure is summarized in Algorithm \ref{alg:eigenvalue_solver}.

\begin{algorithm}[H]
\SetAlgoLined
\DontPrintSemicolon 
Assemble the stiffness matrix $G^p$ and mass matrix $B^p$ using Eq. \eqref{eq2_FE_matrices}.\\
Compute the inverse square root of the mass matrix, $(B^p)^{-1/2}$.\\%, using \texttt{tf.linalg.sqrtm} and \texttt{tf.linalg.inv}.\\
Construct the symmetric matrix $M = (B^p)^{-1/2} G^p (B^p)^{-1/2}$.\\
Compute the eigenvalues and eigenvectors of $M$.\\% using the \texttt{tf.linalg.eigh} algorithm.\\
Recover the angular eigenvectors $\rho_j$ of the generalized problem via the transformation $\vartheta = (B^p)^{-1/2} y_j$ where $y_j$ is the eigenvalue of $M$.\\
\caption{Implementation for solving the FE eigenvalue problem}\label{alg:eigenvalue_solver}
\end{algorithm}

%%%%%%%%%%%%%%%%%%%%%%%%%%%%%%%%%%%%%%%%%%%%%
%%%%%%%%%%%%%%%%%%%%%%%%%%%%%%%%%%%%%%%%%%%%%
%%%%%%%%%%%%%%%%%%%%%%%%%%%%%%%%%%%%%%%%%%%%%
%%%%%%%%%%%%%%%%%%%%%%%%%%%%%%%%%%%%%%%%%%%%%
%%%%%%%%%%%%%%% Appendix D %%%%%%%%%%%%%%%%%%
%%%%%%%%%%%%%%%%%%%%%%%%%%%%%%%%%%%%%%%%%%%%%
%%%%%%%%%%%%%%%%%%%%%%%%%%%%%%%%%%%%%%%%%%%%%
%%%%%%%%%%%%%%%%%%%%%%%%%%%%%%%%%%%%%%%%%%%%%
%%%%%%%%%%%%%%%%%%%%%%%%%%%%%%%%%%%%%%%%%%%%%

\section{Strong Formulation of the Decomposed Solution}\label{Strong Formulation of the Decomposed Solution}

In this section, we derive the expressions for the Laplacian of the decomposed solutions \eqref{eq2_exact solution_1D} and \eqref{eq2_exact solution_2D}, and verify the corresponding flux continuity condition. These results provide the detailed analytical foundation for the formulation of the transmission elliptic problem presented in Subsection~\ref{LS-ReCoNN Loss Function}.

\noindent By considering three components of \eqref{eq2_exact solution_2D}, the Laplacian of $\mathfrak{u}^p$ becomes
\begin{equation}
\begin{split}
\Delta \mathfrak{u}^p(x)=&\Delta (\mathfrak{w}^p(x) + \mathfrak{v}^p(x) + \sum_{i = 1}^K\sum_{j = 1}^{\bar{K}_i}\kappa^p_{ij} \mathfrak{s}_{ij}^p(x)\eta(|x|))\\
=& \Delta (\mathfrak{w}^p(x)  + \mathfrak{v}^p(x)) +\sum_{i = 1}^K\sum_{j = 1}^{\bar{K}_i}\kappa^p_{ij} \Delta \mathfrak{s}_{ij}^p(x)\eta(|x|)+2\kappa^p_{ij} \nabla \mathfrak{s}_{ij}^p(x)\cdot\nabla \eta(|x|)+\kappa^p_{ij} \mathfrak{s}_{ij}^p(x)\Delta \eta(|x|).
\end{split}
\end{equation}
As $\mathfrak{s}_{ij}^p$ is harmonic, this simplifies to 
\begin{equation}
    \Delta \mathfrak{u}^p(x)=\Delta (\mathfrak{w}^p(x) + \mathfrak{v}^p(x)) + \sum_{i = 1}^K\sum_{j = 1}^{\bar{K}_i}2\kappa^p_{ij} \nabla \mathfrak{s}^p_{ij}(x)\cdot\nabla \eta(|x|)+\kappa^p_{ij} \mathfrak{s}^p_{ij}(x)\Delta \eta(|x|).
\end{equation}
On the other hand, if $|x|<\delta_1$ or $|x|>\delta_2$, then $\eta(|x|)$ is constant (either 0 or 1), so $\nabla(\eta(|x|))=0$ and $\Delta \eta(|x|)=0$. Thus, defining $A_{\delta_1,\delta_2}$ to be the annulus, $B_{\delta_2}\setminus B_{\delta_1}$, we have that 
\begin{equation}\label{eq2_Delta_u}
\begin{split}
\Delta u(x)=&\left\{\begin{array}{c l} 
\Delta (\mathfrak{w}^p(x) + \mathfrak{v}^p(x))+\displaystyle  \sum_{i = 1}^K\sum_{j = 1}^{\bar{K}_i}2\kappa^p_{ij} \nabla \mathfrak{s}^p_{ij}(x)\cdot\nabla \eta(|x|)+\kappa^p_{ij} \mathfrak{s}^p_{ij}(x)\Delta \eta(|x|), & x\in A_{\delta_1,\delta_2},\\
\Delta  (\mathfrak{w}^p(x) + \mathfrak{v}^p(x)),&x\in \Omega\setminus A_{\delta_1,\delta_2}. \end{array}\right.
\end{split}
\end{equation}
 $\mathfrak{s}^p_{ij}$ is defined on the annulus $A_{\delta_1,\delta_2}$, defined away from the singularity. Furthermore, since $\eta$ only depends on $r$, as the gradient of $\mathfrak{s}_{ij}^p$ appears with an inner product with $\eta$, using the representation of $\mathfrak{s}_{ij}^p$ in polar coordinates \eqref{eq2_sij definition}, we have 
 \begin{equation}
     \nabla \mathfrak{s}^p_{ij}(x)\cdot\nabla \eta(|x|) = \lambda^p_{ij} r^{\lambda_{ij}^p-1}\phi^p_{ij}(\theta)\eta'(r), \qquad i=1, \ldots, K, \qquad j=1, \ldots, N_i.
 \end{equation}
In particular, no derivatives of $\phi^p_{ij}$ arise. This means that to approximate this term, and consequently the rest of $p(x)\Delta \mathfrak{u}^p(x)$, we need to only properly approximate $\lambda^p_{ij}$ and $\phi^p_{ij}$, but not the derivative of $\phi^p_{ij}$ (nor its second derivative). Substituting the $\mathcal{S}_{ij}^p$ from \eqref{eq2_mathcal{S}ij_main} into the equation \eqref{eq2_Delta_u}, we recover the transmission elliptic problem formulated in \eqref{eq2_transmission elliptic problem}. Note that $\mathcal{S}^p_{ij}$ is zero for $r < \delta_1$ because $\mathfrak{s}^p_{ij}$ is harmonic in this region, and it is also zero for $r > \delta_2$ because $\eta(r) = 0$.

\noindent Next, we turn to the flux continuity condition. We have that 
\begin{equation}
\begin{split}
0=&\left[p(x)\, \nabla\mathfrak{u}^p(x) \cdot \mathbf{n}\right]\\[2mm]
=& \left[p(x)\nabla\left(\mathfrak{w}^p(x) + \mathfrak{v}^p(x)+ \sum_{i = 1}^K\sum_{j = 1}^{\bar{K}_i}\kappa^p_{ij} \mathfrak{s}_{ij}^p(x)\eta(|x|)\right)\cdot \mathbf{n}\right]\\[2mm]
=&\left[p(x)\nabla (\mathfrak{w}^p(x) + \mathfrak{v}^p(x))\cdot\mathbf{n}+p(x)\sum_{i = 1}^K\sum_{j = 1}^{\bar{K}_i}\kappa^p _{ij} \eta(|x|)\nabla\mathfrak{s}^p_{ij}(x)\cdot\mathbf{n}+p(x)\kappa^p_{ij} \mathfrak{s}^p_{ij}(x)\nabla\eta(|x|)\cdot \mathbf{n}\right].
\end{split}
\end{equation}
In the polygonal geometry, $\mathbf{n}$ is parallel to the angle and orthogonal to the radial direction. In particular, as $\eta$ only depends on the radius, $\nabla\eta(|x|)\cdot \mathbf{n}=0$. In addition, $\left[\nabla \mathfrak{w}^p(x)\cdot\mathbf{n}\right] = 0$. With these properties, we obtain 
$$
0=\left[p(x)\right](\nabla \mathfrak{w}^p(x)\cdot\mathbf{n}) + \left[p(x)\nabla \mathfrak{v}^p(x)\cdot\mathbf{n}\right]+\sum_{i = 1}^K\sum_{j = 1}^{\bar{K}_i}\kappa^p_{ij} \eta(|x|)\left[p(x)\nabla\mathfrak{s}^p_{ij}(x)\cdot\mathbf{n}\right]
$$
Finally, since each function $\mathfrak{s}^p_{ij}$ satisfies the flux continuity condition by construction, it follows that the overall flux continuity condition for $\mathfrak{u}^p$ reduces to equation \eqref{eq2_Continuity-flux}.

%%%%%%%%%%%%%%%%%%%%%%%%%%%%%%%%%%%%%%%%%%%%%
%%%%%%%%%%%%%%%%%%%%%%%%%%%%%%%%%%%%%%%%%%%%%
%%%%%%%%%%%%%%%%%%%%%%%%%%%%%%%%%%%%%%%%%%%%%
%%%%%%%%%%%%%%%%%%%%%%%%%%%%%%%%%%%%%%%%%%%%%
%%%%%%%%%%%%%%% Appendix E %%%%%%%%%%%%%%%%%%
%%%%%%%%%%%%%%%%%%%%%%%%%%%%%%%%%%%%%%%%%%%%%
%%%%%%%%%%%%%%%%%%%%%%%%%%%%%%%%%%%%%%%%%%%%%
%%%%%%%%%%%%%%%%%%%%%%%%%%%%%%%%%%%%%%%%%%%%%
%%%%%%%%%%%%%%%%%%%%%%%%%%%%%%%%%%%%%%%%%%%%%

\section{Proof of the Error Bound}\label{Proof of the Error Bound}
\noindent In this section, we present the proof of Theorem~\eqref{Theorem_main} via a three-step approach. 
We begin by establishing two auxiliary lemmas and a theorem in the nonparametric setting, and subsequently extend these results to the parametric transmission problems through a further theorem. 
For notational convenience, we introduce the following definition
\[
\langle A, B \rangle := \sum_{i=1}^{N_s} \sum_{j = 1}^{K_i} A_{ij} B_{ij},
\]
where $A = \left(\left(a_{ij}\right)_{i=1}^{N_s}\right)_{j=1}^{N_3}$ and $B = \left(\left(b_{ij}\right)_{i=1}^{N_s}\right)_{j=1}^{N_3}$

\begin{lemma}\label{Lemma_2_1}
For a fixed parameter $p \in \mathbb{P}$, let $u^{p, \alpha}_*$ be a candidate solution in the decomposed trial space $U_*^{p, \alpha} := U^\alpha_1 + U^\alpha_2 + \mathbb{U}_3^p$, defined as follows
\begin{equation}
u^{p, \alpha}_* := w^{p, \alpha} + v^{p, \alpha} + \langle \mathbf{c}^{p, \alpha}, \mathfrak{s}^p \rangle.
\end{equation}
Then, there exists a constant $C_1 > 0$,  independent of $p$ and the candidate solution $u^{p, \alpha}_*$, such that the following stability estimate holds
\begin{equation}\label{eq2_upper_bound_H01}
\| u_*^{p, \alpha} - \mathfrak{u}^p \|_{H_0^1(\Omega)}^2 \leq C_1 {\mathcal{J}}_*(U_*^{p, \alpha}),
\end{equation}
where the residual functional ${\mathcal{J}}(U_*^p)$ is defined as follows
\begin{equation}
{\mathcal{J}}_*(U_*^{p, \alpha}) = \| p \left( \Delta (w^{p, \alpha} + v^{p, \alpha}) + \langle \mathbf{c}^{p, \alpha}, \mathcal{S}^p \rangle \right) - l^p \|^2_{L^2(\Omega)} + \Theta \| \left[ p \nabla (w^{p, \alpha} + v^{p, \alpha}) \cdot \mathbf{n} \right] \|_{L^2(\Gamma)}^2.
\end{equation}
where \( \mathcal{S}^p = \left( \left( \mathcal{S}_{ij}^p \right)_{j=1}^{N_i} \right)_{i=1}^K \).
\end{lemma}

%%%%%%%%%%%%%%%%%%%%%%%%%%%%%%%%%%%%%%%%%%%%%
%%%%%%%%%%%%%%%%%%%%%%%%%%%%%%%%%%%%%%%%%%%%%
%%%%%%%%%%%%%%%%%%%%%%%%%%%%%%%%%%%%%%%%%%%%%
\begin{proof}
Consider the residual functional $\mathcal{R}(u^{p, \alpha}_*, v)$ given by
\begin{equation} \label{eq2_proof_residual}
\mathcal{R}(u^{p, \alpha}_*, v) := \int_\Omega p \nabla u^{p, \alpha}_* \cdot \nabla v \, dx - \int_\Omega l^p v \, dx, \quad \forall v \in \mathbb{V}.
\end{equation}
Since the coefficient $p(x)$ is piecewise constant and may be discontinuous across the interfaces, we perform integration by parts on the first term of the residual in \eqref{eq2_proof_residual} separately within each subdomain, i.e., for each $1\leq i\leq I$, we have
\begin{align}\label{eq2_proof_int_by_part}
\int_{\Omega_i} p \nabla u^{p, \alpha}_* \cdot \nabla v \, dx = - \int_{\Omega_i} p \Delta u^{p, \alpha}_* \, v \, dx + \int_{\partial \Omega_i} p \nabla u^{p, \alpha}_* \cdot \mathbf{n}_i \, v \, ds,
\end{align}
where $\mathbf{n}_i$ is the outward normal on $\partial \Omega_i$. Using equations \eqref{eq2_Delta_u} and \eqref{eq2_mathcal{S}ij_main}, we simplify \eqref{eq2_proof_int_by_part} as follows
\begin{align}
\int_{\Omega_i} p \nabla u^{p, \alpha}_* \cdot \nabla v \, dx
=  - \int_{\Omega_i} p \left( \Delta (w^{p, \alpha} + v^{p, \alpha}) + \langle \mathbf{c}^{p, \alpha}, \mathcal{S}^p \rangle \right)\, v \, dx + \int_{\partial \Omega_i} p \nabla  u^{p, \alpha}_* \cdot \mathbf{n}_i \, v \, ds.
\end{align}
Then, summing over all subdomains and utilizing the facts that $v$ vanishes on $\partial \Omega$ and $\left[p\nabla \mathfrak {s}^p_{ij}\cdot \mathbf{n}\right]=0$ holds across the interfaces, we obtain 
\begin{align}\label{eq2_proof_int_by_part_Simplify}
\begin{array}{rl}
    \displaystyle \int_{\Omega} p \nabla u^{p, \alpha}_* \cdot \nabla v \, dx
    =  & \displaystyle - \int_{\Omega} p \left( \Delta (w^{p, \alpha} + v^{p, \alpha}) + \langle \mathbf{c}^{p, \alpha}, \mathcal{S}^p \rangle \right) v \, dx  +\int_{\Gamma}[ p \nabla  (w^{p, \alpha} + v^{p, \alpha}) \cdot \mathbf{n}v]\, dx.\\
\end{array}
\end{align}
Substituting \eqref{eq2_proof_int_by_part_Simplify} into \eqref{eq2_proof_residual}, we obtain
\begin{equation}\label{eq2_proof_residual_1}
\begin{array}{rl}
\mathcal{R}(u^{p, \alpha}_*, v) =\displaystyle-\int_{\Omega} \left(p \left( \Delta (w^{p, \alpha} + v^{p, \alpha}) + \langle \mathbf{c}^{p, \alpha}, \mathcal{S}^p \rangle \right) - l^p \right) v \, dx  +\int_{\Gamma}[ p \nabla  (w^{p, \alpha} + v^{p, \alpha}) \cdot \mathbf{n}] v\, dx.
\end{array}
\end{equation}
By using the Cauchy–Schwarz and Poincaré inequalities, we estimate the first term of \eqref{eq2_proof_residual_1} as follows
\begin{align}\label{eq2_proof_Cauchy-Schwarz_Poincaré}
\left| \int_{\Omega} \left(p \left( \Delta (w^{p, \alpha} + v^{p, \alpha}) + \langle \mathbf{c}^{p, \alpha}, \mathcal{S}^p \rangle \right) - l^p \right) v \, dx \right|
&\leq \left\| p \left( \Delta (w^{p, \alpha} + v^{p, \alpha}) + \langle \mathbf{c}^{p, \alpha}, \mathcal{S}^p \rangle  \right) - l^p \right\|_{L^2(\Omega)} \|v\|_{L^2(\Omega)} \\
&\leq \beta_1 \left\| p \left( \Delta (w^{p, \alpha} + v^{p, \alpha}) + \langle \mathbf{c}^{p, \alpha}, \mathcal{S}^p \rangle  \right) - l^p \right\|_{L^2(\Omega)} \|v\|_{H_0^1(\Omega)},
\end{align}
where \( \beta_1 >0 \) is the Poincaré constant associated with \( \Omega \), and is independent of the parameter \( p \), as well as the functions \( w^{p, \alpha} \), \( v^{p, \alpha} \) and \( \mathbf{c}^{p, \alpha} \). Similarly, using the trace inequality for the jump term of \eqref{eq2_proof_residual_1}, we have
\begin{align}\label{eq2_proof_trace}
\left| \int_{\Gamma} [ p \nabla (w^{p, \alpha} + v^{p, \alpha}) \cdot \mathbf{n} ] v \, ds \right|
\leq & \left\| [ p \nabla (w^{p, \alpha} + v^{p, \alpha}) \cdot \mathbf{n} ] \right\|_{L^2(\Gamma)} \|v\|_{L^2(\Gamma)} \\
\leq &\beta_2 \left\| [ p \nabla (w^{p, \alpha} + v^{p, \alpha}) \cdot \mathbf{n} ] \right\|_{L^2(\Gamma)} \|v\|_{H_0^1(\Omega)},
\end{align} 
where \( \beta_2 > 0 \) is the trace constant associated with the set of interfaces \( \Gamma \), defined explicitly as
\[
\beta_2 := \sup_{v \in H_0^1(\Omega) \setminus \{0\}} \frac{\|v\|_{L^2(\Gamma)}}{\|v\|_{H_0^1(\Omega)}},
\]
and is independent of the parameter \( p \) and the functions \( w^{p, \alpha} \) and \( v^{p, \alpha} \). By multiplying \eqref{eq2_proof_trace} by $ \Theta$ and combining the result with \eqref{eq2_proof_trace}, we obtain 
\begin{equation}\label{eq2_proof_1}
    \frac{|\mathcal{R}(u^{p, \alpha}_*, v)|}{\|v\|_{H_0^1(\Omega)}} \leq 
    \beta {\mathcal{J}}_*(U_*^{p, \alpha}),
\end{equation}
where $\beta := \max\{\beta_1, \Theta \beta_2\}$.  On the other hand, recall that the exact solution $\mathfrak{u}^p$ satisfies
\begin{equation}\label{eq2_proof_residual_exact}
\int_\Omega p \nabla \mathfrak{u}^p \cdot \nabla v \, dx - \int_\Omega l^p v \, dx = 0 \quad \forall v \in H_0^1(\Omega).
\end{equation}
Subtracting the two relations \eqref{eq2_proof_residual} and \eqref{eq2_proof_residual_exact}, we obtain
\begin{equation}\label{eq2_proof_residual_minus_exact}
\mathcal{R}(u^{p, \alpha}_*, v) = \int_\Omega p \nabla (u^{p, \alpha}_* - \mathfrak{u}^p) \cdot \nabla v \, dx.
\end{equation}
This means the residual functional represents the error in a weak sense. By the definition of the dual norm in $H^{-1}(\Omega)$ and using the coercivity inequality for the bilinear form \eqref{eq2_proof_residual_minus_exact}, which holds thanks to the assumed bounds on $p$, we obtain
\begin{equation}
\|u^{p, \alpha}_* - \mathfrak{u}^p\|^2_{H_0^1(\Omega)} \leq \frac{1}{\alpha} \sup_{v \in H_0^1(\Omega)} \frac{|\mathcal{R}(u^{p, \alpha}_*, v)|}{\|v\|_{H_0^1(\Omega)}},
\end{equation}
where $\alpha$ is the uniform coercivity constant. Finally, using the estimate in \eqref{eq2_proof_1}, we conclude that
\begin{equation}
 \|u^{p, \alpha}_* - \mathfrak{u}^p\|_{H_0^1(\Omega)}^2  \leq C_1 {\mathcal{J}}_*(U_*^{p, \alpha}),
\end{equation}
possibly with a different constant $C_1 = \beta/\alpha$.
\end{proof}
%%%%%%%%%%%%%%%%%%%%%%%%%%%%%%%%%%%%%%%%%%%%%
%%%%%%%%%%%%%%%%%%%%%%%%%%%%%%%%%%%%%%%%%%%%%
%%%%%%%%%%%%%%%%%%%%%%%%%%%%%%%%%%%%%%%%%%%%%
\noindent
Lemma ~\ref{Lemma_2_1} indicates that minimizing \( {\mathcal{J}}(U_*^{p, \alpha})\) with respect to \( w^{p, \alpha} \), \( v^{p, \alpha} \), and \( \mathbf{c}^{p, \alpha} \), and using exact singular functions, leads to a reliable approximation of the solution. Next, our objective is to refine the estimate \eqref{eq2_upper_bound_H01} by incorporating the approximation errors of eigenvalues and eigenvectors in singular functions.

{\begin{lemma}\label{Lemma_2_2}
Let \( p \in \mathbb{P} \) be fixed, and let   approximate singular functions of the form \eqref{eq2_sjp}. 
Then there exists a constant \( C_2 > 0 \), independent of \( p \), $u_*^{p, \alpha}$, and $u^{p, \alpha}$
such that the following estimate holds:
\begin{equation}\label{eq2_Lemma2_2_equation}
{\mathcal{J}}(U_*^{p, \alpha})
\leq C_2 \left(
\mathcal{J}(U^{p, \alpha})
+  \| \mathbf{c}^{p, \alpha}\|_{\ell ^2}^2
\left(
\| \vartheta^p - \mu^p \|^2_{L^2(0, 2\pi)}
+ \| \lambda^p - \Lambda^p \|^2_{\ell^2}
\right)
\right).
\end{equation}
\end{lemma}
}

%%%%%%%%%%%%%%%%%%%%%%%%%%%%%%%%%%%%%%%%%%%%%
%%%%%%%%%%%%%%%%%%%%%%%%%%%%%%%%%%%%%%%%%%%%%
%%%%%%%%%%%%%%%%%%%%%%%%%%%%%%%%%%%%%%%%%%%%%
\begin{proof}
A straightforward use of the triangle inequality gives 
\begin{equation}\label{eq2_proof_2_1}
\begin{array}{rl}
\|p\left(\Delta  (w^{p, \alpha} + v^{p, \alpha})+ \langle \mathbf{c}^{p, \alpha}, \mathcal{S}^p\rangle\right)-l^p\|_{L^2}\leq &\hspace{-2mm}  \|p\left(\Delta  (w^{p, \alpha} + v^{p, \alpha}) +\langle \mathbf{c}^{p, \alpha}, \mathbf{S}^p\rangle\right)-l^p\|_{L^2}+\\
& \hspace{-2mm}\|\mathbf{c}^{p, \alpha}\|_{L^\infty}\,\|p\left(\mathcal{S}^p-\mathbf{S}^p\right)\|_{L^2},
\end{array}
\end{equation}
where \(\mathbf{S}^p = \left( \left( \mathbf{S}^p_{ij} \right)_{j=1}^{N_i} \right)_{i=1}^K\). 
According to \eqref{eq2_proof_2_1}, we turn to estimating the $L^2$ distance between $\mathcal{S}^p$ and $\mathbf{S}^p$. 
By denoting 
$$
G(r, \lambda) := 
2\lambda \eta'(r) r^{\lambda-1}
+ \left( \eta''(r) + \frac{1}{r} \eta'(r) \right) r^{\lambda}, 
$$
we may write $\mathcal{S}^p_{ij}(r,\theta)=\vartheta_{ij}^p(\theta)G(r,\Lambda_{ij}^p)$ and $\mathbf{S}^p_{ij}(r,\theta)=\mu_{ij}^p(\theta)G(r,\lambda_{ij}^p)$. Then,
\begin{equation}
\begin{split}
\|p\left(\mathcal{S}^p-\mathbf{S}^p\right)\|_{L^2}^2 = & \sum_{i = 1}^K\sum_{j=1}^{N_i}\int_{\delta_1}^{\delta_2}\int_0^{2\pi} p^2(\theta)\left(\vartheta^p_{ij}(\theta)G(r,\Lambda^p_{ij})-\mu^p_{ij}(\theta)G(r,\lambda^p_{ij})\right)^2\,d\theta\,dr\\
\lesssim  &\sum_{i = 1}^{N_s}\sum_{j=1}^{N_3}\int_{\delta_1}^{\delta_2}\int_0^{2\pi} p^2(\theta)\left(\vartheta^p_{ij}(\theta)-\mu^p_{ij}(\theta)\right)^2G(r,\lambda^p_{ij})^2+\\
&\hspace{2.7cm} p^2(\theta)\left(G(r,\Lambda^p_{ij})-G(r,\lambda^p_{ij})\right)^2\vartheta^p_{ij}(\theta)^2\,d\theta\,dr.\\
\end{split}
\end{equation}
We observe that, for $r \in [\delta_1,\delta_2]$ and $\lambda \in [0,1]$, 
the function $G(r,\lambda)$ is bounded and Lipschitz in $\lambda$ and 
uniformly in $r$. Furthermore, without loss of generality, we can take $\vartheta^p$ and $\mu^p$ to have $L^2$-norm equal to one, leading to the upper bound of 
$$
\|p\left(\mathcal{S}^p-\mathbf{S}^p\right)\|_{L^2}^2 \leq \tilde{c}\left(
\| \vartheta^p - \mu^p \|^2_{L^2(0, 2\pi)}
+ \| \lambda^p - \Lambda^p \|^2_{\ell^2}
\right),
$$
where $\tilde{c}>0$ is a constant. Thus, we obtain \eqref{eq2_Lemma2_2_equation} with $C_2 = \max \{\tilde{c},  1 \}$.
\end{proof}

\begin{theorem}\label{Theorem_single}
Let \( p \in \mathbb{P} \) be a fixed parameter. There exists a constant $C > 0$, independent of the LS-ReCoNN solution $u^{p, \alpha}$, such that
\begin{equation}\label{eq:single_param_estimate}
\left\| u^{p, \alpha} - \mathfrak{u}^p \right\|_{H_0^1(\Omega)}^2 
\leq C \left( \mathcal{J}(U^{p, \alpha})
+ \| \mathbf{c}^{p, \alpha}\|_{\ell ^2}^2 \left( \| \vartheta ^p - \mu^p \|_{L^2(0, 2\pi)}^2 + \| \lambda^p - \Lambda^p \|_{\ell ^2}^2 \right) \right) ,
\end{equation}
\end{theorem}

\begin{proof}
This is an immediate consequence of Lemmas~\ref{Lemma_2_1} and~\ref{Lemma_2_2}, applied for a fixed parameter \(p\).  
The independence of the constant \(C\) from \(p\) follows from the uniform ellipticity assumptions.
\end{proof}

{\begin{theorem}\label{Theorem_parametric}
Let \( p \in \mathbb{P} \) be a fixed parameter. There exists a constant $C > 0$, independent of the LS-ReCoNN solution $u^{p, \alpha}$, such that
\begin{equation}\label{eq:parametric_estimate}
\int_{\mathbb{P}} \left\| u^{p, \alpha} - \mathfrak{u}^p \right\|_{H_0^1(\Omega)}^2 \, d\mu(p)
\leq C \left( \mathcal{J}_\mu(U^{p, \alpha})
+ \int_{\mathbb{P}} \| \mathbf{c}^{p, \alpha}\|_{\ell ^2}^2 \left( \| \vartheta ^p - \mu^p \|_{L^2(0, 2\pi)}^2 + \| \lambda^p - \Lambda^p \|_{\ell ^2}^2 \right) \, d\mu(p)\right) ,
\end{equation}
\end{theorem}
}

\begin{proof}
The claim follows directly by integrating the estimate 
\eqref{eq:single_param_estimate} from Theorem~\ref{Theorem_single} over 
the parameter set \(\mathbb{P}\), noting that the constant \( C \) is uniform in \( p \).
\end{proof}

%%%%%%%%%%%%%%%%%%%%%%%%%%%%%%%%%%%%%%%%%%%%%
%%%%%%%%%%%%%%%%%%%%%%%%%%%%%%%%%%%%%%%%%%%%%
%%%%%%%%%%%%%%%%%%%%%%%%%%%%%%%%%%%%%%%%%%%%%
%%%%%%%%%%%%%%%%%%%%%%%%%%%%%%%%%%%%%%%%%%%%%
%%%%%%%%%%%%%%% Appendix F %%%%%%%%%%%%%%%%%%
%%%%%%%%%%%%%%%%%%%%%%%%%%%%%%%%%%%%%%%%%%%%%
%%%%%%%%%%%%%%%%%%%%%%%%%%%%%%%%%%%%%%%%%%%%%
%%%%%%%%%%%%%%%%%%%%%%%%%%%%%%%%%%%%%%%%%%%%%
%%%%%%%%%%%%%%%%%%%%%%%%%%%%%%%%%%%%%%%%%%%%%

\section{Details of the Loss Discretizations }\label{Loss Discretization Details}

\subsection{Parameter Sampling}\label{Parameter Sampling Details}
We generate the parameter samples $p_i \in [p_{\min}, p_{\max}]$ using the following transformation
\begin{equation}
p_i = \frac{p_{\min} + p_{\max}}{2} + \cos(z_i) \frac{p_{\max} - p_{\min}}{2},
\end{equation}
where $z_i$ is drawn from a uniform distribution over the interval $[0, \pi]$. This procedure induces a non-uniform, cosine-weighted distribution across the parameter space. Consequently, the method ensures denser sampling near the boundaries $p_{\min}$ and $p_{\max}$.

%%%%%%%%%%%%%%%%%%%%%%%%%%%%%%%%%%%%%%%%%%%%%
%%%%%%%%%%%%%%%%%%%%%%%%%%%%%%%%%%%%%%%%%%%%%
%%%%%%%%%%%%%%%%%%%%%%%%%%%%%%%%%%%%%%%%%%%%%
%%%%%%%%%%%%%%%%%%%%%%%%%%%%%%%%%%%%%%%%%%%%%
\subsection{Matrix Definitions}\label{Matrix Definitions}
Let $\mathbf{x}^{(1)} := \{ x^{(1)}_j \}_{j=1}^{J_1} \subset \Omega$ and $\mathbf{x}^{(2)}:= \{ x^{(2)}_j \}_{j=1}^{J_2} \subset \Gamma$ denote the random sets of quadrature points for the domain and interfaces, respectively. The corresponding $L^2$-norms are then approximated using Monte Carlo integration as follows
\begin{equation}\label{eq2_L2 discretization_PDE}
    \begin{array}{rl}
    \left\Vert p \left( \Delta \left(\mathbf{a}^{p, \alpha}\cdot\mathbf{w}^{\alpha} + \mathbf{b}^{p, \alpha}\cdot\mathbf{v}^{\alpha}\right) + {\mathbf{c}^{p, \alpha}}\cdot \mathbf{S}^{p} \right)- l^p \right\Vert_{L^2(\Omega)}^2 \approx & \displaystyle \frac{\text{Vol}(\Omega)}{J_1} \sum_{j = 1}^{J_1} \Big( p(x_j^{(1)}) \Big( \Delta \Big(\mathbf{a}^{p, \alpha}\cdot\mathbf{w}^{\alpha}(x_j^{(1)}) +\\ & \mathbf{b}^{p, \alpha}\cdot\mathbf{v}^{\alpha}(x_j^{(1)})\Big) + 
    {\mathbf{b}^{p, \alpha}}\cdot \mathbf{S}^{p}(x_j^{(1)}) \Big) - l^p(x_j^{(1)}) \Big), \\
    
    \left\Vert [p \nabla (\mathbf{a}^{p, \alpha}\cdot \mathbf{w}^{p} + \mathbf{b}^{p, \alpha}\cdot \mathbf{v}^{p}) \cdot \mathbf{n}] \right\Vert^2_{L^2(\Gamma)} \approx &\displaystyle \frac{\text{Length}(\Gamma)}{J_2} \Big[\sum_{j = 1}^{J_2} p(x_j^{(2)}) \nabla (\mathbf{a}^{p, \alpha}\cdot \mathbf{w}^{p}(x_j^{(2)}) + \\
    
    &\mathbf{b}^{p, \alpha}\cdot \mathbf{v}^{p}(x_j^{(2)})) \cdot \mathbf{n} \Big]. 
    \end{array}
\end{equation}
The approximations in \eqref{eq2_L2 discretization_PDE} are reformulated into the following matrix form
\begin{equation}
    \begin{array}{rl}\left\Vert p \left( \Delta \left(\mathbf{a}^{p, \alpha}\cdot\mathbf{w}^{\alpha} + \mathbf{b}^{p, \alpha}\cdot\mathbf{v}^{\alpha}\right) + {\mathbf{c}^{p, \alpha}}\cdot \mathbf{S}^{p} \right) - l^p \right\Vert_{L^2(\Omega)}^2 &\approx \left\Vert \mathbf{B}_1^{p, \alpha}(\mathbf{a}^{p, \alpha};\mathbf{b}^{p, \alpha};\mathbf{c}^{p, \alpha}) - \mathbf{l}^p_1 \right\Vert_2^2, \\
\label{eq2_L2 discretization_tc_Matrix}\left\Vert [p \nabla (\mathbf{a}^{p, \alpha}\cdot \mathbf{w}^{p} + \mathbf{b}^{p, \alpha}\cdot\mathbf{v}^{\alpha}) \cdot \mathbf{n}] \right\Vert^2_{L^2(\Gamma)} &\approx \left\Vert \mathbf{B}_2^{p, \alpha} (\mathbf{a}^{p, \alpha};\mathbf{b}^{p, \alpha}) - \mathbf{l}^p_2\right\Vert_2^2. 
\end{array}
\end{equation}
Herein, for each \( i = 1, 2 \), the matrix \( \mathbf{B}_i^\alpha \in \mathbb{R}^{J_i \times N} \) and the vector \( \mathbf{l}_i \in \mathbb{R}^{J_i} \) are defined as follows
\begin{equation}
    \begin{array}{rlrl}
(\mathbf{B}_1^{p, \alpha})_{j_1n} &= \displaystyle \sqrt{\frac{\text{Vol}(\Omega)}{J_1}} p(x_{j_1}^{(1)}) \left( \Delta \left(\mathbf{w}^{\alpha} + \mathbf{v}^{\alpha}\right)(x_{j_1}^{(1)}), \, \mathbf{S}^{p}(x_{j_1}^{(1)}) \right), &
(\mathbf{l}_1^p)_{j_1} &= \displaystyle \sqrt{\frac{\text{Vol}(\Omega)}{J_1}}\, l^p(x_{j_1}^{(1)}), \\[3mm]
(\mathbf{B}_2^{p, \alpha})_{j_2n} &= \displaystyle \sqrt{\frac{\text{Length}(\Gamma)}{J_2}} \left[ p(x_{j_2}^{(2)}) \nabla (\mathbf{w}^{\alpha} + \mathbf{v}^{\alpha})(x_{j_2}^{(2)}) \cdot \mathbf{n} \right],  &
(\mathbf{l}^p_2)_{j_2} &= \mathbf{0}_{J_2}.\label{eq2_jump_discretization}
\end{array}
\end{equation}
where $j_i = 1, \ldots, J_i$ for $i = 1, 2$ and $\mathbf{0}_{J_2}$ is a zero vector of dimension $J_2$.  Finally, we define
\begin{equation}
\mathbf{B}^{p,\alpha} =
\begin{bmatrix}
\mathbf{B}^{p,\alpha}_1  \\[3mm]
[\sqrt{\Theta}\,\mathbf{B}^{p,\alpha}_2; \mathbf{0}_{J_2\times N_3}]
\end{bmatrix},
\qquad
\mathbf{L}^p =
\begin{bmatrix}
 \mathbf{L}^p_1 \\[3mm]
\mathbf{L}^p_2
\end{bmatrix},
\end{equation}
where $\mathbf{0}_{J_2\times N_3}$ is a zero matrix of dimension $J_2\times N_3$.

\bibliographystyle{abbrv}
\bibliography{references}
\end{document}